\documentclass[a4paper,oneside,parskip=half,fontsize=11pt]{scrartcl}
\usepackage{mathrsfs}
\usepackage{epsfig, graphicx,color}
\usepackage{latexsym,amsfonts,amsbsy,amssymb}
\usepackage{amsmath,amsthm}
\usepackage{url}
\usepackage{hyperref}
\usepackage{float}
\usepackage{ulem}
\usepackage{caption}
\usepackage{subcaption}
\usepackage{algorithm, algorithmic}
\renewcommand{\algorithmiccomment}[1]{\bgroup\hfill//~#1\egroup}
\makeatletter 

\@addtoreset{equation}{section}
\makeatother 
\allowdisplaybreaks 

\newlength{\fixboxwidth}
\newtheorem{theorem}{Theorem}[section]
\newtheorem{Property}[theorem]{Property}
\newtheorem{lemma}{Lemma}[section]

\newtheorem{remark}{Remark}[section]

\def\V{\mathfrak{V}}
\def\ve{\varepsilon}

\def\ve{\varepsilon}

\def\R{\mathbb{R}}

\def\MG{\operatorname{\texttt{MG}}}

\def\th{\textrm{th}}

\def\cW{\mathcal{W}}
\def\cV{\mathcal{V}}

\def\cB{\mathcal{B}}

\def\I{\mathcal{I}}
\def\J{\mathcal{J}}
\def\L{\mathcal{L}}

\def\W{\mathfrak{W}}

\def\H{\mathcal{H}}
\def\T{\mathcal{T}}
\def\O{\mathcal{O}}

\def\<{\big\langle}
\def\>{\big\rangle}

\def\Cond{\operatorname{Cond}}

\def\diiv{\operatorname{div}}
\def\dim{{\operatorname{dim}}}

\def\Img{\operatorname{Im}}
\def\Ker{\operatorname{Ker}}

\def\Span{\operatorname{span}}

\def\ds{\displaystyle}

\newcommand{\br}[1]{{(#1)}}

\newcommand{\Abr}[1]{A^{(#1)}}
\newcommand{\Bbr}[1]{B^{(#1)}}

\newcommand{\Vbr}[1]{\V^{(#1)}}
\newcommand{\Lbr}[1]{\Lambda^{(#1)}}

\newtheorem{Theorem}{Theorem}[section]
\newtheorem{Proposition}[Theorem]{Proposition}

\newtheorem{Corollary}[Theorem]{Corollary}
\newtheorem{Remark}[Theorem]{Remark}
\newtheorem{Example}[Theorem]{Example}

\allowdisplaybreaks 

\definecolor{redcol}{RGB}{255, 0, 0}

\definecolor{bluecol}{RGB}{0, 0, 180}

\newcommand{\revise}[1]{{\color{black}#1}}

\title{Fast eigenpairs computation\\ with operator adapted wavelets and hierarchical subspace correction}

\date{\today}

\author{Hehu Xie\thanks{{\texttt{hhxie@lsec.cc.ac.cn}}. LSEC, NCMIS, Institute of Computational Mathematics, Academy of Mathematics and Systems
Science, Chinese Academy of Sciences, Beijing 100190, China, and University of Chinese Academy of Sciences, Beijing 100049, China.},\ \
Lei Zhang\footnote{{\texttt{lzhang2012@sjtu.edu.cn}}.
School of Mathematical Sciences, Institute of Natural Sciences, and Ministry of Education Key
Laboratory of Scientific and Engineering Computing (MOE-LSC), Shanghai Jiao Tong University,
800 Dongchuan Road, Shanghai 200240, China.
},\ \
Houman Owhadi\footnote{{\texttt{owhadi@caltech.edu}}. California Institute of Technology, Computing \& Mathematical Sciences,
MC 9-94 Pasadena, CA 91125.
}
}

\begin{document}

\maketitle
\begin{abstract}
We present a method for the fast computation of the eigenpairs of a bijective positive symmetric linear operator $\mathcal{L}$. The method is based on a combination of operator adapted wavelets (gamblets) with  hierarchical subspace correction. First, gamblets provide a raw but fast approximation of the eigensubspaces of $\mathcal{L}$ by block-diagonalizing $\mathcal{L}$ into  sparse and well-conditioned blocks. Next, the
hierarchical subspace correction method,  computes the eigenpairs associated with the Galerkin restriction of $\mathcal{L}$ to a coarse (low dimensional) gamblet subspace, and then,  corrects those eigenpairs by solving a hierarchy of linear problems in the finer gamblet subspaces (from coarse to fine, using multigrid iteration). The proposed algorithm is robust to the presence of multiple (a continuum of) scales and is shown to be of near-linear complexity when $\L$ is an (arbitrary local, e.g.~differential) operator mapping $\mathcal{H}^s_0(\Omega)$ to $\mathcal{H}^{-s}(\Omega)$ (e.g.~an elliptic PDE with rough coefficients).


\end{abstract}

\section{Introduction}
\label{sec:introduction}

Solving large scale eigenvalue problems is one of the most fundamental and challenging tasks in modern science
and engineering. Although  high-dimensional eigenvalue problems are ubiquitous in physical sciences, data and imaging sciences,
and machine learning, the class of eigensolvers is not as diverse as that of linear solvers (which comprises
many efficient algorithms such as geometric and algebraic multigrid  \cite{Brandt:1973, Hackbusch:1978}, approximate Gaussian elimination \cite{Kyng:2016}, etc.). In particular, eigenvalue problems may involve operators with nonseparable multiple scales, and the nonlinear interplay between those coupled scales and the eigenvalue problem poses significant challenges for numerical analysis and scientific computing \cite{Arnold:2016,CaoCui,ZhangCaoWang,Malqvist2014a}.

Krylov subspace type methods remain the most reliable and efficient tools for large scale eigenproblems, and alternative approaches such as optimization based methods and nonlinear solver based methods have been pursued in the recent years. For example, the Implicitly Restarted Lanczos/Arnoldi Method (IRLM/IRAM) \cite{Sorensen:1997}, the Preconditioned INVerse ITeration (PINVIT) method \cite{Dyakonov:1980,Bramble:1996,Knyazev:1998}, the Locally Optimal Block Preconditioned Conjugate Gradient (LOBPCG) method \cite{Kynazev:2001,Knyazev:2003}, and the Jacobi-Davidson-type techniques \cite{Bai:2000} have been developed. For those state-of-the-art eigensolvers, the efficient application of preconditioning \cite{Knyazev:1998} is often crucial for the faster convergence and the reduction of computation cost, especially for multiscale eigenproblems.

Recently, two-level \cite{Xu:2001, Malqvist2014a} and multilevel \cite{LinXie_2011,LinXie_2012,LinXie_MultiLevel,Xie_JCP,Xie_IMA} correction methods  have been proposed to reduce the complexity of solving eigenpairs associated with low eigenvalues by first solving a coarse mesh/scale approximation, which can then be corrected by solving linear systems (corresponding to linearized eigenvalue problems)  on a hierarchy of finer meshes/scales. Although the multilevel correction approach has been extended to multigrid methods for linear and nonlinear eigenvalue problems \cite{ChenXieXu,LinXie_2011,LinXie_2012,LinXie_MultiLevel,JiaXieXieXu,Xie_JCP,Xie_IMA}, the regularity estimates required for linear complexity do not hold for PDEs with rough coefficients and a naive application of the correction approach to multiscale eigenvalue problems may converge very slowly.
 For two-level methods \cite{Xu:2001} this lack of robustness can be alleviated by numerical homogenization techniques \cite{Malqvist2014a}, e.g., the so-called Localized Orthogonal Decomposition (LOD) method. For multilevel methods,  gamblets \cite{OwhadiMultigrid:2017,  OwhadiScovel:2017,Owhadi2017a, SchaeferSullivanOwhadi17, OwhScobook2018} (operator-adapted wavelets satisfying three desirable properties: scale orthogonality, well-conditioned multi-resolution decomposition, and localization) provide a natural multiresolution decomposition ensuring robustness for multiscale eigenproblems.  \revise{ As described in \cite[Sec.~5.1.3]{OwhScobook2018}, these three properties are analogous to those
required of Wannier functions \cite{kohn1959, wannier1962dynamics}, which can be characterized as
 linear combinations $\chi_i=\sum_{j} c_{i,j} v_j$  of eigenfunctions $v_{j}$ associated  with eigenvalues
$\lambda_{j}$
such that the size of
$c_{i,j}$ is large
 for $\lambda_j$ close to $\lambda_i$ and small otherwise, and such
 that the resulting  linear combinations  $\chi_i$ are concentrated in space. }



The aim of this paper is therefore to design a fast multilevel numerical method for multiscale eigenvalue problems (e.g. for PDEs that may have rough and highly oscillatory coefficients) associated with a bijective positive symmetric linear operator $\L$, by integrating the multilevel correction approach with the gamblet multiresolution decomposition. In this merger, the gamblet decomposition supplies a hierarchy of coarse (sub)spaces for the multilevel correction method. The overall computational cost is that of solving a sequence of linear problems over this hierarchy (using a gamblet based multigrid approach  \cite{OwhadiMultigrid:2017}).  Recently, Hou et. al. \cite{Hou:2018}  proposed to compute the leftmost eigenpairs of a sparse symmetric positive matrix
by combining the implicitly restarted Lanczos method with a gamblet-like multiresolution decomposition where local eigenfunctions  are used as measurement functions. This paper shows that the gamblet multilevel decomposition (1) enhances the convergence rate of  eigenvalue solvers by enabling (through a gamblet based multigrid method) the fast and robust convergence of  inner iterations (linear solves) in the multilevel correction method, and (2) provides efficient preconditioners for state-of-the-art eigensolvers such as the LOBPCG method.


\paragraph{Outline} This paper is organized as follows: We summarize the gamblet decomposition, its properties, and the gamblet based multigrid method in  Section \S~\ref{sec:gamblet} (see \cite{OwhadiMultigrid:2017,  OwhadiScovel:2017,Owhadi2017a, SchaeferSullivanOwhadi17, OwhScobook2018} for the detailed construction). We present the gamblet based multilevel method for multiscale eigenvalue problems and its rigorous analysis in Section \S~\ref{sec:eigenproblem}. Our theoretical results are numerically illustrated
in Section \S~\ref{sec:numerics} where the proposed method is compared with state-of-the-art eigensolvers (such as LOBPCG). 

\paragraph{Notation} The symbol $C$ denotes generic positive constant that may change from one line
of an estimate to the next. $C$ will be independent from the eigenvalues (otherwise a subscript $\lambda$ will be added), and the dependencies of $C$ will normally be clear from the context or stated explicitly.

\section{Gamblet Decomposition and Gamblet based Multigrid Method}
\label{sec:gamblet}
\def\cB{V}

Although multigrid methods \cite{Brandt:1973, Hackbusch:1978} have been highly  successful in solving elliptic PDEs, their convergence rates can be severely affected by the lack of regularity of the PDE coefficients \cite{wan2000}.  Although classical wavelet based methods \cite{Beylkin:1995,  DorobantuEngquist1998} enable a multi-resolution decomposition of the solution space,  their performance can also be affected by their lack of adaptation to the coefficients of the PDE. The introduction of gamblets in \cite{OwhadiMultigrid:2017} addressed the  problem of designing multigrid/multiresolution methods that are provably robust with respect to rough ($L^\infty$) PDE coefficients.

Gamblets are derived from a game theoretic approach to numerical analysis  \cite{OwhadiMultigrid:2017, OwhadiScovel:2017}.  They (1) are elementary solutions of hierarchical information games associated with the process of computing with partial information and limited resources, (2) have a natural Bayesian interpretation under the mixed strategy emerging from the game theoretic formulation, (3) induce a  multi-resolution decomposition of the solution space that is adapted to the numerical discretization of the underlying PDE. The (fast) gamblet transform has $\mathcal{O}(N\log^{2d+1}N)$ complexity for the first solve and $\mathcal{O}(N\log^{d+1}N)$ for subsequent solves  to achieve grid-size accuracy in $H^1$-norm for elliptic problems \cite{OwhScobook2018}.


\subsection{The abstract setting}\label{secset}
We introduce the formulation of gamblets with an abstract setting since its application is not limited to  scalar elliptic problems such as examples \ref{example1} and \ref{example2}. Let $(\cB,\|\cdot\|), (\cB^*,\|\cdot\|_*)$ and $(\cB_0,\|\cdot\|_0)$ be Hilbert spaces such that  $\cB\subset \cB_0 \subset \cB^*$ and such that  the natural embedding $i:\cB_0 \rightarrow \cB^*$ is compact and dense. Let $(V^*,\|\cdot\|_*)$ be the dual of $(V,\|\cdot\|)$ using  the  dual pairing obtained from the  Gelfand triple.

Let the operator $\L$ be a symmetric positive linear bijection mapping $\cB$ to $\cB^*$.
Write $[\cdot,\cdot]$ for the duality pairing between $\cB^*$ and $\cB$ (derived from the Riesz duality between $\cB_0$ and itself) such that
\begin{equation}
\|u\|^2=[\L u, u]\text{ for } u\in \cB\,.
\end{equation}
The corresponding inner product on $\cB$ is defined by
\begin{equation}
	\<u, v\> : = [\L u, v]\quad\text{for }u, v\in \cB,
\end{equation}
and $\|\cdot\|_*$ is the corresponding dual-norm on $\cB^*$, i.e.
\begin{equation}
\|\phi\|_*=\sup_{v\in \cB, v\neq 0} \frac{[\phi, v]}{\|v\|}\text{ for } \phi\in \cB^*\,.
\end{equation}
Given $g\in \cB^*$, we will  consider the solution  $u$ of the variational problem
\begin{equation}
\<u,v\>=[g, v], \quad\text{ for  }v\in\cB\,.
\label{eq:problem}
\end{equation}

\begin{Example}\label{example1}
Let $\Omega$  be a bounded open subset of $\R^d$ (of arbitrary dimension $d\in \mathbb{N}^*$) with uniformly Lipschitz  boundary.  Given $s\in \mathbb{N}$, let
\begin{equation}\label{eqkjdlhejdjhii}
\L\,:\, \H^s_0(\Omega)\rightarrow \H^{-s}(\Omega)
\end{equation}%
be a continuous linear bijection  between $\H^s_0(\Omega)$ and $\H^{-s}(\Omega)$, where $\H^s_0(\Omega)$ is the Sobolev space of order $s$ with zero trace, and $\H^{-s}(\Omega)$ is the topological dual of $\H^s_0(\Omega)$ \cite{Adams:2003}. Assume $\L$ to be symmetric, positive and local, i.e.
$ [\L u,v]=[u, \L v]$ and $[\L u,u]\geq 0$ for $ u, v\in \H^s_0(\Omega)$ and
 $[\L u,v]=0$ if $u,v$ have disjoint supports in $\Omega$.
In this example $\cB, \cB^*$ and $\cB_0$ are $\H^s_0(\Omega)$, $\H^{-s}(\Omega)$ and $L^2(\Omega)$ endowed with the norms
$\|u\|^2=\int_{\Omega} u \L u$, $\|\phi\|_*^2=\int_{\Omega} \phi \L^{-1}\phi$ and $\|u\|_0=\|u\|_{L^2(\Omega)}$.
\end{Example}

\begin{Example}\label{example2}
Consider Example \ref{example1} with $s=1$,  $\L=-\diiv \big(a(x)  \nabla \cdot\big)$  and $a(x)$ is a symmetric, uniformly elliptic $d\times d$ matrix
with entries in $L^\infty(\Omega)$  such that for all $x\in \Omega$ and $\ell\in \R^d$,
\begin{equation}
\lambda_{\min}(a) |\ell|^2 \leq \ell^T a(x) \ell \leq \lambda_{\max}(a)|\ell|^2.
\end{equation}
Note that
\begin{equation}
\|v\|^2=\int_{\Omega}(\nabla v)^T a \nabla v \ \ \text{ for } v\in \H^1_0(\Omega)\,,
\end{equation}
and
the solution of \eqref{eq:problem} is the solution of the  PDE
\begin{equation}\label{eqn:scalar}
\begin{cases}
-\diiv \big(a(x)  \nabla u(x)\big)=g(x) \quad  x \in \Omega, \\
u=0 \quad \text{on}\quad \partial \Omega\,.
\end{cases}
\end{equation}
\end{Example}


\subsection{Gamblets}

Here we give a brief reminder of the construction of gamblets. See Example \ref{example3} for a concrete example for scalar elliptic equation and Section \S~\ref{sec:spe10} for the numerical implementation, and also \cite{OwhadiMultigrid:2017,  OwhadiScovel:2017,Owhadi2017a, SchaeferSullivanOwhadi17, OwhScobook2018} for more details.

\paragraph{Measurement functions.}
Let $\I^{(1)},\ldots,\I^{(q)}$ be a hierarchy of labels and let  $\phi_i^{(k)}$  be a hierarchy of nested elements of $\cB^*$ such that
\begin{equation}\label{eqejhdgdhdld}
\phi_i^{(k)}=\sum_{j\in \I^{(k+1)}}\pi^{(k,k+1)}_{i,j}\phi_j^{(k+1)}\ \ \text{ for }k\in \{1,\ldots,q-1\}\text{ and }i\in \I^{(k)}\,,
\end{equation}
for some rank $|\I^{(k)}|$, $\I^{(k)}\times \I^{(k+1)}$  matrices $\pi^{(k,k+1)}$ and such that the $(\phi_i^{(k)})_{i\in \I^{(k)}}$ are linearly independent and
 $\pi^{(k,k+1)}\pi^{(k+1,k)}=I_{\I^{(k)}}$ for $k\in \{1,\ldots,q-1\}$ (writing  $I_{\J}$ for the $\J\times \J$ identity matrix and $\pi^{(k+1,k)}$ for $(\pi^{(k,k+1)})^T$). Although not required in the general theory of gamblets \cite{OwhScobook2018} in this paper we assume that the $\phi_i^{(k)}$ are elements of $\cB_0$ and have uniformly well conditioned mass matrices in the sense that
 $C^{-1}|x|^2 \leq \|\sum_i x_i \phi_i^{(k)}\|_0^2\leq C |x|^2$ (for all $x$ and $k$).

\paragraph{Operator adapted pre-wavelets.}
 For $k\in \{1,\ldots,q\}$, let $\Theta^{(k)}$ be the symmetric positive definite matrix with entries $\Theta^{(k)}_{i,j}:=[\phi_i^{(k)},\L^{-1}\phi_j^{(k)}]$ and (writing $\Theta^{(k),-1}$ for the inverse of $\Theta^{(k)}$) let
 \begin{equation}
 \psi_i^{(k)}=\sum_{j\in \I^{(k)}} \Theta^{(k),-1}_{i,j} \phi_j^{(k)}\ \ \text{ for } i\in \I^{(k)}\,.
 \end{equation}
The elements $\psi_i^{(k)}$ form a bi-orthogonal system with respect to the elements $\phi_i^{(k)}$, i.e.
$[\phi_j^{(k)},\psi_i^{(k)}]=\delta_{i,j}$ and
\begin{equation}\label{eqjkekkdgddh}
u^{(k)}:=\sum_{i\in \I^{(k)}}[\phi_i^{(k)},u]\psi_i^{(k)}\,,
\end{equation}
is the $\<\cdot,\cdot\>$ orthogonal projection of $u\in \cB$ onto
\begin{equation}
\V^{(k)}:=\Span\{\psi_i^{(k)}\mid i\in \I^{(k)}\}\,.
\end{equation}
Furthermore $A^{(k)}:=\Theta^{(k),-1}$ can be identified as the stiffness matrix of the $\psi_i^{(k)}$, i.e.
\begin{equation}
A^{(k)}_{i,j}=\<\psi_i^{(k)},\psi_j^{(k)}\>\ \ \text{ for }i,j\in \I^{(k)}\,.
\end{equation}
The  $\psi_i^{(k)}$ are nested pre-wavelets in the sense that $\V^{(k)} \subset \V^{(k+1)}$ and
\begin{equation}
\psi_i^{(k)}=\sum_{j\in \I^{(k+1)}}R^{(k,k+1)}_{i,j} \psi_j^{(k+1)}\,,
\end{equation}
where $R^{(k,k+1)}=A^{(k)}\pi^{(k,k+1)}\Theta^{(k+1)}$ acts as an interpolation matrix.

\paragraph{Gamblets (operator adapted wavelets).}
Let $(\J^{(k)})_{2\leq k\leq q}$ be a hierarchy of labels such that (writing $|\J^{(k)}|$ for the cardinal of $\J^{(k)}$)
$|\J^{(k)}|=|\I^{(k)}|-|\I^{(k-1)}|$.
For $k\in \{2,\ldots,q\}$, let $W^{(k)}$ be a $\J^{(k)}\times \I^{(k)}$ matrix such that (writing $W^{(k),T}$ for the transpose of $W^{(k)}$)
\begin{equation}\label{eqkljwhgdjhgdh}
\Ker(\pi^{(k-1,k)})=\Img(W^{(k),T})\text{ and } W^{(k)}W^{(k),T}=I_{\J^{(k)}}\,.
\end{equation}

Define
\begin{equation}
\chi_i^{(k)}:=\sum_{j\in \I^{(k)}} W^{(k)}_{i,j} \psi_j^{(k)}\ \ \text{ $k\in \{2,\ldots,q\}$ and $i\in \J^{(k)}$ }\,.
\end{equation}
Then $u^{(k)} - u^{(k-1)}$ is the  $\<\cdot,\cdot\>$ orthogonal projection of $u\in \cB$ onto
\begin{equation}
\W^{(k)}:=\Span\{\chi_i^{(k)}\mid i\in \J^{(k)}\}\,.
\end{equation}
We will also write $\J^{(1)}:=\I^{(1)}$, $\chi^{(1)}_i:=\psi_i^{(1)}$,  $\W^{(1)}:=\V^{(1)}$.  We call those operator adapted wavelets $\chi_i^{(k)}$, \emph{gamblets}. Furthermore $\W^{(k)}$ is the $\<\cdot,\cdot\>$-orthogonal complement of $\V^{(k-1)}$ in $\V^{(k)}$, i.e.
$
\V^{(k)}=\V^{(k-1)}\oplus \W^{(k)}\,,
$

\begin{equation}\label{eqn:decomposition}
\V^{(q)}=\V^{(1)}\oplus \W^{(2)}\oplus \cdots \oplus \W^{(q)}\,,
\end{equation}
and writing $\W^{(q+1)}$ for the $\<\cdot,\cdot\>$-orthogonal complement of $\V^{(q)}$ in $\cB$, $u=u^{(1)}+(u^{(2)}-u^{(1)})+\cdots+(u^{(q+1)}-u^{(q)})$ is the multiresolution decomposition of $u$ over $V=\V^{(1)}+\W^{(2)}+\cdots +\W^{(q+1)}$, namely, the \emph{gamblet decomposition} of $u$.

For $k\in \{2,\ldots,q\}$, $B^{(k)}= W^{(k)}A^{(k)}W^{(k),T}$ is the stiffness matrix of the $\chi^{(k)}_i$, i.e.
\begin{equation}
B^{(k)}_{i,j}= \<\chi_i^{(k)},\chi_j^{(k)}\> \ \ \text{ for }i,j\in \J^{(k)}\,.
\end{equation}
and $B^{(1)}:=A^{(1)}$.

\paragraph{Quantitative estimates.} 
Under general stability conditions on the $\phi_i^{(k)}$  \cite{OwhadiMultigrid:2017, Owhadi2017a, OwhadiScovel:2017, SchaeferSullivanOwhadi17, OwhScobook2018} these operator adapted wavelets satisfy the  quantitative estimates of Property \ref{property:optimaldecomposition}, we will first state those estimates and provide an example of their validity in the general setting of Example \ref{example1}.

\begin{Property}
\label{property:optimaldecomposition}
 The following properties are satisfied for some constant  $C>0$ and $H\in (0,1)$:
\begin{enumerate}
\item [1] \emph{Approximation}:
\begin{equation}
\| u - u^{\br{k}}\|_{0}  \leq C H^k \|u - u^{\br{k}}\| \text{ for }u\in \cB,
\label{eqn:l2err}
\end{equation}
and
\begin{equation}
\| u - u^{\br{k}}\|   \leq C H^k \|\L u\|_{0} \text{ for }u\in \L^{-1}\cB_0\,.
\label{eqn:h1err}
\end{equation}

\item [2] \emph{Uniform bounded condition number}: Writing $\Cond(B)$ for the condition number of a matrix $B$  we have
 for $k\in \{1, \cdots, q\}$,
\begin{equation}\label{eqcond}
C^{-1}H^{-2(k-1)}I_{\J^{\br{k}}}\leq B^{(k)} \leq CH^{-2k}I_{\J^{\br{k}}}\text{ and }\Cond(\Bbr{k})\leq CH^{-2}\,.
\end{equation}
and
\begin{equation}
C^{-1}I_{\I^{\br{k}}}\leq A^{(k)} \leq CH^{-2k}I_{\I^{\br{k}}}\,.
\label{eqn:condA}
\end{equation}
\item [3] \emph{Near linear complexity}: The wavelets $\psi_i^{(k)}, \chi_i^{(k)}$ and stiffness matrices $A^{(k)}, B^{(k)}$ can be computed to precision $\varepsilon$ (in $\|\cdot\|$-energy norm for elements of $\cB$ and in Frobenius norm for matrices) in $\operatorname{O}(N \operatorname{polylog} \frac{N}{\varepsilon})$ complexity.
\end{enumerate}
\end{Property}

\begin{figure}[H]
	\begin{center}
			\includegraphics[width=\textwidth]{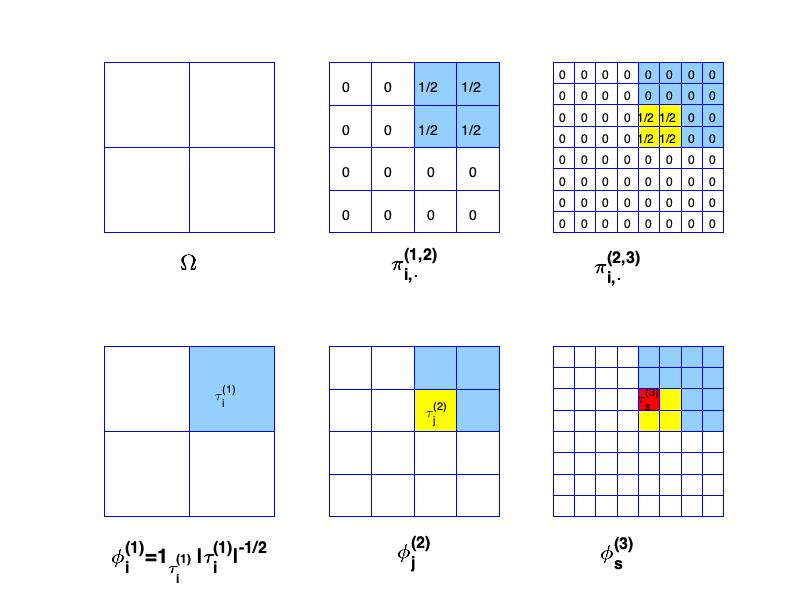}
		\caption{Nested partition of $\Omega=(0,1)^2$ such that the $k$th level corresponds to a uniform partition of $\Omega$ into $2^{-k}\times 2^{-k}$ squares. The top row shows the entries of $\pi^{(1,2)}_{i,\cdot}$ and $\pi^{(2,3)}_{j,\cdot}$. The bottom row shows the support of $\phi_i^{(1)}, \phi_j^{(2)}$ and $\phi_s^{(3)}$. Note that  $j^{(1)}=s^{(1)}=i$ and $s^{(2)}=j$.
}\label{figphipi}
	\end{center}
\end{figure}

\begin{Example}\label{example3}
Consider Example \ref{example1}. Let  $\I^{(q)}$ be the finite set of $q$-tuples of the form $i=(i_1,\ldots,i_q)$.  For $1\leq k<q$ and a $r$-tuple of the form $i=(i_1,\ldots,i_q)$, write $i^{(k)}:=(i_1,\ldots,i_k)$.
For $1\leq k \leq q$ and $i=(i_1,\ldots,i_q)\in \I^{(q)}$,  write $\I^{(k)}:=\{i^{(k)}\,:\, i\in \I^{(q)}\}$.
Let $\delta, h\in (0,1)$.
   Let  $(\tau_i^{(k)})_{i\in \I^{(k)}}$ be  uniformly Lipschitz convex sets forming a nested partition of $\Omega$, i.e. such that
$
   \Omega=\cup_{i\in \I^{(k)}}\tau_i^{(k)},\quad k \in \{1,\ldots, q\}
$
    is a disjoint union except for the boundaries,  and
$\tau_i^{(k)}=\cup_{j\in \I^{(k+1)}: j^{(k)}=i}\tau_j^{(k+1)}, \quad  k \in \{1,\ldots, q-1\}$.
    Assume that each $\tau_i^{(k)}$,  contains a ball  of center $x_i^{(k)}$ and radius $\delta h^k$,
 and is contained in the  ball of center $x_i^{(k)}$ and radius $\delta^{-1} h^k$. Writing $|\tau^{(k)}_i|$ for the volume of $\tau^{(k)}_i$, take
 \begin{equation}
\phi_i^{(k)}:=1_{\tau^{(k)}_i} |\tau^{(k)}_i|^{-\frac{1}{2}}\,.
\end{equation}
The nesting relation \eqref{eqejhdgdhdld} is then satisfied with
$\pi^{(k,k+1)}_{i,j}:=
|\tau^{(k+1)}_j|^{\frac{1}{2}}|\tau^{(k)}_i|^{-\frac{1}{2}}$ for $j^{(k)}=i$ and $\pi^{(k,k+1)}_{i,j}:=0$ otherwise.
Observe also that $\|\sum_i x_i \phi_i^{(k)}\|_{L^2(\Omega)}^2=|x|^2$.
For $i:=(i_1,\ldots,i_{k+1})\in \I^{(k+1)}$
write $i^{(k)}:=(i_1,\ldots,i_k)\in \I^{(k)}$ and
note that
$\pi^{(k,k+1)}$ is cellular in the sense that $\pi^{(k,k+1)}_{i,j}=0$ for $j^{(k)}\not=i$.
Choose $(\J^{(k)})_{2\leq k\leq q}$ to be  a finite set of $k$-tuples of the form $j=(j_1,\ldots,j_k)$ such that $j^{(k-1)}:=(j_1,\ldots,j_{k-1})\in \I^{(k-1)}$ and
$|\J^{(k)}|=|\I^{(k)}|-|\I^{(k-1)}|$. See Figure \ref{figphipi} for an illustration.
 Choose  $W^{(k)}$ as in \eqref{eqkljwhgdjhgdh}
and cellular in the sense that $W^{(k)}_{i,j}=0$ for $i^{(k-1)}\not= j^{(k-1)}$ (see \cite{OwhadiMultigrid:2017, OwhadiScovel:2017,  OwhScobook2018, Owhadi2017a} for examples).
\eqref{eqn:decomposition} then corresponds to a multi-resolution decomposition of $\H^s_0(\Omega)$ that is adapted to the operator $\L$.
\end{Example}

We have the following theorem \cite{OwhadiScovel:2017,  OwhScobook2018,SchaeferSullivanOwhadi17}.

\begin{Theorem}\label{thmgam}
The  properties in Property \ref{property:optimaldecomposition} are satisfied
for Examples \ref{example1} and \ref{example3} with $H=h^s$ and a constant $C$ depending only on
$\delta, \Omega, d, s$,
\begin{equation}
\label{eqkdjhekdfhfu}
 \|\L\|:=\sup_{u\in \H^s_0(\Omega)} \frac{\|\L u\|_{\H^{-s}(\Omega)}}{ \|u\|_{\H^s_0(\Omega)}}\ \text{ and }\
\|\L^{-1}\|:=\sup_{u\in \H^s_0(\Omega)} \frac{ \|u\|_{\H^s_0(\Omega)}}{\|\L u\|_{\H^{-s}(\Omega)}}\,.
\end{equation}
Furthermore, the wavelets $\chi_i^{(k)}$ and $\psi_i^{(k)}$ are exponentially localized, i.e.
\begin{equation}
\|\psi_i^{(k)}\|_{\H^s(\Omega\setminus B(x_i^{(k)},nh^k))} \leq  C h^{-sk}e^{-n/C}\text{ and }\|\chi_i^{(k)}\|_{\H^s(\Omega\setminus B(x_i^{(k)},nh^{k-1}))} \leq  C h^{-sk}e^{-n/C}\,,
\label{eqn:decay}
\end{equation}
 and
the wavelets $\psi_i^{(k)}, \chi_i^{(k)}$ and stiffness matrices $A^{(k)}, B^{(k)}$ can be computed to precision $\varepsilon$ (in $\|\cdot\|$-energy norm for elements of $\cB$ and in Frobenius norm for matrices)
 in $\operatorname{O}(N \log^{2d+1} \frac{N}{\varepsilon})$ complexity \cite{OwhScobook2018}.
\end{Theorem}

\begin{Remark}
Rigorous exponential decay/localization results such as \eqref{eqn:decay} have been pioneered in
\cite{MaPe:2012} for the localized orthogonal decomposition (LOD)  basis functions.
Although gamblets are derived from a  different perspective (namely, a game theoretic approach),
 from the numerical point of view, gamblets can be seen as a multilevel generalization of optimal recovery splines \cite{micchelli1977survey} and of numerical homogenization basis functions such as RPS (rough polyharmonic splines) \cite{OwhadiZhangBerlyand:2014} and variational multiscale/LOD basis functions \cite{hughes1998variational, MaPe:2012}.
\end{Remark}

\begin{Remark}
For Examples \ref{example1} and \ref{example3},
the wavelets $\psi_i^{(k)}, \chi_i^{(k)}$ and stiffness matrices $A^{(k)}, B^{(k)}$ can also be computed in $\operatorname{O}(N \log^{2d} \frac{N}{\varepsilon})$ complexity using the incomplete Cholesky factorization approach of \cite{SchaeferSullivanOwhadi17}.
\end{Remark}

\paragraph{Discrete case}
From now on we will consider the situation where $\cB$ is finite-dimensional and
$\V^{(q)}=\cB$. In the setting of Example  \ref{example1} we will identify $\cB$ with the linear space spanned by the finite-elements $\tilde{\psi}_i$ (e.g.~piecewise linear or bi-linear tent functions on a fine mesh/grid in the setting of Example \ref{example2}) used  to discretize the operator $\L$, use $\I^{(q)}$ to label the elements $\tilde{\psi}_i$ and set $\psi_i^{(q)}=\tilde{\psi}_i^{(q)}$ for $i\in \I^{(q)}$.
The gamblet transform \cite{OwhadiMultigrid:2017,  OwhadiScovel:2017,  OwhScobook2018} is then summarized in Algorithm \ref{alggtphiior} and we have the decompsoition
\begin{equation}\label{eqn:decompositionexacter}
\cB=\W^{(1)}\oplus \W^{(2)}\oplus \cdots \oplus \W^{(q)}\,.
\end{equation}

\begin{algorithm}[H]
\caption{The Gamblet Transform.}\label{alggtphiior}
\begin{algorithmic}[1]
\STATE\label{step3g} $\psi^{(q)}_i= \tilde{\psi}_i$
\STATE\label{step5g} $A^{(q)}_{i,j}= \< \psi_i^{(q)}, \psi_j^{(q)}\>$
\FOR{$k=q$ to $2$}
\STATE\label{step7g} $B^{(k)}= W^{(k)}A^{(k)}W^{(k),T}$
\STATE\label{step9g}  $\chi^{(k)}_i=\sum_{j \in \I^{(k)}} W_{i,j}^{(k)} \psi_j^{(k)}$
\STATE\label{step12g} $ R^{(k-1,k)}=\pi^{(k-1,k)} (I^{(k)}-A^{(k)} W^{(k),T}B^{(k),-1}
W^{(k)})$
\STATE\label{step13g} $A^{(k-1)}= R^{(k-1,k)}A^{(k)}R^{(k,k-1)}$
\STATE\label{step14g} $\psi^{(k-1)}_i=\sum_{j \in  \I^{(k)}} R_{i,j}^{(k-1,k)} \psi_j^{(k)}$
\ENDFOR
\end{algorithmic}
\end{algorithm}

\paragraph{Fast gamblet transform}
The acceleration of Algorithm \ref{alggtphiior} to
$\operatorname{O}(N \log^{2d+1} \frac{N}{\varepsilon})$ complexity is based on the truncation and localization of the computation of the interpolation matrices $R^{(k,k+1)}$ that is enabled by the exponential decay of  gamblets and the uniform bound on $\Cond(B^{(k)})$.
In the setting of Examples \ref{example1} and \ref{example3}, this acceleration is equivalent to localizing the computation of each gamblet $\psi_i^{(k)}$ to a sub-domain centered on $\tau_i^{(k)}$ and of diameter $\mathcal{O}( H^k \log \frac{1}{H^k})$.
We refer to
\cite{OwhadiMultigrid:2017,  OwhadiScovel:2017,  OwhScobook2018} for a detailed description of this acceleration.

\paragraph{Higher order problems}
Although the local linear elliptic operators of Example \ref{example1} are used as prototypical examples,
the proposed theory and algorithms is presented in the abstract setting of linear operators on Hilbert spaces to not only emphasize
the generality of the proposed method (which could also be applied  to Graph Laplacians with well behaved gamblets) but also to clarify/simplify its application  to higher order eigenvalue problems. For such applications the method is directly applied to the SPD matrix representation $A$ of the discretized  operator as described in \cite[Chap.~21]{OwhScobook2018}. The identification of level $q$ gamblets $\psi_i^{(q)}$ in Step \ref{step3g} of Algorithm \ref{alggtphiior} with the finite elements $\tilde{\psi}_i$ used to discretize the operator is, when the gamblet transform is applied to the SPD matrix $A$,  equivalent to the identification of level $q$ gamblets $\psi_i^{(q)}$  with  the unit vectors of $\R^N$. The nesting matrices $\pi^{(k-1,k)}$ remain those associated with the Haar pre-wavelets of Example \ref{example3} (the algorithm does not require the explicit input of measurement functions, only those nesting matrices are used as inputs and they remain unchanged). Of course fine scale finite elements $\tilde{\psi}_i$ (used to discretize the operator) have to be of sufficient accuracy for the approximation of the required eigenpairs (see \cite{brenner2014c0, zhang2017multi} and references therein for further discussion of the discretization issue).

\subsection{Gamblet based Multigrid Method}

The gamblet decomposition enables the construction of efficient multigrid solvers and preconditioners. Suppose we have computed the decomposition \eqref{eqn:decompositionexacter}, the stiffness matrices $A^{(k)}$ and interpolation matrices $R^{(k-1,k)}$ in Algorithm \ref{alggtphiior}, or more precisely their numerical approximations using the fast gamblet transform \cite{OwhadiMultigrid:2017,  OwhadiScovel:2017, SchaeferSullivanOwhadi17, OwhScobook2018}, to a degree that is sufficient to obtain grid-size accuracy in the resolution of the discretization of \eqref{eq:problem}.
We will write $R^{(k,k-1)}:=(R^{(k-1,k)})^T$ for the restriction matrix associated with the interpolation matrix $R^{(k-1,k)}$.

For $g^{(k)}\in \R^{\I^{(k)}}$ consider the linear system
\begin{equation}\label{eqjhhkhguyguyyj}
\Abr{k} z = g^{(k)}.
\end{equation}
Algorithm \ref{alg:multigrid} provides a multigrid approximation $\MG(k, z_0, g^{(k)})$ of the solution $z$ of
\eqref{eqjhhkhguyguyyj} based on an initial guess $z_0$ and a number of iterations $k$.
In that algorithm,  $m_1$ and $m_2$ are nonnegative integers (and $p=1$ or $2$.  $p=1$
corresponds to a $\cV$-cycle method and $p=2$ corresponds to a $\cW$-cycle method). $\Lbr{k}$ is an upper bound for the spectral radius of $\Abr{k}$. Under Condition \ref{property:optimaldecomposition} we take $\Lbr{k}= C H^{-2k}$ where $C$
and $H$ are the constants appearing in the bound $A^{(k)} \leq CH^{-2k}I_{\I^{\br{k}}}$.

\begin{Remark}
We use the simple Richardson iteration in the smoothing step of Algorithm \ref{alg:multigrid}. In practice, Gauss-Seidel and CG can also be used as a smoother.
\end{Remark}

\begin{Remark}\label{Computational_Work_Proposition}
The number of operations  required in the $k$-th level iteration defined by Algorithm \ref{alg:multigrid}
is $\ds\O(N_k(\log \frac{N_k}{\ve})^{2d+1})$, where $N_k:=\textrm{dim}(\V^{\br{k}})$.
\end{Remark}

\begin{algorithm}[H]
\caption{Gamblet based Multigrid ($k$-th Level Iteration)}\label{alg:multigrid}
For $k=1$, $\MG(1, z_0, g^{(1)})$ is the solution obtained from a direct method. Namely
\begin{eqnarray}
A^{(1)}\MG(1, z_0, g^{(1)}) = g^{(1)}.
\end{eqnarray}
For $k>1$, $\MG(k, z_0, g^{(k)})$ is obtained recursively in three steps,
\begin{enumerate}
\item Presmoothing: For $1\leq \ell \leq m_1$, let
\begin{eqnarray}
z_\ell = z_{\ell-1} + \frac{1}{\Lbr{k}} (g^{(k)}-\Abr{k} z_{\ell-1}),
\end{eqnarray}
\item Error Correction: Let $g^{(k-1)} := R^{(k-1, k)}(g^{(k)}-\Abr{k} z_0)$ and $q_0^{(k-1)} = 0$. For $1\leq i\leq p$, let
\begin{eqnarray}
q_i^{(k-1)} = \MG(k-1, q_{i-1}^{(k-1)}, g^{(k-1)}).
\end{eqnarray}
Then $z_{m_1+1} := z_{m_1} + R^{(k, k-1)}q_p^{(k-1)}$.
\item Postsmoothing: For $m_1+2\leq \ell \leq m_1+m_2 +1$, let
\begin{eqnarray}
z_\ell = z_{\ell-1} + \frac{1}{\Lbr{k}} (g^{(k)}-\Abr{k} z_{\ell-1}).
\end{eqnarray}
\end{enumerate}
Then the output of the $k^{\th}$ level iteration is
\begin{eqnarray}
\MG(k, z_0, g^{(k)}) : = z_{m_1+m_2+1}.
\end{eqnarray}
\end{algorithm}

%

Items 1 and 2 of Property \ref{property:optimaldecomposition} (i.e. the bounds on approximation errors and condition numbers)  imply the following result of $\cV$-cycle convergence.


\begin{Theorem}[Convergence of the $k^{\th}$ Level Iteration]\label{Convergence_MG_k_Theorem}
Let $m_1=m_2=m/2$, and $k$ be the level number of Algorithm \ref{alg:multigrid}, and $p=1$.
For any $0<\theta<1$, there exists  $m$ independent from $k$ such that
\begin{equation}
\|z-\MG(k, z_0, g)\|_{A^{(k)}} \leq \theta \|z-z_0\|_{A^{(k)}}, 
\label{eqn:vcycleconvergence}
\end{equation}
\end{Theorem}

\begin{proof}
Theorem \ref{Convergence_MG_k_Theorem} follows  from the following
 smoothing  and approximation properties introduced in \cite[Section 3.3.7]{Olshanskii:2014}:
	
	\emph{Smoothing property}: The iteration matrix on every grid level can be written as $S^{(k)} = I - M^{(k),-1} A^{(k)}$, where $M^{(k)}$ is symmetric and satisfies
	\begin{equation}\label{eqjhgugybhbh}
		M^{(k)}  \geq A^{(k)}
	\end{equation}
	
	\emph{Approximation property}: It holds true that
	\begin{equation}\label{eqjkffjhbjhbb}
		\| A^{(k),-1} - R^{(k,k-1)} A^{(k-1),-1}R^{(k-1,k)} \|_2 \leq C \|M^{(k)}\|_2^{-1}
	\end{equation}
	Taking $\ds M^{(k)} = \Abr{k} I^{(k)}=C H^{-2k}$ in Algorithm \ref{alg:multigrid} implies the smoothing property  \eqref{eqjhgugybhbh}  by equation \eqref{eqn:condA} in Property \ref{property:optimaldecomposition}.
There are two approaches to proving the approximation property \eqref{eqjkffjhbjhbb}.
The first one would be to adapt the classical approach as presented in \cite[p.130]{Olshanskii:2014} (in that approach \eqref{eqjkffjhbjhbb} is implied by equations \eqref{eqn:l2err} and \eqref{eqn:h1err} and the it requires the mass matrix of the gamblets to be well conditioned which follows from \cite[Thm.~6.3]{yoo2018noising}). Here we present a second approach.
 First observe that $D:=A^{(k),-1} - R^{(k,k-1)} A^{(k-1),-1}R^{(k-1,k)}$ is symmetric and positive. Indeed using $R^{(k-1,k)}=A^{(k-1)}\pi^{(k-1,k)}A^{(k),-1}$,
we have for $z=A^{(k)}y$, $z^T D z=y^T A^{(k)}y -y^T \pi^{(k,k-1)} A^{(k-1)}\pi^{(k-1,k)} y =\|u^{(k)}-u^{(k-1)}\|^2$ with $u^{(k)}=\sum_{i\in \I^{(k)}} y_i \psi_i^{(k)}$ (see \cite[Prop.~13.30]{OwhScobook2018} for details). Therefore $\|D\|_2=\sup_{|x|=1} x^T D x$. Now take $g=\sum_i x_i \phi_i^{(k)}$ and $u=\L^{-1} g$. Then $x^TA^{(k),-1}x=\|u\|^2$ and $x^T R^{(k,k-1)} A^{(k-1),-1}R^{(k-1,k)} x= \|u^{(k-1)}\|^2$.
Therefore $x^T D x =\|u\|^2-\|u^{(k-1)}\|^2=\|u-u^{(k-1)}\|^2$. Using  equation \eqref{eqn:h1err} in Property \ref{property:optimaldecomposition} we have $\|u-u^{(k-1)}\|^2 \leq C H^{2(k-1)} \|g\|_0^2$.
Using  $\|g\|_0^2\leq C |x|^2$ we deduce that
$\|D\|_2 \leq C H^{2(k-1)} $ which implies the result for $\ds M^{(k)} = \Abr{k} I^{(k)}=C H^{-2k}$ (a factor $H^{-2}$ is absorbed into $C$).

 	We conclude the proof of \eqref{eqn:vcycleconvergence} by applying \cite[Theorem 3.9]{Olshanskii:2014}, and taking $\ds \theta \geq \frac{C}{C+m}$, where $C$ is the constant in \eqref{eqjkffjhbjhbb}.

\end{proof}

\begin{Remark}
The LOD \cite{MaPe:2012} based   Schwarz\index{Schwarz} subspace decomposition/correction method of \cite{KornhuberYserentant16,Kornhuber:2018} also leads to a robust  two-level preconditioning method for PDEs with rough coefficients.
From a multigrid perspective, a multilevel version of \cite{KornhuberYserentant16,Kornhuber:2018}  would be closer to a domain decomposition/additive multigrid method compared to the proposed gamblet multigrid, which is a variant of the multiplicative multigrid method.
\end{Remark}

\section{Gamblet Subspace Correction Method for Eigenvalue Problem}
\label{sec:eigenproblem}

We will now describe the gamblet based multilevel  correction method.
Consider the abstract setting of Section \ref{secset} and write $\<\cdot,\cdot\>_0$ for the scalar product associated with the norm $\|\cdot\|_0$ placed on $\cB_0$. Since $[\cdot,\cdot]$ is the dual product between $\cB^*$ and $\cB$ induced by the Gelfand triple $\cB \subset \cB_0 \subset \cB^*$  we will also write $[u,v]:=\<u,v\>_0$ for $u,v\in \cB_0$.

Consider the  eigenvalue problem:
Find $(\lambda, v )\in \R\times \cB$, such that $\langle v,v\rangle=1$, and
\begin{eqnarray}\label{weak_eigenvalue_problem}
\langle v,w\rangle&=&\lambda  [v,w],\quad \forall w\in \cB.
\end{eqnarray}
The compact embedding property implies that the eigenvalue problem (\ref{weak_eigenvalue_problem})
has a sequence of eigenvalues $\{\lambda_j \}$ (see \cite{BabuskaOsborn_Book,Chatelin}):
$$0<\lambda_1\leq \lambda_2\leq\cdots\leq\lambda_\ell\leq\cdots,\ \ \
\lim_{\ell\rightarrow\infty}\lambda_\ell=\infty,$$
with associated eigenfunctions
$$v_1, v_2, \cdots, v_\ell, \cdots,$$
where $\langle v_i,v_j\rangle=\delta_{i,j}$ ($\delta_{i,j}$ denotes the Kronecker function).
In the sequence $\{\lambda_j\}$, the $\lambda_j$'s are repeated according to their geometric multiplicity.
For our analysis,  recall the following definition for the smallest eigenvalue (see \cite{BabuskaOsborn_Book,Chatelin})
\begin{eqnarray}\label{Smallest_Eigenvalue}
\lambda_1 = \min_{0\neq w\in V}\frac{\langle w,w\rangle}{[w,w]}.
\end{eqnarray}

Define the subspace approximation problem for eigenvalue problem (\ref{weak_eigenvalue_problem})
on  $\V^{\br{k}}$ as follows: Find $(\bar\lambda^{\br{k}}, \bar v^{\br{k}})\in \R\times \V^{\br{k}}$
such that $\langle\bar v^{(k)}, \bar v^{(k)}\rangle=1$ and
\begin{eqnarray}\label{Weak_Eigenvalue_Discrete}
\langle \bar v^{\br{k}},w\rangle
&=&\bar\lambda^{\br{k}} [\bar v^{(k)},w],\quad\ \  \ \forall w\in \V^{\br{k}}.
\end{eqnarray}
From \cite{BabuskaOsborn_1989,BabuskaOsborn_Book,Chatelin},
the discrete eigenvalue problem (\ref{Weak_Eigenvalue_Discrete}) has eigenvalues:
$$0<\bar \lambda_{1}^{\br{k}}\leq \bar \lambda_{2}^{\br{k}}\leq\cdots\leq \bar \lambda_{j}^{\br{k}}
\leq\cdots\leq \bar\lambda_{N_k}^{\br{k}},$$
and corresponding eigenfunctions
$$\bar v_{1}^{\br{k}}, \bar v_{2}^{\br{k}},\cdots, \bar v_j^{\br{k}} \cdots, \bar v_{N_k}^{\br{k}},$$
where $\langle\bar v_i^{\br{k}},\bar v_j^{\br{k}}\rangle=\delta_{i,j}, 1\leq i,j\leq N_k$, and $N_k:=\dim(\V^{(k)})$.

From the min-max principle \cite{BabuskaOsborn_1989,BabuskaOsborn_Book}, we have the following
upper bound result
\begin{eqnarray}\label{Uppero_Bound_Result}
\lambda_i \leq \bar\lambda_i^{(k)},\ \ \ \ 1\leq i\leq N_k.
\end{eqnarray}

Define
\begin{eqnarray}
\eta(\V^{\br{k}})&=&\sup_{f\in \cB_0,\|f\|_0=1}\inf_{w \in \V^{\br{k}}}\|\L^{-1}f-w\|.\label{eta_a_h_Def}
\end{eqnarray}

Let $M(\lambda_i)$ denote the eigenspace corresponding to the
eigenvalue $\lambda_i$, namely,
\begin{eqnarray}
M(\lambda_i):=\big\{v\in \cB\mid \langle v,w\rangle=\lambda_i  [v,w],\quad \forall w\in \cB \big\},
\end{eqnarray}
and define
\begin{eqnarray}
\delta_k(\lambda_i)=\sup_{v\in M(\lambda_i), \|v\|=1}\inf_{w\in\V^{\br{k}}}\|v-w\|.
\end{eqnarray}

\begin{Proposition}\label{propCondition_1}
Property \ref{property:optimaldecomposition} implies
\begin{equation}\label{eqCondition_1}
\eta(\V^{\br{k}})\leq C H^k,\ \ \ \delta_k(\lambda_i)\leq C \sqrt{\lambda_i} H^k\ \ \text{and}\ \ \delta_k(\lambda_i) \leq \sqrt{\lambda_i} \eta(\V^{\br{k}})\,,
\end{equation}
for $k\in \{1,\ldots,q\}$ where $C$ and $H$ are the constants appearing in Property \ref{property:optimaldecomposition}.
\end{Proposition}


In order to provide the error estimate for the numerical scheme (\ref{Weak_Eigenvalue_Discrete}), we define the corresponding projection operator $\mathcal P_k$ as follows
\begin{eqnarray}\label{FEM_Projection}
\langle \mathcal P_k u, w\rangle = \langle u, w\rangle,\ \ \ \ \forall w\in \V^{\br{k}}.
\end{eqnarray}
It is obvious that
\begin{eqnarray*}
\|u-\mathcal P_k u\| = \inf_{w\in\V^{\br{k}}}\|u-w\|.
\end{eqnarray*}

The following Rayleigh quotient expansion of the eigenvalue error is a useful tool
to obtain error estimates for eigenvalue approximations.

\begin{lemma}(\cite{BabuskaOsborn_1989})\label{Rayleigh_quotient_expansion_lem}
Assume $(\lambda,v)$ is an eigenpair for the eigenvalue problem (\ref{weak_eigenvalue_problem}). Then for any $w \in V\backslash\{0\}$,
the following expansion holds:
\begin{equation}\label{Rayleigh_quotient_expansion}
\frac{\langle w,w\rangle}{[w,w]} - \lambda = \frac{\langle w-u,w-u\rangle}{[w,w]} - \lambda\frac{[w-u,w-u]}{[w,w]},\ \ \ \forall u\in M(\lambda).
\end{equation}
\end{lemma}

For simplicity we will from now on restrict the presentation to the identification of a simple eigenpair $(\lambda,v)$ (the numerical method and results  can naturally be extended to multiple eigenpairs).
Let $E\,:\, \cB \rightarrow M(\lambda_i)$ be the spectral projection operator \cite{BabuskaOsborn_1989} defined by
\begin{eqnarray}\label{Spectral_Projection_Operator}
E=\frac{1}{2\pi \textrm{i}}\int_{\Gamma}\big(z-\L\big)^{-1}dz,
\end{eqnarray}
where $\Gamma$ is a Jordan curve in $\mathbb{C}$ enclosing the desired eigenvalue $\lambda_i$
and no other eigenvalues.


We introduce the following lemma from \cite{Strang:1973} before stating error
estimates of the subspace projection method. 
\begin{lemma}(\cite[Lemma 6.4]{Strang:1973})\label{Strang_Lemma}
For any exact eigenpair $(\lambda,v)$ of (\ref{weak_eigenvalue_problem}), the following equality holds
\begin{eqnarray*}\label{Strang_Equality}
(\bar{\lambda}_{j}^{\br{k}}-\lambda)[\mathcal P_kv, \bar v_j^{\br{k}}]
=\lambda [v-\mathcal P_kv, \bar v_{j}^{\br{k}}],\ \ \
j = 1, \cdots, N_k.
\end{eqnarray*}
\end{lemma}

The following lemma gives the error estimates for the gamblet subspace approximation, which is a direct application of the subspace approximation theory for eigenvalue problems, see \cite[Lemma 3.6, Theorem 4.4]{BabuskaOsborn_1989} and \cite{Chatelin}.
\begin{lemma}\label{Err_Eigen_Global_Lem}
Let  $(\lambda,v)$ denote an exact eigenpair of the eigenvalue problem (\ref{weak_eigenvalue_problem}).
Assume the eigenpair approximation $(\bar\lambda_i^{\br{k}},\bar v_i^{\br{k}})$ has the property that
$\bar\mu_i^{\br{k}}=1/\bar\lambda_i^{\br{k}}$ is closest to $\mu=1/\lambda$.
The corresponding spectral projection $E_{i,k}: V\mapsto {\mathrm span}\{\bar v_i^{\br{k}}\}$ is defined as follows
\begin{eqnarray*}
\langle E_{i,k}w, \bar v_i^{\br{k}}\rangle = \langle w, \bar v_i^{\br{k}}\rangle,\ \ \ \ \forall w\in V.
\end{eqnarray*}

The eigenpair approximations
$(\bar \lambda_{i}^{\br{k}},\bar v_{i}^{\br{k}})$ $(i = 1, 2, \cdots, N_k)$ has the following error estimates
\begin{eqnarray}
\|E_{i,k}v- v\|
&\leq& \sqrt{1+\frac{1}{\lambda_1\delta_{\lambda}^{(k),2}}\eta^2(\V^{\br{k}})}\delta_k(\lambda_i),\label{Err_Eigenfunction_Global_1_Norm} \\
\|E_{i,k}v- v\|_0
&\leq& \Big(1+\frac{1}{\lambda_1\delta_{\lambda}^{(k)}}\Big)\eta(\V^{\br{k}})\|E_{i,k}v-v\|,\label{Err_Eigenfunction_Global_0_Norm}
\end{eqnarray}
where $\delta_{\lambda}^{(k)}$ is defined as follows
\begin{eqnarray}\label{Definition_Delta}
\delta_{\lambda}^{(k)} &:=&\min_{j\neq i} \Big|\frac{1}{\bar\lambda_j^{\br{k}}}-\frac{1}{\lambda}\Big|
\end{eqnarray}
and $\delta_{\lambda}^{(k),2} = (\delta_{\lambda}^{(k)})^2$.
\end{lemma}
\begin{proof}
Following a classical duality argument found in  finite element method, we have
\begin{eqnarray}\label{L2_Energy_Estiate}
&&\|(I-\mathcal P_k)u\|_0=\sup_{\|g\|_0=1}[(I-\mathcal P_k)u,g]
=\sup_{\|g\|_0=1}\langle(I-\mathcal P_k)u,\mathcal L^{-1}g\rangle\nonumber\\
&&=\sup_{\|g\|_0=1}\langle(I-\mathcal P_k)u,(I-\mathcal P_k)\mathcal L^{-1}g\rangle
\leq \eta(\V^{\br{k}})\|(I-\mathcal P_k)u\|.
\end{eqnarray}

Since $(I-E_{i,k})\mathcal P_kv\in\V^{\br{k}}$ and
$\langle(I-E_{i,k})\mathcal P_kv, \bar v_i^{\br{k}}\rangle=0$,
the following orthogonal expansion holds
\begin{eqnarray}\label{Orthogonal_Decomposition}
(I-E_{i,k})\mathcal P_kv=\sum_{j\neq i}\alpha_j\bar v_{j}^{\br{k}},
\end{eqnarray}
where $\alpha_j=\langle\mathcal P_kv,\bar v_{j}^{\br{k}}\rangle$. From Lemma \ref{Strang_Lemma}, we have
\begin{eqnarray}\label{Alpha_Estimate}
\alpha_j&=&\langle\mathcal P_kv,\bar v_{j}^{\br{k}}\rangle = \bar\lambda_{j}^{\br{k}}
[\mathcal P_kv,\bar v_j^{\br{k}}]
=\frac{\bar\lambda_{j}^{\br{k}}\lambda}{\bar\lambda_{j}^{\br{k}}-\lambda}\big[v-\mathcal P_kv,\bar v_j^{\br{k}}\big]\nonumber\\
&=&\frac{1}{\mu-\bar\mu_{j}^{\br{k}}}\big[v-\mathcal P_kv,\bar v_j^{\br{k}}\big],
\end{eqnarray}
where $\mu = 1/\lambda$ and $\bar\mu_{j}^{\br{k}} = 1/\bar\lambda_j^{\br{k}}$.

From the property of eigenvectors $\bar v_1^{\br{k}},\cdots, \bar v_{m}^{\br{k}}$, the following identities hold
\begin{eqnarray*}
1 = \langle\bar v_j^{\br{k}},\bar v_j^{\br{k}}\rangle= \bar\lambda_j^{\br{k}}\big[\bar v_j^{\br{k}},\bar v_j^{\br{k}}\big]=\bar\lambda_j^{\br{k}}\|\bar v_j^{\br{k}}\|_0^2,
\end{eqnarray*}
which leads to the following property
\begin{eqnarray}\label{Equality_u_j}
\|\bar v_j^{\br{k}}\|_0^2=\frac{1}{\bar\lambda_j^{\br{k}}}=\bar\mu_j^{\br{k}}.
\end{eqnarray}
From (\ref{Weak_Eigenvalue_Discrete}) and definitions of eigenvectors $\bar v_1^{\br{k}},\cdots, \bar u_m^{\br{k}}$,
we have the following equalities
\begin{eqnarray}\label{Orthonormal_Basis}
\langle\bar v_j^{\br{k}},\bar v_{i}^{\br{k}}\rangle=\delta_{ij},
\ \ \ \ \ \Big[\frac{\bar v_j^{\br{k}}}{\|\bar v_j^{\br{k}}\|_0},
\frac{\bar v_{i}^{\br{k}}}{\|\bar v_{i}^{\br{k}}\|_0}\Big]=\delta_{ij},\ \ \ 1\leq i, j\leq N_k.
\end{eqnarray}
Combining (\ref{Orthogonal_Decomposition}), (\ref{Alpha_Estimate}), (\ref{Equality_u_j})
and (\ref{Orthonormal_Basis}),
the following estimates hold
\begin{eqnarray}\label{Equality_5}
&&\|(I-E_{i,k})\mathcal P_kv\|^2=\Big\|\sum_{j\neq i}\alpha_j\bar v_j^{\br{k}}\Big\|^2
=\sum_{j\neq i}\alpha_j^2=\sum_{j\neq i}\Big(\frac{1}{\mu-\bar\mu_{j}^{\br{k}}}\Big)^2
\Big|\big[v-\mathcal P_kv,\bar v_j^{\br{k}}\big]\Big|^2\nonumber\\
&&\leq\frac{1}{\delta_{\lambda}^{(k),2}}\sum_{j\neq i}\|\bar v_j^{\br{k}}\|_0^2
\Big|\Big[v-\mathcal P_kv,\frac{\bar v_j^{\br{k}}}{\|\bar v_j^{\br{k}}\|_0}\Big]\Big|^2
=\frac{1}{\delta_{\lambda}^{(k),2}}\sum_{j\neq i}\bar\mu_{j}^{\br{k}}
\Big|\Big[v-\mathcal P_kv,\frac{\bar v_j^{\br{k}}}{\|\bar v_j^{\br{k}}\|_0}\Big]\Big|^2\nonumber\\
&&\leq \frac{\bar\mu_{1}^{\br{k}}}{\delta_{\lambda}^{(k),2}}\|v-\mathcal P_kv\|_0^2.
\end{eqnarray}
From (\ref{L2_Energy_Estiate}), (\ref{Equality_5}) and the orthogonal property
$\langle v-\mathcal P_kv, (I-E_{i,k})\mathcal P_kv\rangle=0$,
we have the following error estimates
\begin{eqnarray*}\label{Error_estimate_Energy_Discrete}
&&\|v-E_{i,k}v\|^2=\|v-\mathcal P_kv\|^2
+\|(I-E_{i,k})\mathcal P_kv\|^2\nonumber\\
&&\leq\|(I-\mathcal P_k)v\|^2
+\frac{\bar\mu_{1}^{\br{k}}}{\delta_{\lambda}^{(k),2}}\|v-\mathcal P_kv\|_0^2
\leq \Big(1+\frac{\bar\mu_{1}^{\br{k}}}{\delta_{\lambda}^{(k),2}}\eta(\V^{\br{k}})^2\Big)\|(I-\mathcal P_k)v\|^2,
\end{eqnarray*}
which is the desired result (\ref{Err_Eigenfunction_Global_1_Norm}).

Similarly, combining (\ref{Orthogonal_Decomposition}), (\ref{Alpha_Estimate}), (\ref{Equality_u_j}) and (\ref{Orthonormal_Basis}),
leads to the following estimates
\begin{eqnarray}\label{Equality_6}
&&\|(I-E_{i,k})\mathcal P_kv\|_0^2 = \Big\|\sum_{j\neq i}\alpha_j\bar v_j^{\br{k}}\Big\|_0^2
=\sum_{j\neq i}\alpha_j^2\|\bar v_j^{\br{k}}\|_0^2 \nonumber\\
&=&\sum_{j\neq i}\Big(\frac{1}{\mu-\bar\mu_j^{\br{k}}}\Big)^2
\Big|\big[v-\mathcal P_kv,\bar v_j^{\br{k}}\big]\Big|^2\|\bar v_j^{\br{k}}\|_0^2\nonumber\\
&\leq&\frac{1}{\delta_{\lambda}^{(k),2}}\sum_{j\neq i}
\Big|\Big[v-\mathcal P_kv,\frac{\bar v_j^{\br{k}}}{\|\bar v_j^{\br{k}}\|_0}\Big]\Big|^2
\|\bar v_j^{\br{k}}\|_0^4\nonumber\\
&=&\frac{1}{\delta_{\lambda}^{(k),2}}\sum_{j\neq i}(\bar\mu_j^{\br{k}})^2
\Big|\Big[v-\mathcal P_kv,\frac{\bar v_j^{\br{k}}}{\|\bar v_j^{\br{k}}\|_0}\Big]\Big|^2
\leq \Big(\frac{\bar\mu_{1}^{\br{k}}}{\delta_{\lambda}^{(k)}}\Big)^2\|v-\mathcal P_kv\|_0^2.
\end{eqnarray}

By (\ref{L2_Energy_Estiate}) and (\ref{Equality_6}), we have the following inequalities
\begin{eqnarray}\label{Equality_8}
\|(I-E_{i,k})\mathcal P_kv\|_0 \leq \frac{\bar\mu_{1}^{\br{k}}}{\delta_{\lambda}^{(k)}}\|v-\mathcal P_kv\|_0
\leq \frac{\bar\mu_{1}^{\br{k}}}{\delta_{\lambda}^{(k)}}\eta(\V^{\br{k}})\|(I-\mathcal P_k)v\|.
\end{eqnarray}
From (\ref{L2_Energy_Estiate}), (\ref{Equality_8}) and the triangle inequality,
we conclude that the following error estimate for the eigenvector approximation in $L^2$-norm holds,
\begin{eqnarray}\label{Inequality_6}
&&\|v-E_{i,k}v\|_0\leq \|v-\mathcal P_kv\|_0
+ \|(I-E_{i,k})\mathcal P_kv\|_0\nonumber\\
&\leq&\|v-\mathcal P_kv\|_0
+ \frac{\bar\mu_{1}^{\br{k}}}{\delta_{\lambda}^{(k)}}\eta(\V^{\br{k}})\|(I-\mathcal P_k)v\|\nonumber\\
&\leq&\Big(1+\frac{\bar\mu_{1}^{\br{k}}}{\delta_{\lambda}^{(k)}}\Big)\eta(\V^{\br{k}})\|(I-\mathcal P_k)v\|\nonumber\\
&\leq&\Big(1+\frac{\bar\mu_{1}^{\br{k}}}{\delta_{\lambda}^{(k)}}\Big)\eta(\V^{\br{k}})\|(I- E_{i,k})v\|.
\end{eqnarray}
This is the second desired result (\ref{Err_Eigenfunction_Global_0_Norm}) and the proof is complete.
\end{proof}

In order to analyze the method which will be given in this section, we state some error estimates in the following lemma.
\begin{lemma}\label{Error_Superclose_Lemma}
Under the conditions of Lemma \ref{Err_Eigen_Global_Lem}, the following error estimates hold
\begin{eqnarray}
\|v-\bar v_i^{\br{k}}\| &\leq& \sqrt{2\Big(1+\frac{1}{\lambda_1\delta_{\lambda}^{(k),2}}\eta^2(\V^{\br{k}})\Big)}
\|(I-\mathcal P_k)v\|,\label{Err_Norm_1}\\
\|\lambda v - \bar\lambda_{i}^{\br{k}}\bar v_i^{\br{k}}\|_0 &\leq& C_\lambda
\eta(\V^{\br{k}})\|v-\bar v_i^{\br{k}}\|, \label{Err_Norm_0_Lambda}\\
\|v-\bar v_i^{\br{k}}\| &\leq& \frac{1}{1-D_\lambda\eta(\V^{\br{k}})}\|v-\mathcal P_kv\|,\label{Err_Norm_1_Superclose}
\end{eqnarray}
where
\begin{eqnarray}
C_\lambda = 2|\lambda|\Big(1+\frac{1}{\lambda_1\delta_{\lambda}^{(k)}}\Big)
+ \bar\lambda_i^{\br{k}}\sqrt{1+\frac{1}{\lambda_1\delta_{\lambda}^{(k),2}}\eta^2(\V^{\br{k}})^2},
\end{eqnarray}
and
\begin{eqnarray}\label{Definition_D_Lambda}
D_\lambda = \frac{1}{\sqrt{\lambda_1}}\left(2|\lambda|\Big(1+\frac{1}{\lambda_1\delta_{\lambda}^{(k)}}\Big)+ \bar\lambda_i^{\br{k}} \sqrt{1+\frac{1}{\lambda_1\delta_{\lambda}^{(k),2}}\eta^2(\V^{\br{k}})}\right).
\end{eqnarray}
\end{lemma}

\begin{proof}
Let us set $\alpha>0$ such that $E_{i,k}v = \alpha \bar v_{i}^{\br{k}}$. Then it implies that
\begin{eqnarray}\label{Definition_Alpha}
1= \|v\| \geq \|E_{i,k}v\| = \alpha  \|\bar v_{i}^{\br{k}}\| = \alpha.
\end{eqnarray}
Based on the error estimates in Lemma \ref{Err_Eigen_Global_Lem}, the property $\|v\|=\|\bar v_i^{\br{k}}\|=1$
and (\ref{Definition_Alpha}), we have the following estimations
\begin{eqnarray}\label{Error_1}
&&\|v-\bar v_{i}^{\br{k}}\|^2 = \|v-E_{i,k}v\|^2+ \|\bar v_i^{\br{k}}-E_{i,k}v\|^2\nonumber\\
&=& \|v-E_{i,k}v\|^2 + \|\bar v_i^{\br{k}}\|^2 - 2\langle\bar v_i^{\br{k}}, E_{i,k}v\rangle
+ \|E_{i,k}v\|^2\nonumber\\
&=& \|v-E_{i,k}v\|^2 + 1 - 2\|\bar v_i^{\br{k}}\|\|E_{i,k}v\| + \|E_{i,k}v\|^2\nonumber\\
&=& \|v-E_{i,k}v\|^2 + \|v\|^2 - 2\|v\|\|E_{i,k}v\| + \|E_{i,k}v\|^2\nonumber\\
&\leq& \|v-E_{i,k}v\|^2 + \|v\|^2 - 2\langle v, E_{i,k}v\rangle + \|E_{i,k}v\|^2
\leq 2 \|v-E_{i,k}v\|^2.
\end{eqnarray}
(\ref{Err_Eigenfunction_Global_1_Norm}) and (\ref{Error_1}) lead to the desired result \eqref{Err_Norm_1}.

With the help of (\ref{Err_Eigenfunction_Global_0_Norm}) and the property (\ref{Definition_Alpha}) and
 $\|v\|_0 =\frac{1}{\sqrt{\lambda}} \geq \|\bar v_i^{\br{k}}\|_0 = \frac{1}{\sqrt{\bar\lambda_i^{\br{k}}}}$,
 we have the following estimates for $\|v - \bar v_i^{\br{k}}\|_0$
\begin{eqnarray}\label{Error_2}
&&\|v - \bar v_i^{\br{k}}\|_0 \leq \|v - E_{i,k}v\|_0 + \|E_{i,k}v-\bar v_i^{\br{k}}\|_0\nonumber\\
&=& \|v - E_{i,k}v\|_0  + \|\bar v_i^{\br{k}}\|_0 - \|E_{i,k}v\|_0
= \|v - E_{i,k}v\|_0  + \frac{1}{\sqrt{\bar\lambda_{i}^{\br{k}}}} - \|E_{i,k}v\|_0 \nonumber\\
&\leq & \|v - E_{i,k}v\|_0  + \frac{1}{\sqrt{\lambda}} - \|E_{i,k}v\|_0
=  \|v - E_{i,k}v\|_0  + \|v\|_0 - \|E_{i,k}v\|_0 \nonumber\\
&\leq & \|v - E_{i,k}v\|_0  + \|v - E_{i,k}v\|_0 \leq 2\|v - E_{i,k}v\|_0 \nonumber\\
&\leq& 2\Big(1+\frac{1}{\bar\lambda_1^{\br{k}}\delta_{\lambda}^{(k)}}\Big)\eta(\V^{\br{k}})\|v-E_{i,k}v\|
\leq 2\Big(1+\frac{1}{\lambda_1\delta_{\lambda}^{(k)}}\Big)\eta(\V^{\br{k}})\|v- \bar v_i^{\br{k}}\|.
\end{eqnarray}
From the expansion (\ref{Rayleigh_quotient_expansion}), the definition  (\ref{eta_a_h_Def}),
error estimate (\ref{Err_Eigenfunction_Global_1_Norm})
 and the property $\|\bar v_i^{\br{k}} - E \bar v_i^{\br{k}}\| = \|v-E_{i,k}v\|\leq \|v-\bar v_i^{\br{k}}\|$,
the following error estimates hold
\begin{eqnarray}\label{Error_3}
&&|\lambda-\bar\lambda_i^{\br{k}}| \leq \frac{\|\bar v_i^{\br{k}} - E \bar v_i^{\br{k}}\|^2}{\|\bar v_i^{\br{k}}\|_0^2}
 =\frac{\|v - E_{i,k} v\|^2}{\|\bar v_i^{\br{k}}\|_0^2} \nonumber\\
 &\leq& \bar\lambda_i^{\br{k}}\sqrt{1+\frac{1}{\lambda_1\delta_{\lambda}^{(k),2}}\eta^2(\V^{\br{k}})}
\|(I-\mathcal P_k)v\| \|v-\bar v_i^{\br{k}}\|\nonumber\\
&\leq& \lambda\bar\lambda_{i}^{\br{k}}\sqrt{1+\frac{1}{\lambda_1\delta_{\lambda}^{(k),2}}\eta^2(\V^{\br{k}})}
\|(I-\mathcal P_k)\mathcal L^{-1}v\| \|v-\bar v_i^{\br{k}}\|\nonumber\\
&\leq& \lambda\bar\lambda_{i}^{\br{k}}\sqrt{1+\frac{1}{\lambda_1\delta_{\lambda}^{(k),2}}\eta^2(\V^{\br{k}})}
\eta(\V^{\br{k}})\|v\|_0 \|v-\bar v_i^{\br{k}}\|\nonumber\\
&\leq& \sqrt{\lambda}\bar\lambda_{i}^{\br{k}}\sqrt{1+\frac{1}{\lambda_1\delta_{\lambda}^{(k),2}}\eta^2(\V^{\br{k}})}
\eta(\V^{\br{k}})\|v-\bar v_i^{\br{k}}\|.
\end{eqnarray}



Then the combination of (\ref{Error_2}), (\ref{Error_3}) and the property
$\|\bar v_i^{\br{k}}\|_0 = 1/\sqrt{\bar \lambda_{i}^{\br{k}}} \leq 1/\sqrt{\lambda}$ leads to the following estimate
\begin{eqnarray}\label{Error_5}
&&\|\lambda v-\bar\lambda_{i}^{\br{k}}\bar v_i^{\br{k}}\|_0 \leq |\lambda|\| v - \bar v_i^{\br{k}}\|_0
+ \|\bar v_i^{\br{k}}\|_0 |\lambda - \bar \lambda_{i}^{\br{k}}\|_0\nonumber\\
&\leq& \left(2|\lambda|\Big(1+\frac{1}{\lambda_1\delta_{\lambda}^{(k)}}\Big)+
\|\bar v_i^{\br{k}}\|_0\sqrt{\lambda}\bar\lambda_i^{\br{k}}\sqrt{1+\frac{1}{\lambda_1\delta_{\lambda}^{(k),2}}\eta(\V^{\br{k}})^2}
\right)\eta(\V^{\br{k}})\|v-\bar v_i^{\br{k}}\|\nonumber\\
&\leq& \left(2|\lambda|\Big(1+\frac{1}{\lambda_1\delta_{\lambda}^{(k)}}\Big)
+ \bar\lambda_i^{\br{k}}\sqrt{1+\frac{1}{\lambda_1\delta_{\lambda}^{(k),2}}\eta^2(\V^{\br{k}})^2}\right)\eta(\V^{\br{k}})\|v-\bar v_i^{\br{k}}\|,
\end{eqnarray}
which is the desired result (\ref{Err_Norm_0_Lambda}).

We now investigate the distance of $\mathcal P_k v$ from $\bar v_i^{\br{k}}$. First, the following
estimate holds
\begin{eqnarray}\label{Error_6}
&&\|\mathcal P_k v - \bar v_{i}^{\br{k}}\|^2 = \langle \mathcal P_k v-\bar v_{i}^{\br{k}},\mathcal P_k v-\bar v_{i}^{\br{k}}\rangle =
\langle v-\bar v_{i}^{\br{k}},\mathcal P_k v-\bar v_{i}^{\br{k}}\rangle\nonumber\\
&&=[\lambda v-\bar\lambda_{i}^{\br{k}}\bar v_i^{\br{k}}, \mathcal P_k v-\bar v_{i}^{\br{k}}]
\leq \|\lambda v-\bar\lambda_{i}^{\br{k}}\bar v_i^{\br{k}}\|_0 \|\mathcal P_k v-\bar v_{i}^{\br{k}}\|_0\nonumber\\
&&\leq \frac{1}{\sqrt{\lambda_1}} \|\lambda v-\bar\lambda_{i}^{\br{k}}\bar v_i^{\br{k}}\|_0 \|\mathcal P_k v-\bar v_{i}^{\br{k}}\|.
\end{eqnarray}
From (\ref{Error_5}) and (\ref{Error_6}), we have the following estimate
\begin{eqnarray}\label{Error_7}
&&\|\mathcal P_k v-\bar v_{i}^{\br{k}}\|
\leq \frac{1}{\sqrt{\lambda_1}} \|\lambda v-\bar\lambda_{i}^{\br{k}}\bar v_i^{\br{k}}\|_0\nonumber\\
&&\leq \frac{1}{\sqrt{\lambda_1}}\left(2|\lambda|
\Big(1+\frac{1}{\lambda_1\delta_{\lambda}^{(k)}}\Big)+ \bar\lambda_i^{\br{k}}
\sqrt{1+\frac{\bar\mu_{1}^{\br{k}}}{\delta_{\lambda}^{(k),2}}\eta(\V^{\br{k}})^2}
\right)\eta(\V^{\br{k}})\|v-\bar v_i^{\br{k}}\|.
\end{eqnarray}
\eqref{Error_7} and the triangle inequality lead to the following inequality
\begin{eqnarray*}
&&\|v-\bar v_i^{\br{k}}\| \leq \|v-\mathcal P_kv\|  + \|\mathcal P_k v-\bar v_{i}^{\br{k}}\|\nonumber\\
&&\leq \|v-\mathcal P_kv\|+ \frac{1}{\sqrt{\lambda_1}}\left(2|\lambda|\Big(1+\frac{1}{\lambda_1\delta_{\lambda}^{(k)}}\Big)
+ \bar\lambda_i^{\br{k}} \sqrt{1+\frac{\bar\mu_1^{\br{k}}}{\delta_{\lambda}^{(k),2}}\eta^2(\V^{\br{k}})}\right)\eta(\V^{\br{k}})
\|v-\bar v_i^{\br{k}}\|,
\end{eqnarray*}
which in turn implies that
\begin{eqnarray*}\label{Error_8}
\|v-\bar v_i^{\br{k}}\|&\leq& \frac{1}{1-\frac{1}{\sqrt{\lambda_1}}\left(2|\lambda|\Big(1+\frac{1}{\lambda_1\delta_{\lambda}^{(k)}}\Big)+ \bar\lambda_i^{\br{k}} \sqrt{1+\frac{1}{\lambda_1\delta_{\lambda}^{(k),2}}\eta^2(\V^{\br{k}})}\right)\eta(\V^{\br{k}})} \|v-\mathcal P_kv\|\nonumber\\
&\leq&\frac{1}{1-D_\lambda\eta(\V^{\br{k}})}\|v-\mathcal P_kv\|.
\end{eqnarray*}
This completes the proof of the desired result (\ref{Err_Norm_1_Superclose}).
%
\end{proof}


\subsection{One Correction Step}
To describe the multilevel correction method  we first present the ``one correction step''. Given an eigenpair approximation
$(\lambda^{(k,\ell)},v^{(k,\ell)})\in \R\times \V^{\br{k}}$, Algorithm \ref{Multigrid_Smoothing_Step} produces an improved
eigenpair approximation $(\lambda^{(k,\ell+1)},v^{(k,\ell+1)})\in \R\times \V^{\br{k}}$.
In this algorithm, the superscript $(k,\ell)$ denotes the $\ell$-th correction step in the $k$-th level gamblet space.

\begin{algorithm}[H]\caption{One Correction Step}\label{Multigrid_Smoothing_Step}
\begin{enumerate}
\item Let $\widetilde{v}^{(k,\ell+1)}\in \V^{\br{k}}$ be the solution of the linear system
\begin{eqnarray}\label{aux_problem}
\langle \widetilde{v}^{(k,\ell+1)}, w\rangle&=&
\lambda^{(k,\ell)} [v^{(k,\ell)},w],\ \
\ \forall w\in \V^{\br{k}}.
\end{eqnarray}
Approximate $\widetilde{v}^{(k,\ell+1)}$ by $\widehat v^{(k,\ell+1)} = \MG(k, v^{(k,\ell)}, \lambda^{(k,\ell)} v^{(k,\ell)}) $ using Algorithm \ref{alg:multigrid}.

\item  Let $\V^{\br{1}}$ be the coarsest gamblet space, define
$$\V^{\br{1,k}}=\V^{\br{1}}+{\textrm{span}}\{\widehat{v}^{(k,\ell+1)}\}$$ and solve
the subspace eigenvalue problem: Find $(\lambda^{(k,\ell+1)},v^{(k,\ell+1)})\in\R\times \V^{\br{1,k}}$ such
that $\langle v^{(k,\ell+1)},v^{(k,\ell+1)}\rangle = 1$ and
\begin{eqnarray}\label{Eigen_Augment_Problem}
\langle v^{(k,\ell+1)}, w\rangle &=&\lambda^{(k,\ell+1)}[v^{(k,\ell+1)},w],\ \ \
\forall w\in \V^{\br{1,k}}.
\end{eqnarray}
\end{enumerate}

Let ${\texttt{EigenMG}}$ be the function summarizing the action of the steps described above, i.e.
\begin{eqnarray*}
(\lambda^{(k,\ell+1)},v^{(k,\ell+1)})
={\texttt{EigenMG}}(\V^{\br{1}},\lambda^{(k,\ell)}, v^{(k,\ell)},\V^{\br{k}}).
\end{eqnarray*}
\end{algorithm}

\begin{Remark}
\revise{Notice that in \eqref{Eigen_Augment_Problem}, the orthogonalization is only performed in the coarse space $\V^{\br{1,k}}$ with dimension $1+{\mathrm dim}\V^{(1)}$. }
\end{Remark}

For simplicity of notation, we assume that the eigenvalue gap $\delta_\lambda^{\br{k}}$
has a uniform lower bound which is denoted by $\delta_\lambda$ (which can be seen as the "true" separation of the eigenvalues) in the following parts of this paper.
This assumption is reasonable when the mesh size $H$ is small enough.
 We refer to
\cite[Theorem 4.6]{Saad:2011} for details on the dependence of error estimates on the eigenvalue gap.
Furthermore, we also assume the concerned eigenpair approximation $(\lambda^{\br{k,\ell}}, v^{\br{k,\ell}})$
is closet to the exact eigenpair $(\bar\lambda^{\br{k}}, \bar v^{\br{k}})$ of (\ref{Weak_Eigenvalue_Discrete})
and $(\lambda, v)$ of (\ref{weak_eigenvalue_problem}).
\begin{theorem}\label{Error_Estimate_One_Smoothing_Theorem}
Assume there exists exact eigenpair $(\bar\lambda^{\br{k}}, \bar v^{\br{k}})$ such
that the  eigenpair approximation $(\lambda^{(k,\ell)},v^{(k,\ell)})$ satisfies $\|v^{(k,\ell)}\|=1$ and
\begin{eqnarray}\label{Estimate_h_k_b}
\|\bar\lambda^{\br{k}}\bar v^{\br{k}}-\lambda^{(k,\ell)}v^{(k,\ell)}\|_0 \leq C_1 \eta (\V^{\br{1}})\|\bar v^{\br{k}}-v^{(k,\ell)}\|
\end{eqnarray}
for some constant  $C_1$.
The multigrid iteration for the linear equation (\ref{aux_problem}) has the following uniform contraction rate
\begin{eqnarray}\label{Contraction_Rate}
\|\widehat v^{(k,\ell+1)}-\widetilde v^{(k,\ell+1)}\|\leq \theta \|v^{(k,\ell)}-\widetilde v^{(k,\ell+1)}\|
\end{eqnarray}
with $\theta<1$ independent from $k$ and $\ell$.

Then the eigenpair approximation
$(\lambda^{(k,\ell+1)},v^{(k,\ell+1)})\in\mathbb{R}\times \V^{\br{k}}$ produced by
 Algorithm \ref{Multigrid_Smoothing_Step} satisfies
\begin{eqnarray}
\|\bar v^{\br{k}}-v^{(k,\ell+1)}\| &\leq & \gamma \|\bar v^{\br{k}}-v^{(k,\ell)}\|,\label{Estimate_h_k_1_a}\\
\|\bar\lambda^{\br{k}}\bar v^{\br{k}}-\lambda^{(k,\ell+1)}v^{(k,\ell+1)}\|_0&\leq&
\bar C_\lambda \eta (\V^{\br{1}})\|\bar v^{\br{k}}-v^{(k,\ell+1)}\|,\label{Estimate_h_k_1_b}
\end{eqnarray}
where the constants $\gamma$, $\bar C_\lambda$ and $D_\lambda$ are defined as follows
\begin{eqnarray}\label{Gamma_Definition}
\gamma = \frac{1}{1-\bar D_\lambda \eta(\V^{\br{1}})}
\Big(\theta+(1+\theta)\frac{C_1}{\sqrt{\lambda_1}}\eta (\V^{\br{1}})\Big)\,.
\end{eqnarray}
\begin{eqnarray}\label{Definition_C_Bar}
\bar C_\lambda = 2|\lambda|\Big(1+\frac{1}{\lambda_1\delta_\lambda}\Big)
+ \bar\lambda_i^{\br{1}}\sqrt{1+\frac{1}{\lambda_1\delta_\lambda^2}\eta^2(\V^{\br{1}})^2},
\end{eqnarray}
\begin{eqnarray}\label{Definition_D_Lambda_Bar}
\bar D_\lambda = \frac{1}{\sqrt{\lambda_1}}\left(2|\lambda|\Big(1+\frac{1}{\lambda_1\delta_\lambda}\Big)+ \bar\lambda_i^{\br{1}} \sqrt{1+\frac{1}{\lambda_1\delta_\lambda^2}\eta^2(\V^{\br{1}})}\right).
\end{eqnarray}
\end{theorem}
\begin{proof}
From \eqref{Smallest_Eigenvalue}, (\ref{Weak_Eigenvalue_Discrete}) and (\ref{aux_problem}), we have for  $w\in \V^{\br{k}}$
\begin{eqnarray*}\label{One_Correction_1}
&&\langle \bar v^{\br{k}}-\widetilde{v}^{(k,\ell+1)}, w\rangle
=[(\bar\lambda^{\br{k}}\bar v^{\br{k}}-\lambda^{(k,\ell)}v^{(k,\ell)}),w]\nonumber\\
&\leq&\|\bar\lambda^{\br{k}}\bar v^{\br{k}}-\lambda^{(k,\ell)}v^{(k,\ell)}\|_0\|w\|_0
\leq C_1\eta(\V^{\br{k}})\|\bar v^{\br{k}}-v^{(k,\ell)}\| \|w\|_0 \nonumber\\
&\leq& \frac{1}{\sqrt{\lambda_1}}C_1\eta(\V^{\br{k}})\|\bar v^{\br{k}}-v^{(k,\ell)}\| \|w\|.
\end{eqnarray*}
Taking $w = \bar v^{\br{k}}-\widetilde{v}^{(k,\ell+1)}$ we deduce from  \eqref{Estimate_h_k_b} that
\begin{eqnarray}\label{One_Correction_2}
\|\bar v^{\br{k}}-\widetilde{v}^{(k,\ell+1)}\| &\leq&\frac{C_1}{\sqrt{\lambda_1}}\eta (\V^{\br{1}})\| \bar v^{\br{k}}-v^{(k,\ell)}\|. 
\end{eqnarray}
Using \eqref{Contraction_Rate} and \eqref{One_Correction_2} we deduce that
\begin{eqnarray}
\|\bar v^{\br{k}}-\widehat{v}^{(k,\ell+1)}\|&\leq& \|\bar v^{\br{k}}-\widetilde{v}^{(k,\ell+1)}\|  + \|\widetilde{v}^{(k,\ell+1)}-\widehat{v}^{(k,\ell+1)}\|\nonumber\\
&\leq& \|\bar v^{\br{k}}-\widetilde{v}^{(k,\ell+1)}\|  + \theta \|\widetilde{v}^{(k,\ell+1)}-v^{(k,\ell)}\|\nonumber\\
&\leq& \|\bar v^{\br{k}}-\widetilde{v}^{(k,\ell+1)}\|  + \theta \|\widetilde{v}^{(k,\ell+1)}- \bar {v}^{(k)}\|
+\theta\| \bar {v}^{(k)}-v^{(k,\ell)}\|\nonumber\\
&\leq& (1+\theta)\|\bar v^{\br{k}}-\widetilde{v}^{(k,\ell+1)}\|+\theta\| \bar {v}^{(k)}-v^{(k,\ell)}\|\nonumber\\
&\leq& \Big(\theta+(1+\theta)\frac{C_1}{\sqrt{\lambda_1}}\eta (\V^{\br{1}})\Big)\| \bar{v}^{(k)}-v^{(k,\ell)}\|.
\end{eqnarray}
The eigenvalue problem \eqref{Eigen_Augment_Problem} can be seen as
a low dimensional subspace approximation of the eigenvalue problem (\ref{Weak_Eigenvalue_Discrete}).
Using  (\ref{Err_Norm_1_Superclose}), Lemmas \ref{Err_Eigen_Global_Lem}, \ref{Error_Superclose_Lemma}, and their proof,
we obtain that
\begin{eqnarray}\label{Error_u_u_h_2}
\|\bar v^{\br{k}}-v^{(k,\ell+1)}\| &\leq& \frac{1}{1-\bar D_\lambda \eta(\V^{\br{1,k}})}\inf_{w^{\br{1,k}}\in
\V^{\br{1,k}}}\|\bar v^{\br{k}}-w^{\br{1,k}}\|\nonumber\\
&\leq& \frac{1}{1-\bar D_\lambda \eta(\V^{\br{1}})}\|\bar v^{\br{k}}-\widehat{v}^{(k,\ell+1)}\|\nonumber\\
&\leq& \gamma \|\bar v^{\br{k}}-v^{(k,\ell)}\|,
\end{eqnarray}
and
\begin{eqnarray}\label{Error_u_u_h_2_Negative}
\|\bar{\lambda}^{\br{k}}\bar v^{\br{k}}-\lambda^{(k,\ell+1)} v^{(k,\ell+1)}\|_0&\leq&
\bar C_\lambda\eta(\V^{\br{1,k}})\|\bar v^{\br{k}}-v^{(k,\ell+1)}\|\nonumber\\
&\leq& \bar C_\lambda\eta(\V^{\br{1}})\|\bar v^{\br{k}}-v^{(k,\ell+1)}\|.
\end{eqnarray}
Then we have the desired results (\ref{Estimate_h_k_1_a}) and (\ref{Estimate_h_k_1_b})
and conclude the proof.
\end{proof}
\begin{remark}
Definition (\ref{Gamma_Definition}),  Theorem \ref{Convergence_MG_k_Theorem},
Lemmas  \ref{Err_Eigen_Global_Lem} and \ref{Error_Superclose_Lemma} imply that $\gamma$ is less than $1$ when $\eta(\V^{\br{1}})$ is small enough. If $\lambda$ is large or the spectral gap $\delta_\lambda$ is small, then we need to use a smaller $\eta(\V^{\br{1}})$ or $H$. Furthermore, we can
increase the multigrid smoothing steps $m_1$ and $m_2$ to reduce $\theta$ and then $\gamma$. These theoretical restrictions do not limit practical applications where (in numerical implementations), $H$ is simply chosen (just) small enough so that the  number of elements of corresponding coarsest space (just) exceeds  the required number of eigenpairs  ($H$ and the coarsest space are adapted to the  number of eigenpairs to be computed).
\end{remark}


\subsection{Multilevel Method for Eigenvalue Problem}
In this subsection, we introduce the multilevel method
based on the subspace correction method defined in Algorithm \ref{Multigrid_Smoothing_Step} and
the properties of gamblet spaces. This multilevel method can achieve
the same order of accuracy as the direct solve of the eigenvalue problem on the finest (gamblet) space.
The multilevel method is presented in Algorithm \ref{Full_Multigrid}.

\begin{algorithm}[ht]
\caption{Multilevel Correction Scheme}
\begin{enumerate}
\item Define the following eigenvalue problem in $\V^{\br{1}}$:
Find $(\lambda^{\br{1}}, v^{\br{1}})\in \R\times \V^{\br{1}}$ such that $\langle v^{\br{1}},v^{\br{1}}\rangle=1$ and
\begin{equation*}
\langle v^{\br{1}}, w^{\br{1}}\rangle = \lambda^{\br{1}} [v^{\br{1}}, w^{\br{1}}], \quad \forall w^{\br{1}}\in  \V^{\br{1}}.
\end{equation*}
$(\lambda^{\br{1}},v^{\br{1}})\in\R\times \V^{\br{1}}$ is the initial eigenpair approximation.
\item For $k=2,\cdots,q$, do the following iterations
\begin{itemize}
\item Set $\lambda^{(k,0)}=\lambda^{\br{k-1}}$ and $v^{(k,0)} = v^{\br{k-1}}$.
\item Perform the following subspace correction steps for $\ell=0,\cdots, \varpi-1$:
\begin{eqnarray*}
(\lambda^{(k,\ell+1)}, v^{(k,\ell+1)})=
 {\texttt{EigenMG}}(\V^{\br{1}},\lambda^{(k,\ell)},v^{(k,\ell)},\V^{\br{k}}).
\end{eqnarray*}
\item Set $\lambda^{\br{k}}=\lambda^{(k,\varpi)}$ and $v^{\br{k}}=v^{(k,\varpi)}$.
\end{itemize}
End Do
\end{enumerate}
Finally, we obtain an eigenpair approximation
$(\lambda^{\br{q}},v^{\br{q}})\in \R\times \V^{\br{q}}$ in the finest gamblet space.
\label{Full_Multigrid}
\end{algorithm}

\begin{theorem}\label{Error_Full_Multigrid_Theorem}
After implementing Algorithm \ref{Full_Multigrid}, the resulting
eigenpair approximation $(\lambda^{\br{q}},v^{\br{q}})$ has the following error estimates
\begin{eqnarray}
\|\bar v^{\br{q}}-v^{\br{q}}\| &\leq&
2\sum_{k=1}^{q-1} \gamma^{(q-k)\varpi}\delta_{k}(\lambda),\label{FM_Err_fun}\\
\|\bar v^{\br{q}}-v^{\br{q}}\|_0 &\leq& 2\Big(1+\frac{1}{\lambda_1\delta_\lambda}\Big)\eta(\V^{\br{1}})
\|\bar v^{\br{q}} - v^{\br{q}}\|, \label{FM_Err_0_Norm}\\
|\bar{\lambda}^{\br{q}}-\lambda^{\br{q}}| &\leq& \lambda^{\br{q}}\|v^{\br{q}} - \bar v^{\br{q}}\|^2,\label{FM_Err_eigen}
\end{eqnarray}
where $\varpi$ is the number of subspace correction steps in Algorithm \ref{Full_Multigrid}.
\end{theorem}
\begin{proof}
Define $e_{k}:=\bar v^{\br{k}}-v^{\br{k}}$. From step 1
in Algorithm \ref{Full_Multigrid}, it is obvious $e_1=0$.
Then the assumption (\ref{Estimate_h_k_b}) in
Theorem \ref{Error_Estimate_One_Smoothing_Theorem} is satisfied for $k=1$. From the definitions of Algorithms \ref{Multigrid_Smoothing_Step} and \ref{Full_Multigrid},
Theorem \ref{Error_Estimate_One_Smoothing_Theorem} and recursive argument, the assumption (\ref{Estimate_h_k_b})
holds for each level of space $\V^{\br{k}}$ ($k=1,\cdots,q$) with $C_1 = \bar C_\lambda$ in (\ref{Definition_C_Bar}).
Then the convergence rate (\ref{Estimate_h_k_1_a}) is valid for all $k=1, \cdots, q$ and $\ell = 0, \cdots, \varpi-1$.

For $k=2,\cdots,q$, by 
Theorem \ref{Error_Estimate_One_Smoothing_Theorem} and recursive argument,  we have
\begin{eqnarray}\label{FM_Estimate_1}
\|e_k\|&\leq& \gamma^\varpi\|\bar v^{\br{k}}- v^{\br{k-1}}\|\nonumber\\
&\leq& \gamma^\varpi\big(\|\bar v^{\br{k}}-\bar v^{\br{k-1}}\|
+\|\bar v^{\br{k-1}}-v^{\br{k-1}}\|\big)\nonumber\\
&\leq& \gamma^\varpi\big(\|\bar v^{\br{k}}-v\|+\|v-\bar v^{\br{k-1}}\|
+\|\bar v^{\br{k-1}}-v^{\br{k-1}}\|\big)\nonumber\\
&=& \gamma^\varpi \big(\delta_k(\lambda)+\delta_{k-1}(\lambda)+\|e_{k-1}\|\big)\nonumber\\
&\leq& \gamma^\varpi \big(2\delta_{k-1}(\lambda)+\|e_{k-1}\|\big).
\end{eqnarray}

By iterating inequality (\ref{FM_Estimate_1}), the following
inequalities hold
\begin{eqnarray}
\|e_q\|\leq 2\big(\gamma^\varpi\delta_{q-1}(\lambda)+ 
\cdots +\gamma^{(q-1)\varpi}\delta_{1}(\lambda)\big)
\leq 2\sum_{k=1}^{q-1} \gamma^{(q-k)\varpi}\delta_{k}(\lambda).
\end{eqnarray}
which leads to the desired result (\ref{FM_Err_fun}).

From \eqref{Rayleigh_quotient_expansion}, \eqref{Error_2}, \eqref{Error_3} and  \eqref{FM_Err_fun}, we have the following error estimates
\begin{eqnarray*}
\|\bar v^{\br{q}} - v^{\br{q}}\|_0 &\leq& 2\Big(1+\frac{1}{\lambda_1\delta_\lambda}\Big)\eta(\V^{\br{1}})
\|\bar v^{\br{q}} - v^{\br{q}}\|,\nonumber\\
|\bar\lambda^{\br{q}}-\lambda^{\br{q}}| &\leq& \frac{\|v^{\br{q}} - \bar v^{\br{q}}\|^2}{\|v^{\br{q}}\|_0^2}
\leq \lambda^{\br{q}}\|v^{\br{q}} - \bar v^{\br{q}}\|^2,
\end{eqnarray*}
which are the desired results (\ref{FM_Err_0_Norm}) and (\ref{FM_Err_eigen}).

\end{proof}

\begin{remark}
The proof of Theorem \ref{Error_Full_Multigrid_Theorem} implies that the assumption (\ref{Estimate_h_k_b}) in Theorem \ref{Error_Estimate_One_Smoothing_Theorem}
holds for $C_1=\bar C_\lambda$ in each level of space $\V^{\br{k}}$ ($k=1, \cdots, q$).
The structure of Algorithm \ref{Full_Multigrid}, implies that $\bar C_\lambda$ does not change
as the algorithm progresses from the initial space $\V^{\br{1}}$ to the finest one $\V^{\br{q}}$.
\end{remark}

\begin{Corollary}
Let $\gamma$ be the constant in \eqref{Gamma_Definition}. Given the uniform contraction rate $0<\theta<1$ (obtained from Theorem \ref{Convergence_MG_k_Theorem}) and given the bound   $\eta(\V^{\br{1}})\leq C H$ (obtained from Property \ref{property:optimaldecomposition}, which is implied by Theorem \ref{thmgam}) select
 $0<H<1$ small enough so that  $0<\gamma<1$  and then choose the integer $\varpi>1$ to satisfy
\begin{equation}\label{Convergence_Condition}
\frac{\gamma^{\varpi}}{H}<1\,.
\end{equation}
Then the resulting eigenpair approximation $(\lambda^{\br{q}},v^{\br{q}})$
obtained by Algorithm \ref{Full_Multigrid} has  the following error estimates
\begin{eqnarray}\label{FM_Err_fun_Final}
\|v -v^{\br{q}}\| &\leq& C {C_{\lambda}^\prime} \sqrt{\lambda} H^q,\label{FM_Err_fun_Norm_1_Final}\\
\|v-v^{\br{q}}\|_0 &\leq& 2C^2\left( \Big(1+\frac{1}{\lambda_1\delta_\lambda}\Big)\Big(1+H^{q-1}\Big)\right){C_{\lambda}^\prime} H^q,\label{FM_Err_fun_Norm_0_Final}\\
 |\lambda-\lambda^{\br{q}}| &\leq&  \lambda\lambda^{\br{q}} (CC_{\lambda}^\prime)^2H^{2q},\label{FM_Err_Eigenvalue_Final}
\end{eqnarray}
where the constant $C$ comes from Property \ref{property:optimaldecomposition} or Proposition \ref{propCondition_1}
and ${C_{\lambda}^\prime}$ is defined as follows
\begin{eqnarray*}
{C_{\lambda}^\prime} = \left(\sqrt{2\lambda\Big(1+\frac{1}{\lambda_1\delta_\lambda^2}\eta^2(\V^{\br{q}})\Big)}+ 2\frac{1-\Big(\frac{\gamma^\varpi}{H}\Big)^q}{1-\frac{\gamma^\varpi}{H}}\right).
\end{eqnarray*}
\end{Corollary}
\begin{proof}
From Lemma \ref{Error_Superclose_Lemma}, Theorem \ref{Error_Full_Multigrid_Theorem}, 
(\ref{eqCondition_1}), (\ref{Err_Norm_1}) and (\ref{Convergence_Condition}),
we have the following estimates
\begin{eqnarray}\label{Estimate_Final}
&&\|v-v^{\br{q}}\|  \leq \| v - \bar v^{\br{q}}\| + \|\bar v^{\br{q}}-v^{\br{q}}\|\nonumber\\
&\leq& \sqrt{2\Big(1+\frac{1}{\lambda_1\delta_\lambda^2}\eta^2(\V^{\br{q}})\Big)}\delta_q(\lambda) + 2\sum_{k=1}^{q-1} \gamma^{(q-k)\varpi}\delta_{k}(\lambda)\nonumber\\
&\leq& C\sqrt{2\Big(1+\frac{1}{\lambda_1\delta_\lambda^2}\eta^2(\V^{\br{q}})\Big)}\sqrt{\lambda}H^q + 2C\sum_{k=1}^{q-1}\gamma^{(q-k)\varpi}\sqrt{\lambda}H^k\nonumber\\
&\leq& C\sqrt{2\lambda\Big(1+\frac{1}{\lambda_1\delta_\lambda^2}\eta^2(\V^{\br{q}})\Big)}\sqrt{\lambda}H^q+ 2C\sqrt{\lambda}H^q\sum_{k=0}^{q-1} \Big(\frac{\gamma^\varpi}{H}\Big)^k\nonumber\\
&\leq& C\left(\sqrt{2\lambda\Big(1+\frac{1}{\lambda_1\delta_\lambda^2}\eta^2(\V^{\br{q}})\Big)}+ 2\frac{1-\Big(\frac{\gamma^\varpi}{H}\Big)^q}{1-\frac{\gamma^\varpi}{H}}\right)\sqrt{\lambda}H^q.
\end{eqnarray}
This is the desired result (\ref{FM_Err_fun_Norm_1_Final}).

From (\ref{eqCondition_1}), (\ref{Error_2}), \eqref{FM_Err_fun}, \eqref{FM_Err_0_Norm}
and \eqref{Estimate_Final}, $\|v- v^{\br{q}}\|_0$ has the following
estimates  
\begin{eqnarray*}
&&\|v- v^{\br{q}}\|_0 \leq \|v-\bar v^{\br{q}}\|_0 + \|\bar v^{\br{q}}-v^{\br{q}}\|_0\nonumber\\
&\leq& 2\Big(1+\frac{1}{\lambda_1\delta_\lambda}\Big)\eta(\V^{\br{q}})\|v- \bar v^{\br{q}}\|
+  2\Big(1+\frac{1}{\lambda_1\delta_\lambda}\Big)\eta(\V^{\br{1}})
\|\bar v^{\br{q}} - v^{\br{q}}\|\nonumber\\
&\leq& 2C\Big(1+\frac{1}{\lambda_1\delta_\lambda}\Big)\eta(\V^{\br{q}})
\sqrt{2\Big(1+\frac{1}{\lambda_1\delta_\lambda^2}\eta^2(\V^{\br{q}})\Big)}H^q \nonumber\\
&&+ 4C\Big(1+\frac{1}{\lambda_1\delta_\lambda}\Big)\eta(\V^{\br{1}})
\frac{1-\Big(\frac{\gamma^\varpi}{H}\Big)^q}{1-\frac{\gamma^\varpi}{H}}H^q\nonumber\\
&\leq& 2C\left( \Big(1+\frac{1}{\lambda_1\delta_\lambda}\Big)\Big(1+H^{q-1}\Big)\right){C_{\lambda}^\prime} H^q.
\end{eqnarray*}
From (\ref{Rayleigh_quotient_expansion})
and (\ref{FM_Err_fun_Norm_1_Final}), the error estimate for $|\lambda-\lambda^{\br{q}}|$ can be deduced
as follows
\begin{eqnarray*}
|\lambda-\lambda^{\br{q}}| &\leq& \frac{\|v^{\br{q}} - v\|^2}{\|v^{\br{q}}\|_0^2}
\leq \lambda^{\br{q}}\|v^{\br{q}} -v\|^2
\leq \lambda\lambda^{\br{q}} C^2{C_{\lambda}^\prime}^2H^{2q}.
\end{eqnarray*}
Then  the desired results (\ref{FM_Err_fun_Norm_0_Final}) and (\ref{FM_Err_Eigenvalue_Final}) is obtained
and the proof is complete.

\end{proof}

\begin{remark}
The main computational work of  Algorithm \ref{Multigrid_Smoothing_Step} is to solve the
linear equation (\ref{aux_problem}) by the multigrid method defined in Algorithm \ref{alg:multigrid}. Therefore
Remark \ref{Computational_Work_Proposition} implies the bound $\mathcal O(N(\log(\frac{N}{\varepsilon}))^{2d+1}\log(\varepsilon)/\log(\gamma))$ on the number of operations required to achieve accuracy $\ve$
(see \cite{OwhadiMultigrid:2017,  OwhadiScovel:2017,Owhadi2017a, SchaeferSullivanOwhadi17, OwhScobook2018}).
\end{remark}

\section{Numerical Results}
\label{sec:numerics}
In this section, numerical examples are presented to illustrate the
efficiency of the Gamblet based multilevel correction method for benchmark multiscale eigenvalue problems. 
Furthermore, we will show that the Gamblets can also be used as efficient preconditioner for state-of-the-art eigensolvers such as LOBPCG method. 

\subsection{SPE10}
\label{sec:spe10}

In the first example, we solve the eigenvalue problem (\ref{eqn:scalar}) on  $\Omega = [-1,1]\times [-1,1]$, and the coefficient matrix $a(x)$ is taken from the data of the SPE10 benchmark \\
(\url{http://www.spe.org/web/csp/}). The contrast of $a(x)$ is $\lambda_{\max}(a)/\lambda_{\min}(a) \simeq 1\cdot 10^6$.

The fine mesh $\T_h$ is a regular square mesh with mesh size $h=2(1+2^{q})^{-1}$ and $128\times 128$ interior nodes. At the finest level, we use continuous bilinear nodal basis elements $\varphi_i$ spanned by $\displaystyle\{1,x_1,x_2,x_1 x_2\}$ in each element of $\T_h$. $a(x)$ is piecewise constant over $\T_h$ as illustrated in Figure \ref{fig:spe10}. The measurement function is chosen as in Example \ref{example3}. For the gamblet decomposition, we choose $H = 1/2$, $q=7$. The pre-wavelets $\psi$ and the gamblet decomposition of the solution $u$ for the elliptic equation $-\diiv a(x) \nabla u = \sin(\pi x)\sin(\pi y)$ are shown in Figure \ref{fig:basisa8} and Figure \ref{fig:decompa8}, respectively.

\vspace{-10pt}
\begin{figure}[H]
\centering
\includegraphics[scale = 0.3]{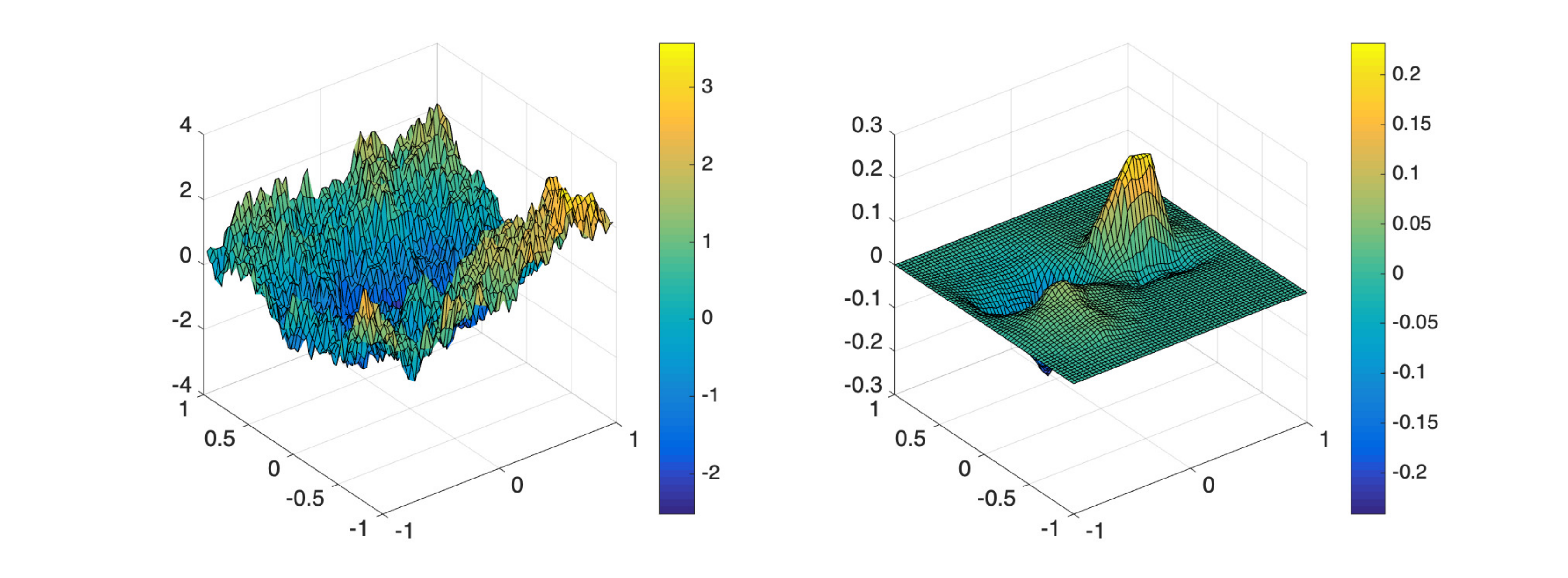}
\caption{Left: coefficient $a(x)$ from SPE10 benchmark, in $\log_{10}$ scale; Right: solution $u$ for the elliptic equation $-\diiv a(x) \nabla u = \sin(\pi x) \sin(\pi y)$.}
\label{fig:spe10}
\end{figure}

\begin{figure}[H]
\centering
\includegraphics[scale = 0.3]{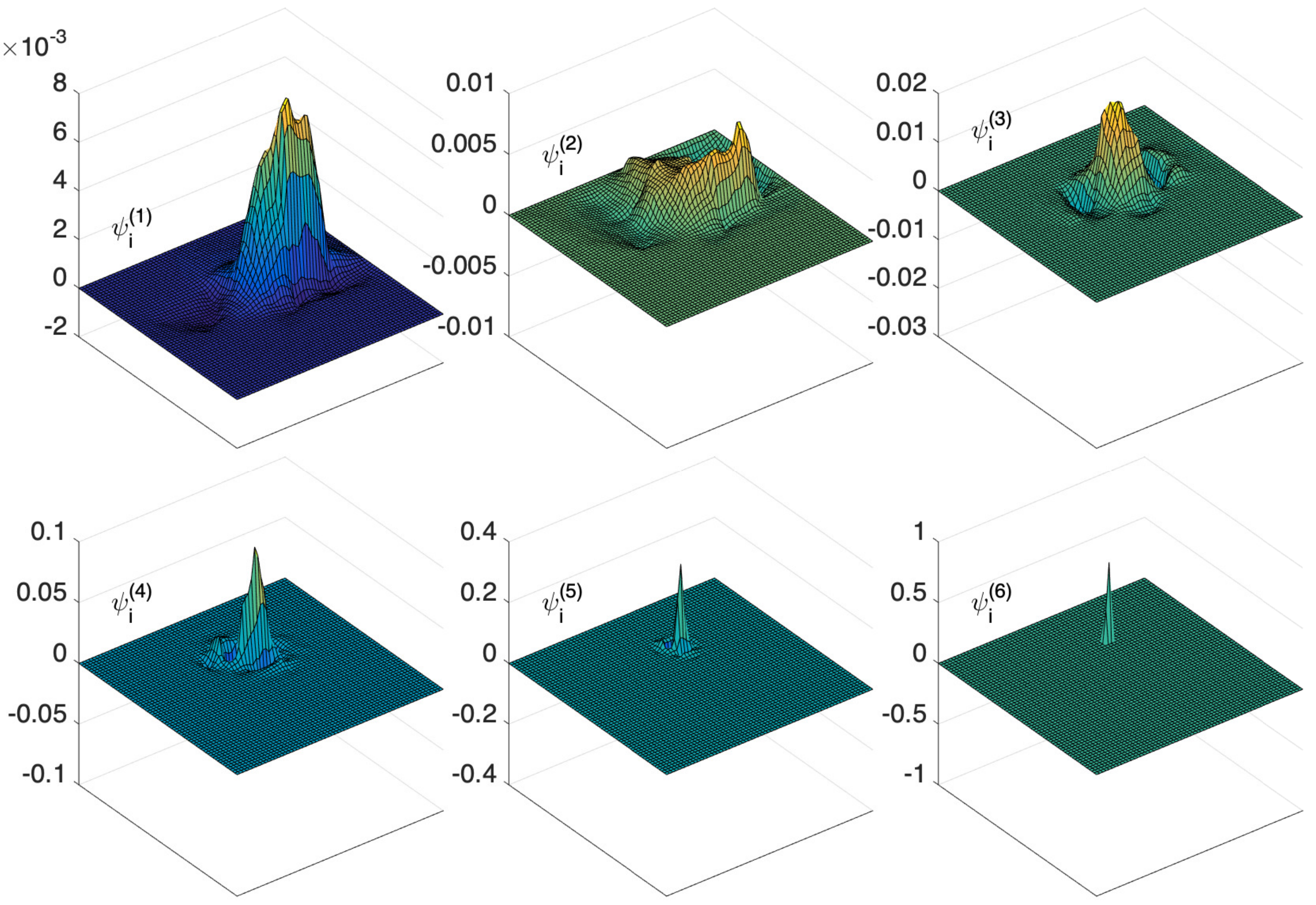}
\caption{Pre-wavelets $\psi$ at different scales.}
\label{fig:basisa8}
\end{figure}

\begin{figure}[H]
\centering
\includegraphics[scale = 0.3]{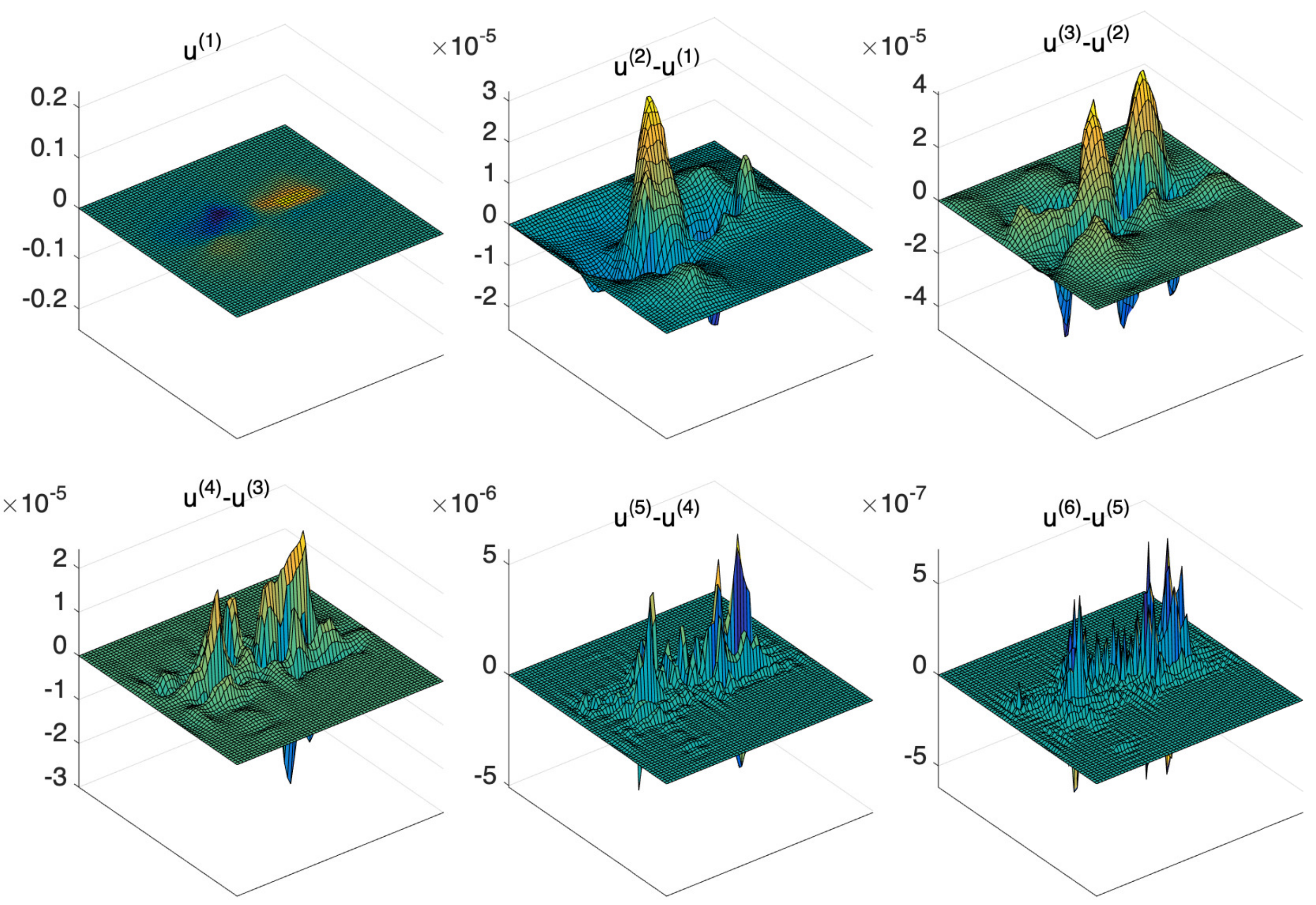}
\caption{Solution for the elliptic equation with $f = \sin(\pi x) \sin(\pi y)$.}
\label{fig:decompa8}
\end{figure}

We calculate the first $12$ eigenvalues using the multilevel correction method in Algorithm \ref{Full_Multigrid}, therefore we actually take $\Vbr{2}$ as the coarsest subspace and the effective mesh size is $H^2 = 1/4$. We choose parameters $m_1=m_2=2$ and $p=1$ in the multigrid iteration step defined in Algorithm \ref{alg:multigrid} to solve the linear equation (\ref{aux_problem}), and use Gauss-Seidel as the smoother.

We compare the gamblet based multilevel correction method with  geometric multigrid multilevel correction method.
In Table \ref{Table_MG_Exam_1}, we show the numerical results for the first 12 eigenvalues, here we take the number of subspace correction steps $\varpi=1$ for $k = 3, \dots, q$. For comparison, we also show the corresponding numerical results in Table \ref{Table_GMG_Exam_1} with the standard geometric multigrid linear solver. We observe much faster convergence for the gamblet based multilevel correction method ($10^6$ smaller for the first eigenvalue).

\begin{table}[H]
\centering
\caption{Relative errors $|(\lambda_i^{(k)}-\lambda_i)/\lambda_i|$ for the gamblet based multilevel correction method, first a few iterations on the coarser levels}\label{Table_MG_Exam_1}.
\begin{tabular}{|c|c|c|c|c|c|c|c|}
\hline
i & k = 2 & k = 3 & k = 4 & k=5 & k = 6 & k = 7 \\
\hline
1   & 6.1568e-2 & 1.3356e-2  & 3.0902e-3 & 1.2586e-3 & 3.8293e-4 & 1.5586e-8  \\
2   & 1.6827e-1 & 3.0270e-2  & 4.6347e-3 & 1.0656e-3 & 2.4616e-4 & 5.3456e-8  \\
3   & 7.9106e-1 & 1.1814e-1  & 2.3155e-2 & 2.8431e-3 & 2.9124e-4 & 4.7883e-6  \\
4   & 5.8274e-1 & 1.9203e-1  & 4.5203e-2 & 7.7621e-3 & 7.7980e-4 & 4.4444e-5  \\
5   & 7.5657e-1 & 1.6533e-1  & 1.6978e-2 & 2.8863e-3 & 3.3941e-4 & 1.0250e-5  \\
6   & 9.4417e-1 & 2.9132e-1  & 5.0443e-2 & 7.0754e-3 & 7.7061e-4 & 4.4771e-5  \\
7   & 1.7033e0  & 2.8337e-1  & 8.1393e-2 & 2.4187e-2 & 4.6014e-3 & 7.2897e-4  \\
8   & 2.4517e0  & 5.0598e-1  & 1.3164e-1 & 2.4945e-2 & 4.6447e-3 & 8.7663e-4  \\
9   & 6.4576e0  & 6.6654e-1  & 2.6205e-1 & 9.9177e-2 & 1.6962e-2 & 3.3086e-3  \\
10 & 6.9955e0  & 6.8507e-1  & 2.4108e-1 & 4.7575e-2 & 1.9529e-2 & 9.6051e-3  \\
11 & 1.0927e1  & 8.6987e-1  & 2.6043e-1 & 7.5851e-2 & 1.9996e-2 & 8.3358e-3  \\
12 & 1.3665e1  & 9.5975e-1  & 3.3355e-1 & 5.9182e-2 & 1.9377e-2 & 7.5015e-3  \\
\hline
\end{tabular}
\end{table}

\begin{Remark}
It is shown in \cite{Malqvist2014a} that for approximate eigenvalues with respect to the LOD coarse spaces on scale $H$, a post-processing step can improve the eigenvalue error from $H^4$ to $H^6$. The post-processing step is a correction with exact solve on the finest level. Since we are using an approximate solve in the correction step, this corresponds to the multilevel correction scheme with one correction step on each level, which is shown in Table \ref{Table_MG_Exam_1}. Comparing Table \ref{Table_MG_Exam_1} with Table 2 in \cite{Malqvist2014a} shows a similar improvement of accuracy at the finer levels (although the coefficients $a(x)$ are not the same, we expect a similar behavior for the approximation errors of eigenvalues). However, with geometric multigrid, the error reduction is very slow, which is shown by Table \ref{Table_GMG_Exam_1}.
\end{Remark}

\begin{table}[ht]
\centering
\caption{Relative errors $|(\lambda_i^{(k)}-\lambda_i)/\lambda_i|$ for multilevel correction with geometric multigrid, first a few iterations on the coarser levels}\label{Table_GMG_Exam_1}
\begin{tabular}{|c|c|c|c|c|c|c|c|}
\hline
i & k = 2 & k = 3 & k = 4 & k=5 & k = 6 & k = 7 \\
\hline
1   & 2.6912e0   & 2.6698e0   & 2.5627e0 & 2.0948e0 & 4.4351e-1 & 5.0859e-2  \\
2   & 2.4310e0   & 2.3886e0   & 2.3037e0 & 1.8812e0 & 4.8645e-1 & 5.4931e-2  \\
3   & 2.3129e0   & 2.2749e0   & 2.1802e0 & 1.8076e0 & 4.9837e-1 & 6.4541e-2  \\
4   & 2.6706e0   & 2.6225e0   & 2.5193e0 & 2.0636e0 & 5.8780e-1 & 9.2958e-2  \\
5   & 3.1593e0   & 2.9673e0   & 2.8141e0 & 2.2948e0 & 6.2242e-1 & 9.8928e-2  \\
6   & 2.7198e0   & 2.5764e0   & 2.4233e0 & 1.9427e0 & 5.3022e-1 & 7.5071e-2  \\
7   & 2.9581e0   & 2.8158e0   & 2.6886e0 & 2.2162e0 & 6.1367e-1 & 9.9160e-2  \\
8   & 2.9712e0   & 2.8012e0   & 2.6446e0 & 2.1981e0 & 6.4002e-1 & 9.3180e-2  \\
9   & 3.7158e0   & 3.2765e0   & 3.0548e0 & 2.4382e0 & 6.8837e-1 & 1.1892e-1  \\
10 & 3.1307e0   & 2.7671e0   & 2.5808e0 & 2.0749e0 & 5.9963e-1 & 8.4462e-2  \\
11 & 3.0937e0   & 2.8429e0   & 2.6748e0 & 2.1673e0 & 5.7858e-1 & 8.8655e-2  \\
12 & 3.1317e0   & 2.7967e0   & 2.6259e0 & 2.1068e0 & 5.8031e-1 & 8.6055e-2  \\
\hline
\end{tabular}

\end{table}

If higher accuracy is pursued, we can take more correction steps at the finest level $k=q$. See Figure \ref{fig:historya8} for the convergence history of both the gamblet based method and the geometric multigrid based method up to $10^{-14}$. The gamblet based method converges much faster than the geometric multigrid based method.

\begin{figure}[H]
\centering
\includegraphics[scale = 0.3]{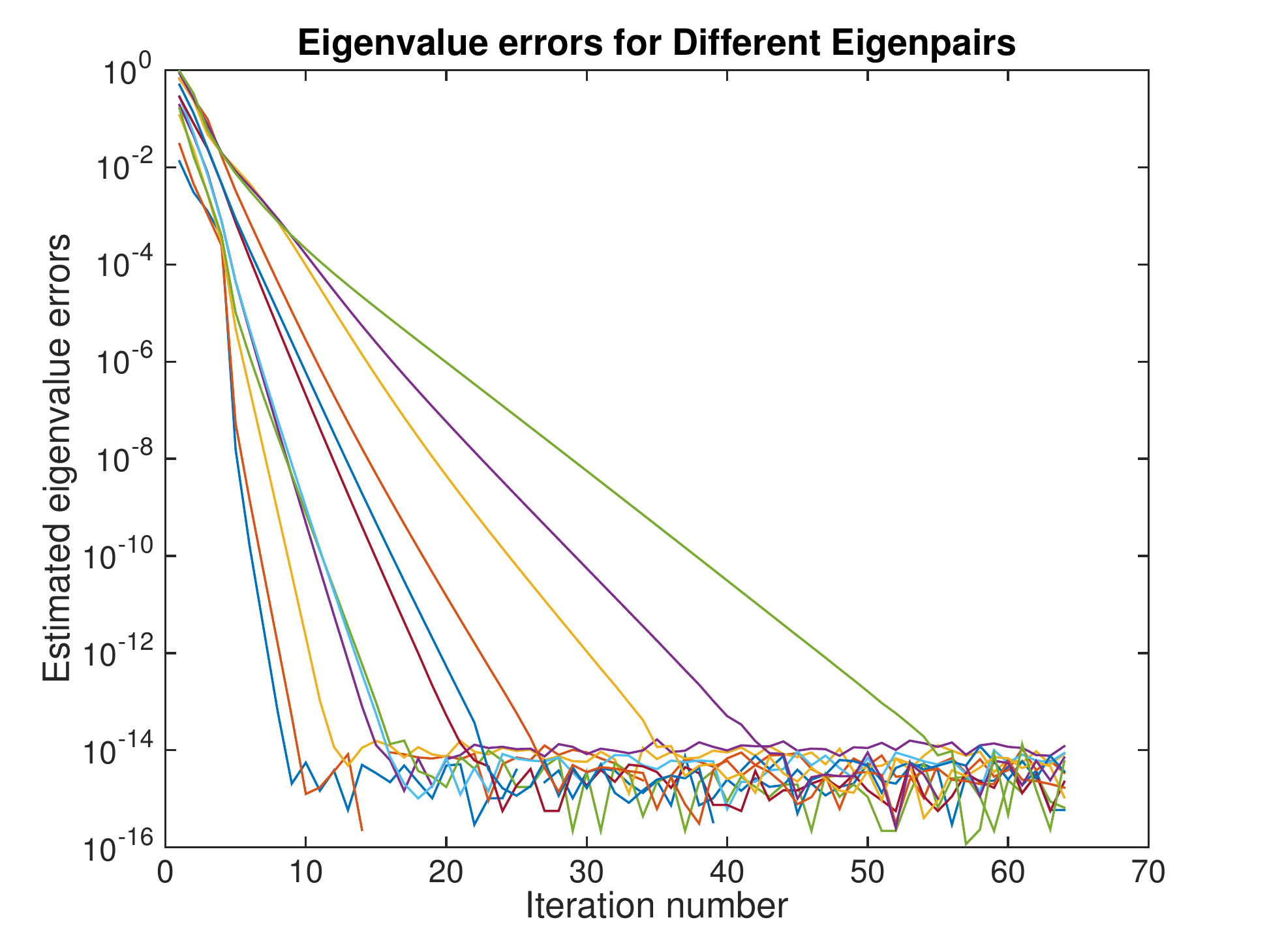}
\includegraphics[scale = 0.3]{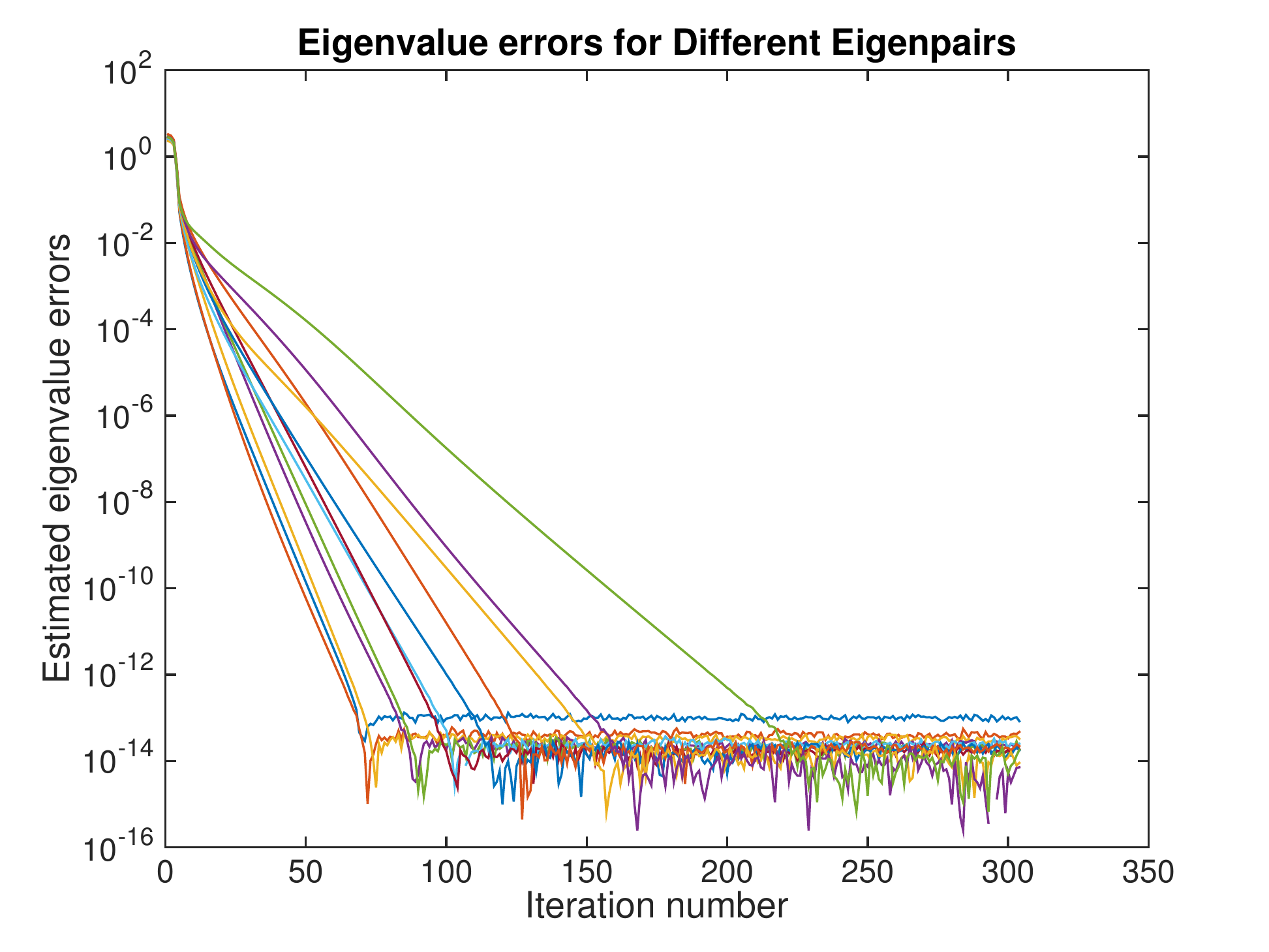}
\caption{Convergence history for first 12 eigenvalues. Left: Gamblet based multilevel method Right: Geometric mutligrid based multilevel method. The iteration number corresponds to the number of correction steps, namely, the outer iteration number. The first a few iterations are on the coarse levels $k = 3, \dots, q-1$, and the following iterations are on the finest level $k = q$.  }
\label{fig:historya8}
\end{figure}

\begin{figure}[H]
\centering
\includegraphics[scale = 0.3]{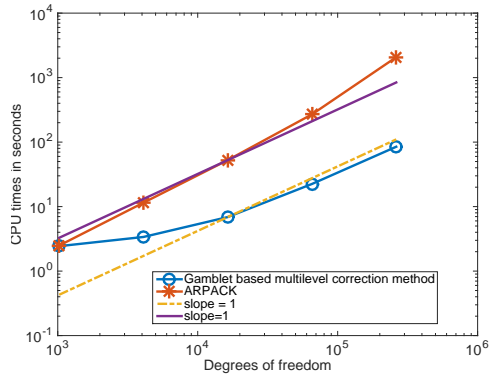}
\caption{CPU time for Gamblet based multilevel correction and ARPACK}
\label{fig:cputime}
\end{figure}

\revise{We now compare the efficiency of the multilevel correction method with the benchmark solver ARPACK (\url{https://www.caam.rice.edu/software/ARPACK/}). We implement the multilevel correction method in C (with a precomputed Gamblet decomposition), and run the code on a machine with two 6-core dual thread Intel Xeon E5-2620 2.00GHz CPUs with 72G memory. We solve for 12 eigenvalues, and stop the multilevel correction method when  relative errors for all eigenvalues are below $10^{-9}$. For comparison, we use the ARPACK library to solve the same eigenvalue problems, and use the geometric multigrid method to solve the corresponding linear systems. The results in Figure \ref{fig:cputime} show that the Gamblet based multilevel correction method achieves a ten-fold acceleration in terms of CPU time. We only plot the ``online" computing time for eigenpairs in Figure \ref{fig:cputime}, the ``offline" precomputing time for the Gamblet decomposition is not included since we only have a Matlab implementation for this part. For the Matlab implementation of the multilevel correction method,  the running time for the ``online" and ``offline" parts are usually proportional, and  the a priori theoretical bound on the complexity of the Gamblets precomputation is $\mathcal{O}(N \ln^{2d+1} N)$.}



\subsection{Random Checkerboard}
In the second example, we consider the eigenvalue problem for the random checkerboard case. Here,  $\Omega = [-1, 1]\times [-1, 1]$
and the matrix $a(x)$ is a realization of random coefficients taking values $20$ or $1/20$ with probability $1/2$
at small scale $\varepsilon=1/64$, see Figure \ref{fig:randomcheckerboard}. The coefficient $a(x)$ has contrast $4\times 10^2$, and is highly oscillatory.

We calculate the first $12$ eigenvalues. The parameters for Algorithm \ref{Full_Multigrid} are $H = 1/2$, $q=7$, and we take $\Vbr{2}$ as the coarsest subspace. We choose $m_1=m_2=2$ and $p=1$, and use Gauss-Seidel as the smoother in Algorithm \ref{alg:multigrid}. We take the number of subspace correction steps $\varpi=1$ for $k = 3, \dots, q-1$, then we run the subspace correction at the finest level $k=q$ until convergence.

\begin{figure}[ht]
\centering
\includegraphics[width=5cm]{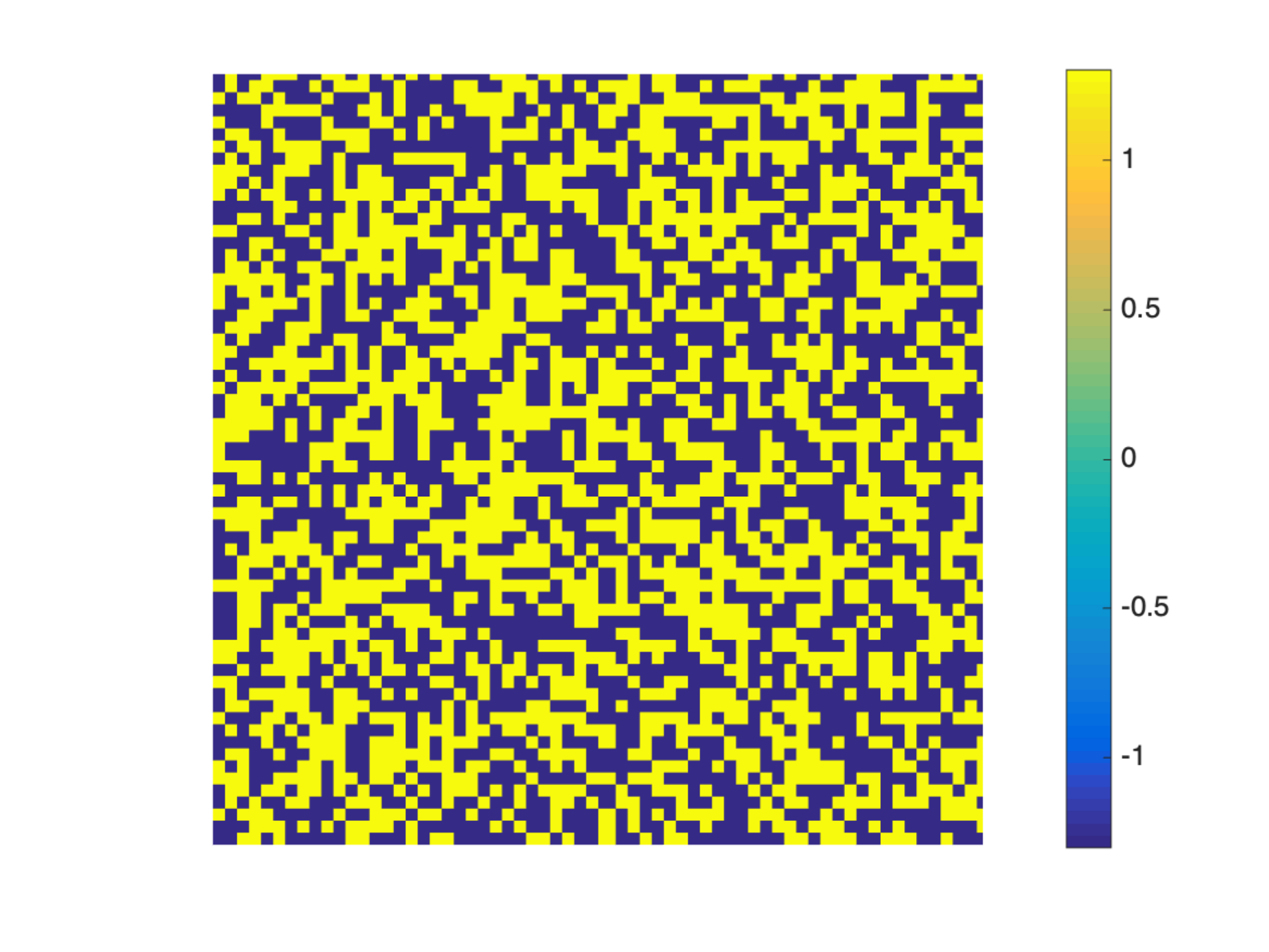}
\caption{Random Checkerboard coefficient, in $\log_{10}$ scale}
\label{fig:randomcheckerboard}
\end{figure}
	The convergence rates  shown in Figure \ref{fig:historya7} suggest a ten fold acceleration \revise{in terms of iteration number} when comparing the gamblet and based multilevel correction method to the geometric multigrid  based multilevel correction method. While it  takes more than 800 iterations for geometric multigrid based multilevel correction method to converge for the first 12 eigenvalues to converge to accuracy $10^{-14}$, the
gamblet based multilevel correction method converges to that accuracy within 70 outer iterations.

\begin{figure}[H]
\centering
\includegraphics[scale = 0.35]{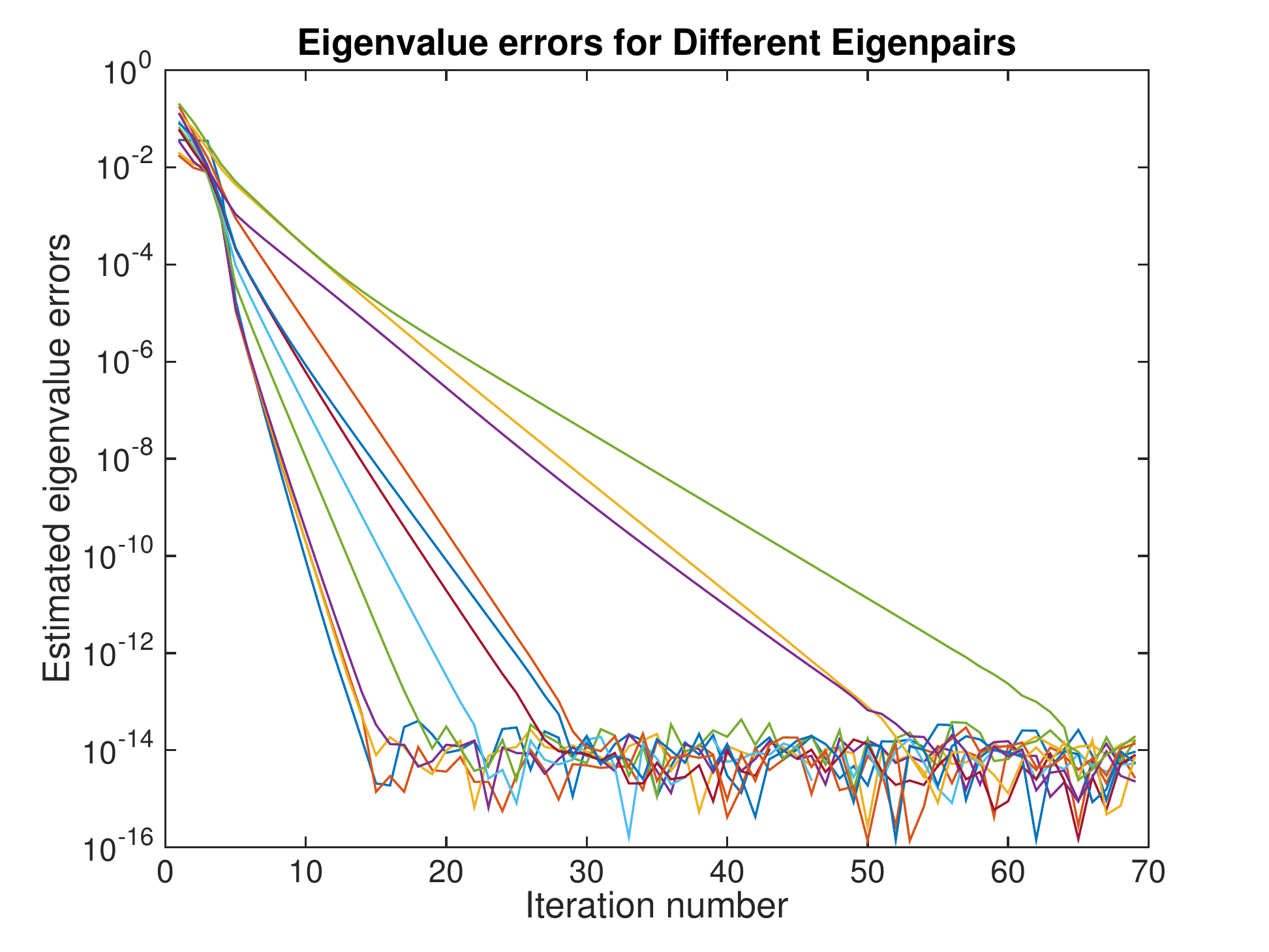}
\includegraphics[scale = 0.35]{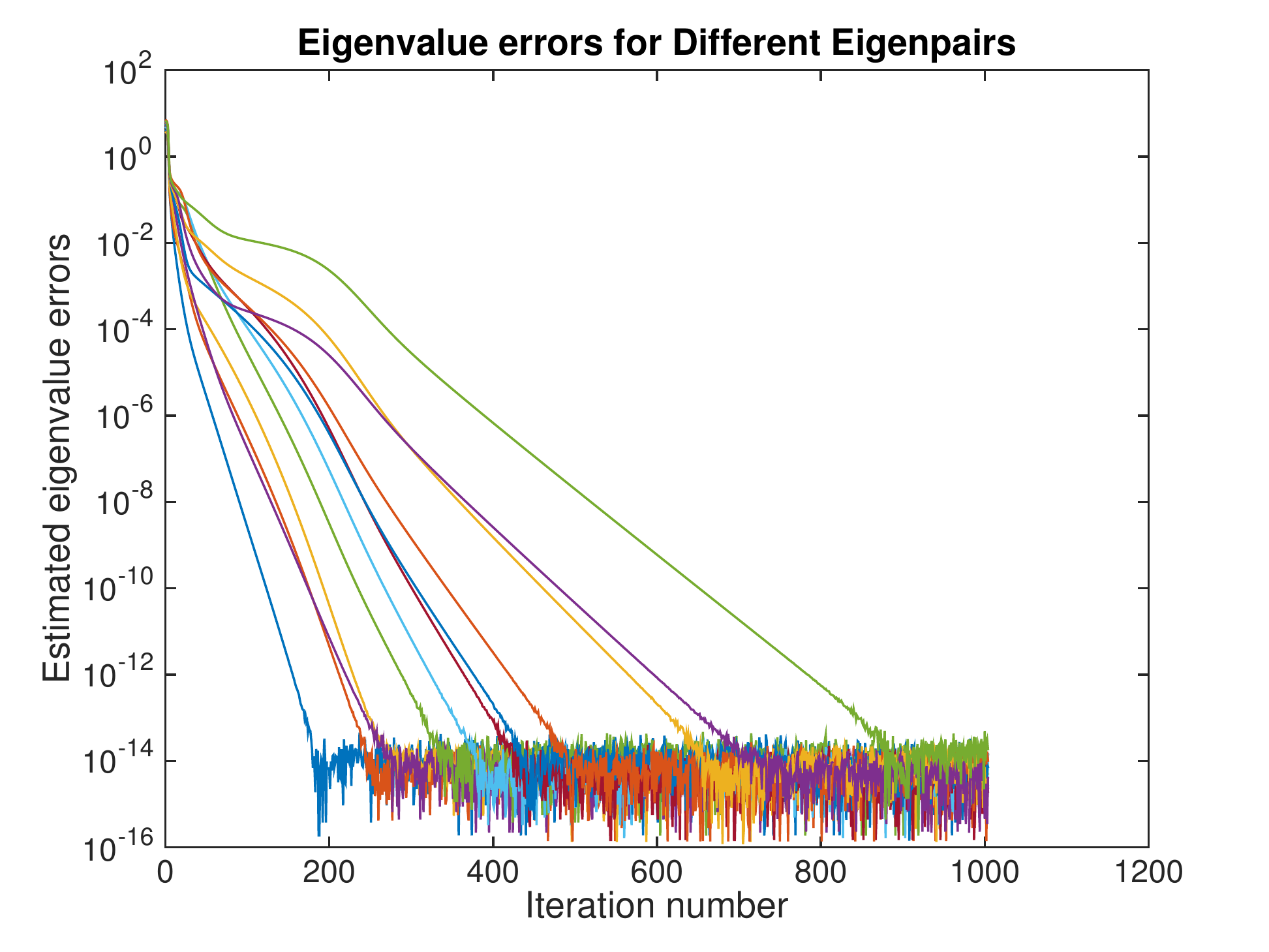}
\caption{Convergence history for first 12 eigenvalues. Left: Gamblet based method Right: Geometric mutligrid based method. The iteration number corresponds to the number of correction steps, namely, the outer iteration number. The first a few iterations are on the coarse level $k = 3, \dots, q-1$, and the following iterations are on the finest level $k = q$.  }
\label{fig:historya7}
\end{figure}

\subsection{\revise{Gamblet Preconditioned LOBPCG Method}}

\revise{In the previous sections, we have proposed the Gamblet based multilevel correction scheme, proved its convergence and numerically demonstrated its performance. In this section, we will show that Gamblets can also be used as an efficient preconditioner for existing eigensolvers. To be precise, we construct the Gamblet based preconditioner for the Locally Optimal Block Preconditioned Conjugate Gradient (LOBPCG) method \cite{Kynazev:2001,Knyazev:2003}, which is a class of widely used eigensolvers}.

A variety of Krylov subspace-based methods are designed to solve a few extreme eigenvalues of symmetric positive matrix \cite{Sorensen:1997,Dyakonov:1980,Bramble:1996,Knyazev:1998,Kynazev:2001,Knyazev:2003,Bai:2000}. Many studies have shown that LOBPCG is one of the most effective method at this task \cite{Knyazev:2017,Duersch:2018} and there are various recent developments of LOBPCG for indefinite eigenvalue problems \cite{Bai:2012}, nonlinear eigenvalue problems \cite{Szyld:2016}, electronic structure calculation \cite{Vecharynski:2015}, and tensor decomposition \cite{Rakhuba:2016}.
The main advantages of LOBPCG are that the costs per iteration and the memory use are competitive with those of the Lanczos method,
linear convergence is theoretically guaranteed and practically observed, it allows utilizing highly efficient matrix-matrix operations, e.g., BLAS 3, and it can directly take advantage of preconditioning, in contrast to the Lanczos method.

LOBPCG can be seen as a generalization of the Preconditioned Inverse Iteration (PINVIT) method \cite{Dyakonov:1980,Bramble:1996,Knyazev:1998}.
The PINVIT method \cite{Dyakonov:1980,Bramble:1996,Knyazev:1998,Kynazev:2001,Knyazev:2003}, can be motivated as an inexact Newton-method for the minimization of the Rayleigh quotient. The Rayleigh quotient $\mu(x)$ for a vector $x$ and a symmetric, positive definite matrix $M$ is defined by
\begin{displaymath}
	\mu(x,M):=\mu(x)= \frac{x^TMx}{x^Tx}	
\end{displaymath}
The global minimum of $\mu(x)$ is achieved at $x = v_1$, with $\lambda_1= \mu(x)$, where $(\lambda_1, v_1)$ is the eigenvalue pair of $M$ corresponding to the smallest eigenvalue $\lambda_1$. This means that minimizing the Rayleigh quotient is equal to computing the smallest eigenvalue. With the following inexact Newton method:
\begin{align*}
	w_{i} &= B^{-1}(Mx_i-\mu(x_i)x_i), \\
	x_{i+1}  &= x_{i}-w_{i}.
\end{align*}
we get the  preconditioned inverse iteration (PINVIT). The preconditioner
$B$ for $M$ have to satisfy $\|I-B^{-1}M\|_{􏱼􏱼M} \leq c <1$. The inexact Newton method can be relaxed by adding a step size $\alpha$
\begin{displaymath}
	x_{i+1} =x_{i}-\alpha_i w_i,
\end{displaymath}
Finding the optimal step size $\alpha_i$ is equivalent to solving the a small eigenvalue problem with respect to $M$ in the subspace $\{x_i, w_i\}$. In \cite{Kynazev:2001} Knyazev used the optimal vector in the subspace $\{x_{i-1}, w_i, x_i \}$ as the next iterate. The resulting method is called locally optimal (block) preconditioned conjugate method (LOBPCG).

In the following comparison, we adopt the Matlab implementation of LOBPCG by Knyazev \cite{Knyazev:2015}. We use the gamblet based multigrid  as a preconditioner in the LOBPCG method, and compare its performance for SPE10 example with \revise{geometric multigrid  preconditioned CG (GMGCG)} and general purpose ILU based preconditioner, the results are shown in Figure \ref{fig:lobpcg}. It is clear that the gamblet preconditioned LOBPCG as well as the gamblet multilevel correction scheme (see Figure \ref{fig:lobpcg}) have better performance than the GMGCG or ILU preconditioned LOBPCG \revise{in terms of iteration number. The Gamblet based LOBPCG converges with the accuracy (residuals) of about $10^{-15}$, with 56 iterations in about 30 seconds (in addition, the precomputation of the Gamblets costs about 18 seconds). While the GMG preconditioned LOBPCG in Figure \ref{fig:lobpcg} fails to converge in 1000 iterations, and the residuals are above $10^{-5}$ when it is stopped at 1000 iterations in about 60 seconds. Although our implementation in Matlab is not optimized in terms of speed, the above observations indicate that the Gamblet preconditioned LOBPCG has potential to achieve even better performance with an optimized implementation.}

\begin{Remark}
The LOBPCG method has a larger subspace for the small Rayleigh-Ritz eigenvalue problem, compared with the multilevel correction scheme in \eqref{Eigen_Augment_Problem}. This could be the reason why the gamblet preconditinoed LOBPCG scheme has fewer (but comparable) outer iterations compared with the multilevel correction scheme shown in Figure \ref{fig:historya8}.  \revise{On the other hand, orthogonalization is crucial for a robust implementation of LOBPCG, and adaptive stopping criteria needs to be used for efficiency. Delicate strategies \cite{Duersch:2018} are proposed in order to ensure the robustness of LOBPCG. Comparing with LOBPCG, the Gamblet based multilevel correction scheme appears to be very robust in our numerical experiments: we only solve an eigenvalue problem at the coarsest level, and still achieve an accuracy of $10^{-14}$ without using any adaptive stopping criteria, for example, see Figure \ref{fig:historya8}}.
\end{Remark}

\begin{figure}[H]
    \centering
    \begin{subfigure}[b]{0.45\textwidth}
        \centering
        \includegraphics[width=\textwidth]{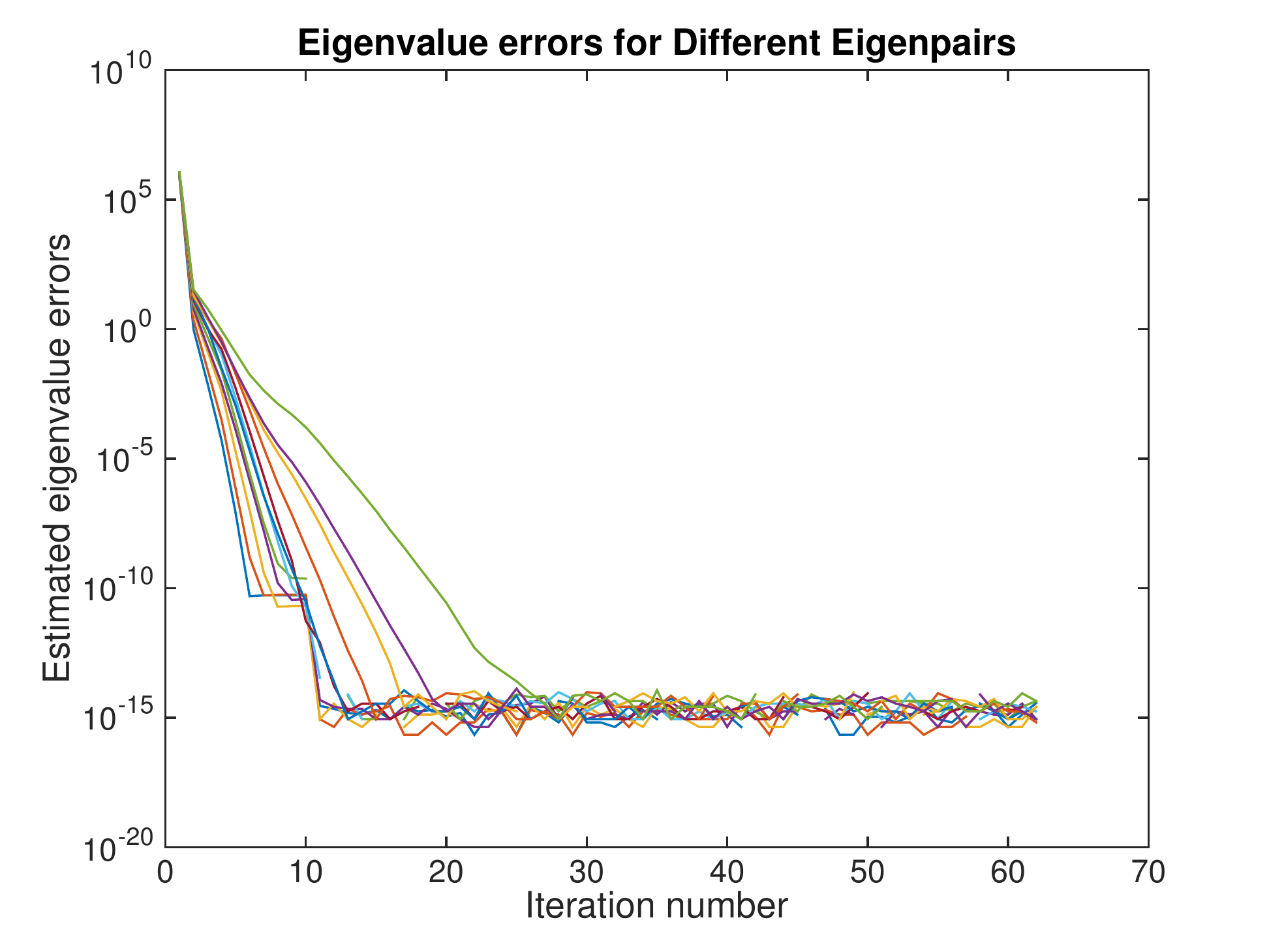}
        \caption{{\footnotesize Eigenvalue errors for the gamblet preconditioned LOBPCG;}}
        \label{fig:eigerr_gamblet_lobpcg}
    \end{subfigure}
    \hfill
    \begin{subfigure}[b]{0.45\textwidth}
        \centering
        \includegraphics[width=\textwidth]{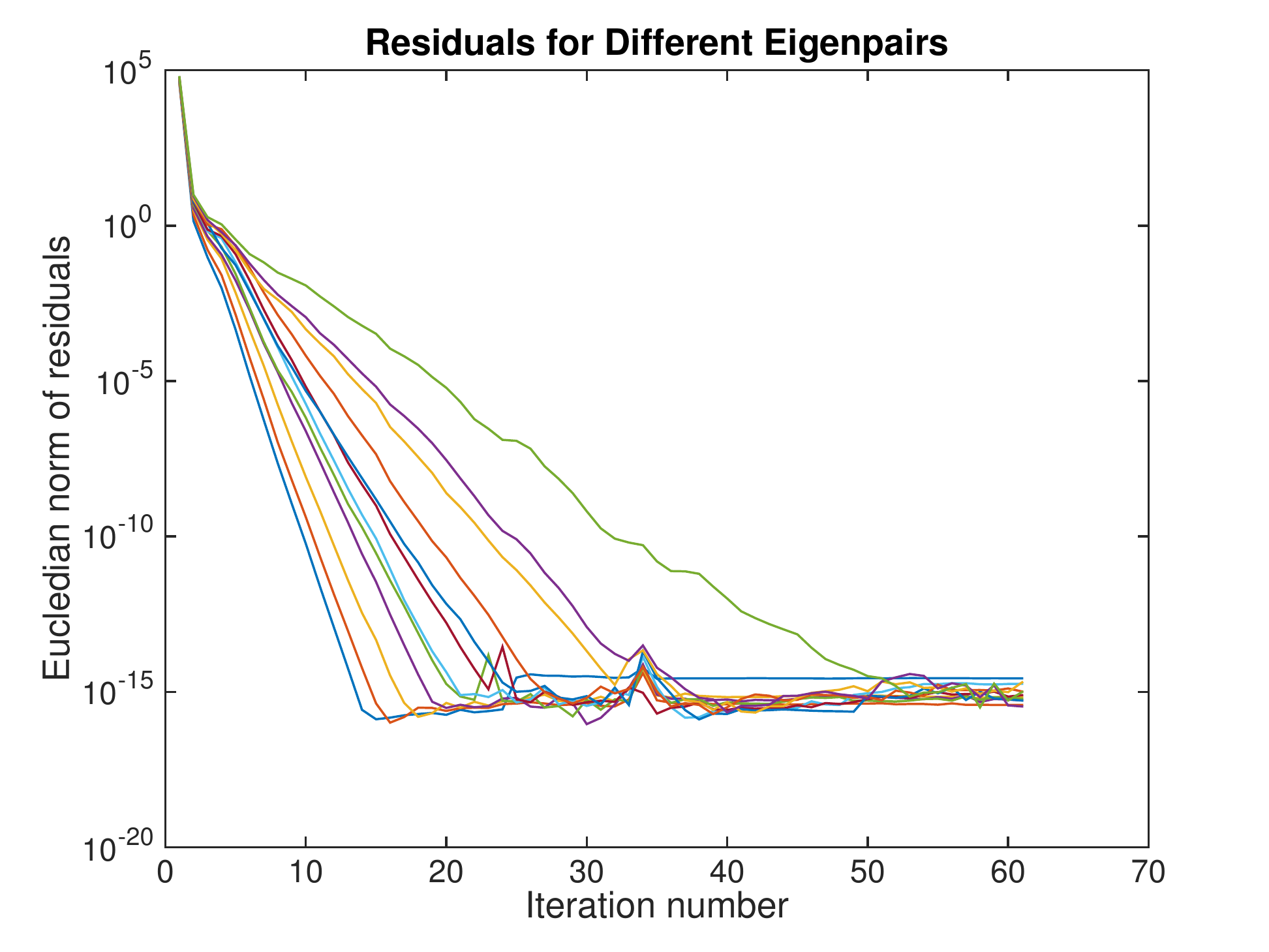}
        \caption{{\footnotesize Residuals for the gamblet preconditioned LOBPCG;}}
        \label{fig:residual_gamblet_lobpcg}
    \end{subfigure}
    \vskip\baselineskip
    \begin{subfigure}[b]{0.45\textwidth}
        \centering
        \includegraphics[width=\textwidth]{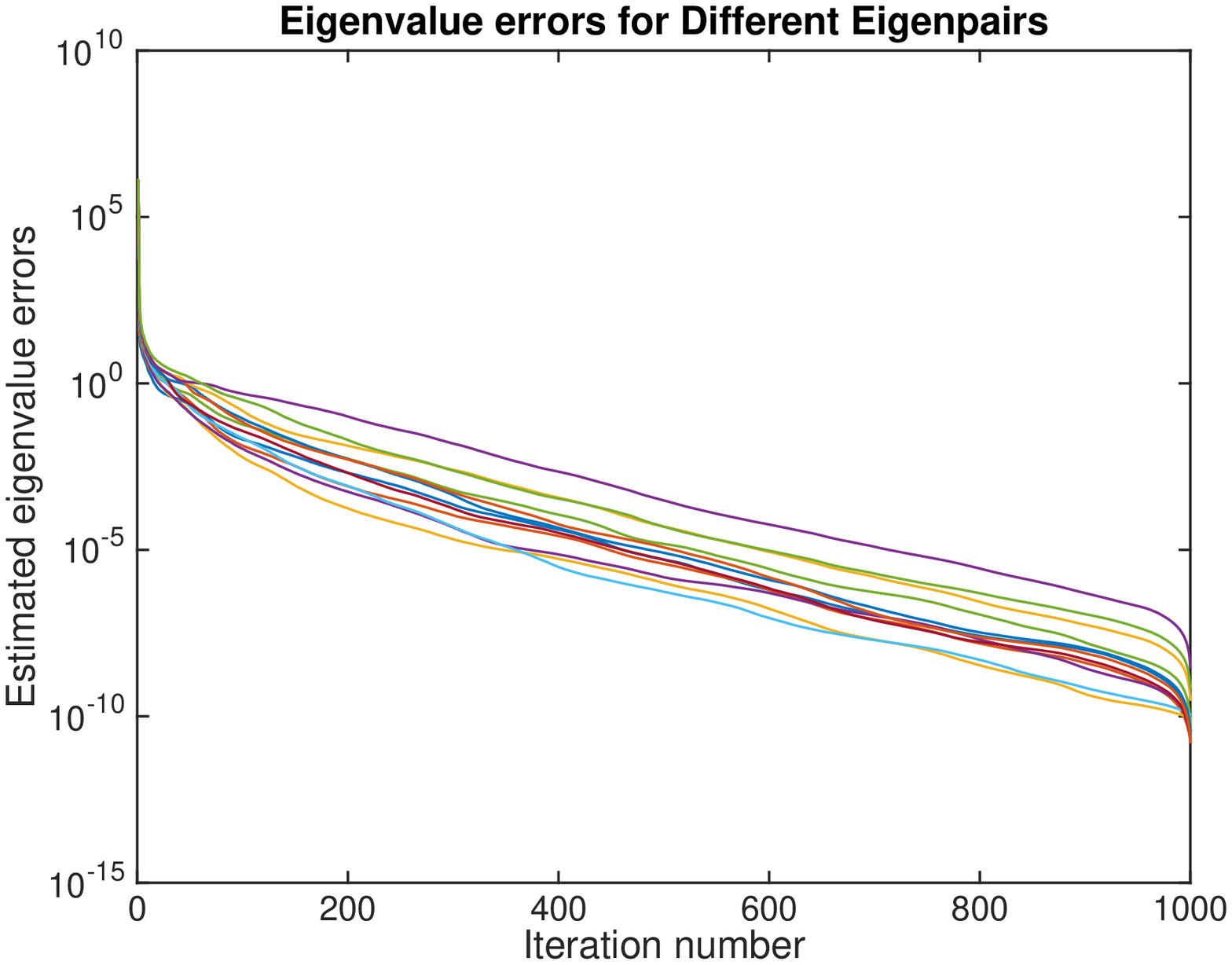}
        \caption{{\footnotesize Eigenvalue errors for the GMGCG preconditioned LOBPCG;}}
        \label{fig:eigerr_ilu_lobpcg}
    \end{subfigure}
    \quad
    \begin{subfigure}[b]{0.45\textwidth}
        \centering
        \includegraphics[width=\textwidth]{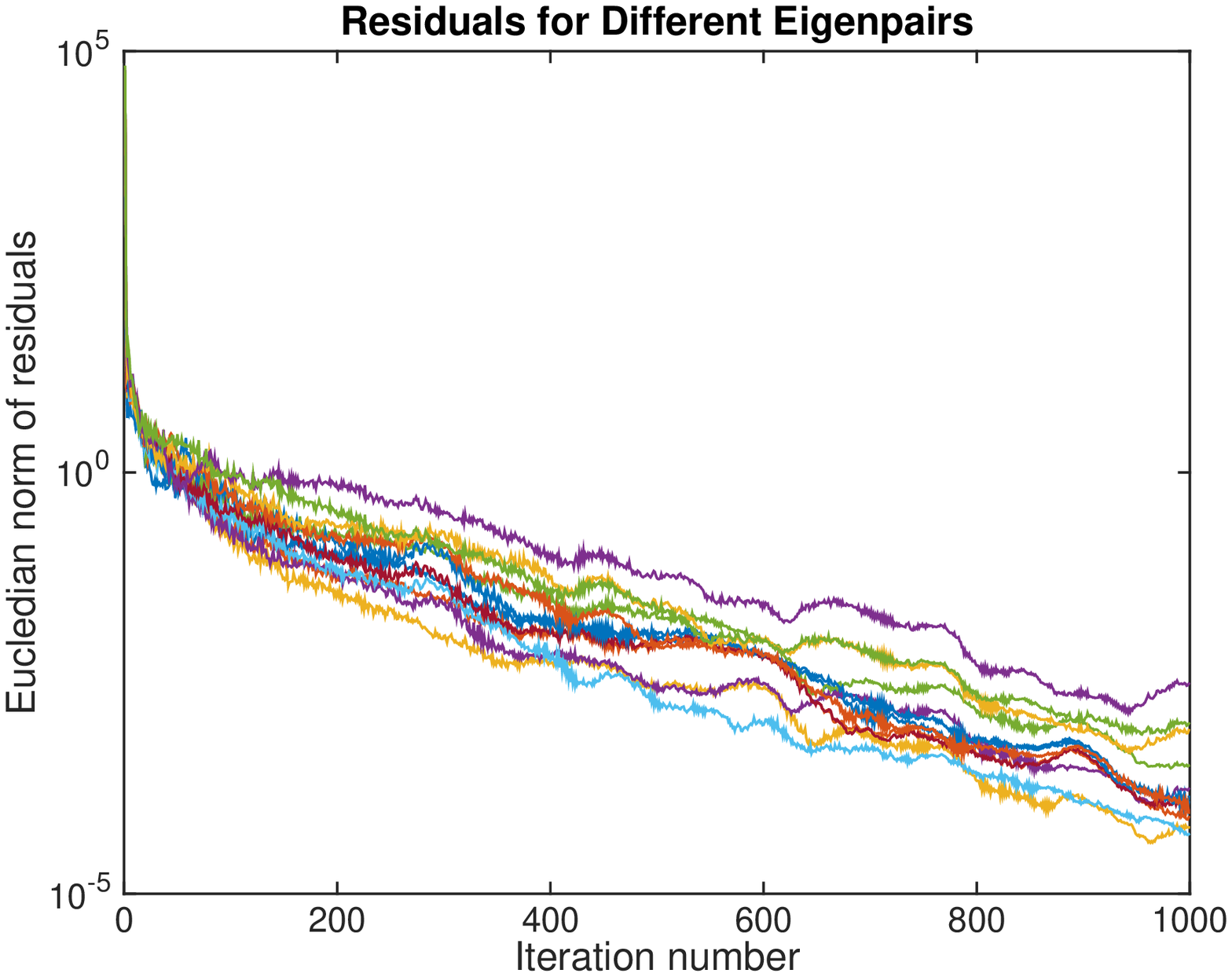}
        \caption{{\footnotesize Residuals for the GMGCG preconditioned LOBPCG.}}
        \label{fig:residual_ilu_lobpcg}
    \end{subfigure}
    \vskip\baselineskip
    \begin{subfigure}[b]{0.45\textwidth}
        \centering
        \includegraphics[width=\textwidth]{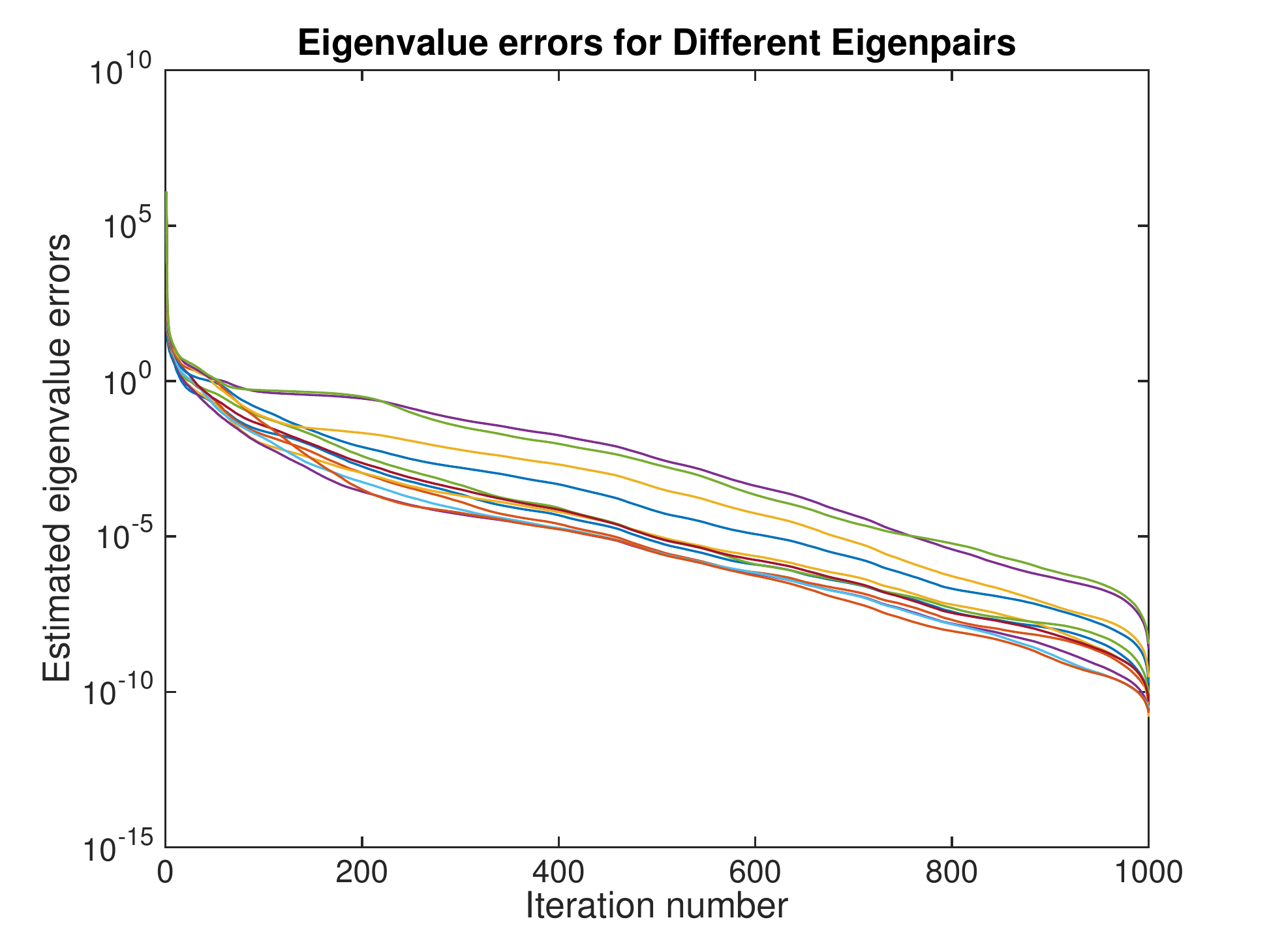}
        \caption{{\footnotesize Eigenvalue errors for the ILU preconditioned LOBPCG;}}
        \label{fig:eigerr_ilu_lobpcg}
    \end{subfigure}
    \quad
    \begin{subfigure}[b]{0.45\textwidth}
        \centering
        \includegraphics[width=\textwidth]{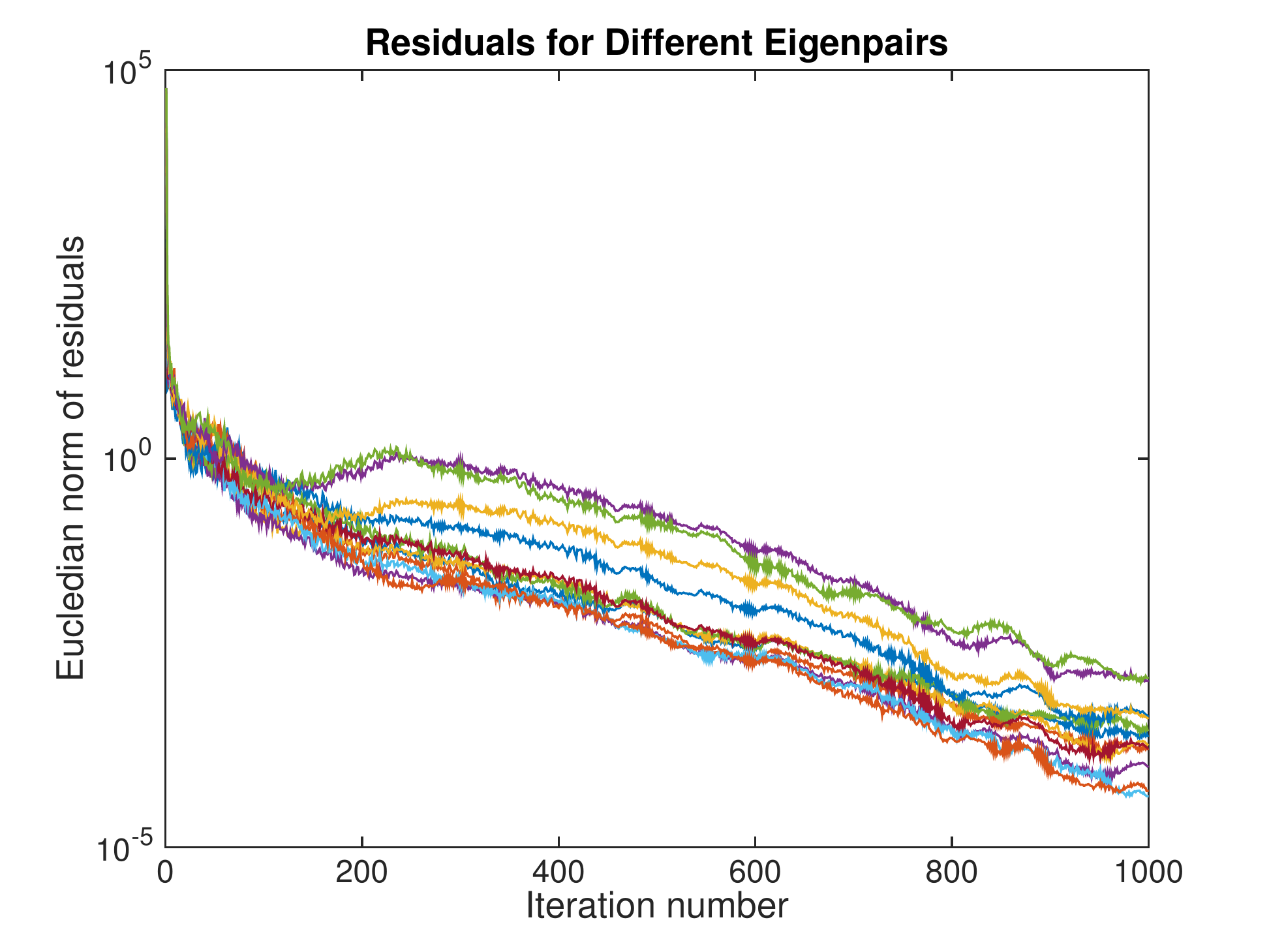}
        \caption{{\footnotesize Residuals for the ILU preconditioned LOBPCG.}}
        \label{fig:residual_ilu_lobpcg}
    \end{subfigure}
    \caption{\small Eigenvalue errors and residuals for the first 12 eigenpairs of the eigenvalue problems for SPE 10 case. Top row: Gamblet preconditioned LOBPCG. Middle row: geometric multigrid preconditioned LOBPCG. Bottom row: ILU preconditioned LOBPCG (using Matlab command \emph{ichol(A,struct('michol','on')))}.}
    \label{fig:lobpcg}
\end{figure}

\paragraph{Combination of Multilevel Correction with LOBPCG}
\revise{We noticed that a ``good" initial value is important for the convergence of the LOBPCG method.   }
Therefore, we propose to combine the multilevel correction scheme and LOBPCG to derive a hybrid method. In this combination, the gamblet based multilevel correction scheme is used to compute, to a high accuracy, an initial approximation for the eigenpairs for the gamblet preconditioned LOBPCG scheme. We use this combined method to solve the so-called Anderson Localization eigenvalue problem in the following subsection. Since LOBPCG is based on the so-called Ky Fan trace minimization principle, at each step the sum of the eigenvalues are minimized \cite{Kressner:2014}. Therefore the convergence rate of different eigenvalues will be balanced.

\subsubsection{Anderson localization}
\label{sec:anderson}
\def\H{\mathcal{H}}

Consider the linear Schr\"{o}dinger operator $\H:=-\Delta + V(x)$ with disorder potential $V(x)$ (as presented  in \cite{Altmann:2018}) whose Anderson localization  \cite{Anderson:1958} properties are analyzed in  \cite{Arnold:2016} and in \cite{Altmann:2018} (see \cite{Billy:2008} and references therein for the ubiquity and importance of localization in wave physics).

\def\ve{\varepsilon}
Let $\Omega:=[-1,1]^2$ be the domain of the operator. By \cite{Altmann:2018}, $V(x)$ is a disorder potential that vary randomly between two values $\ds\beta \geq \frac{1}{\ve^2} \gg \alpha$ on a small scale $\ve$. In the numerical experiment, we choose $\ve = 0.01$, $\beta = 10^4$, and $\alpha = 1$ (the eigenvalue problem becomes more difficult as  $\ve$ becomes smaller). See Figure \ref{fig:historya1c5} for results using the Gamblet based multilevel correction method, Gamblet preconditioned LOBPCG, and the hybrid method.

\begin{figure}[H]
\centering
\includegraphics[width = 0.3\textwidth]{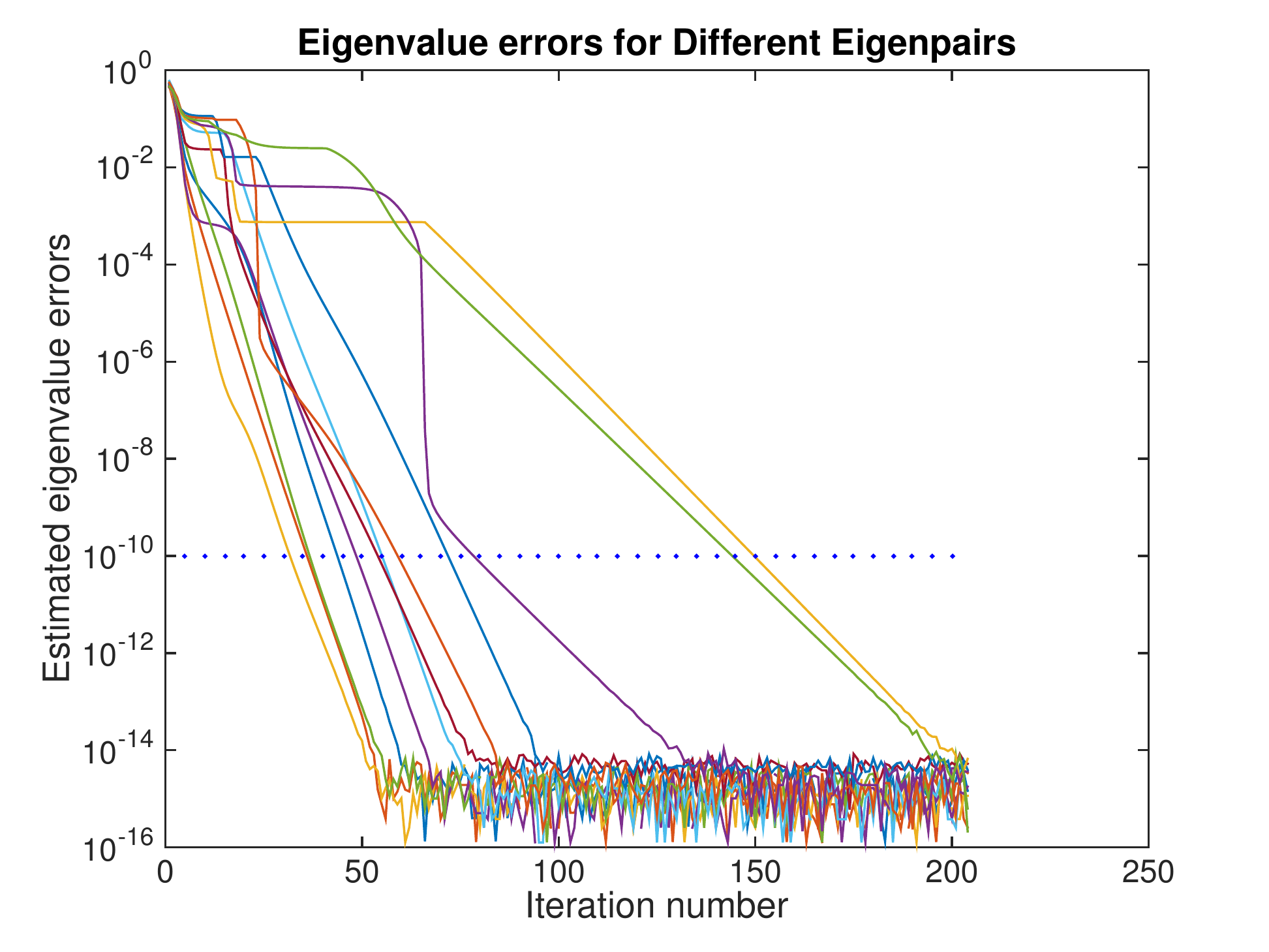}
\includegraphics[width = 0.3\textwidth]{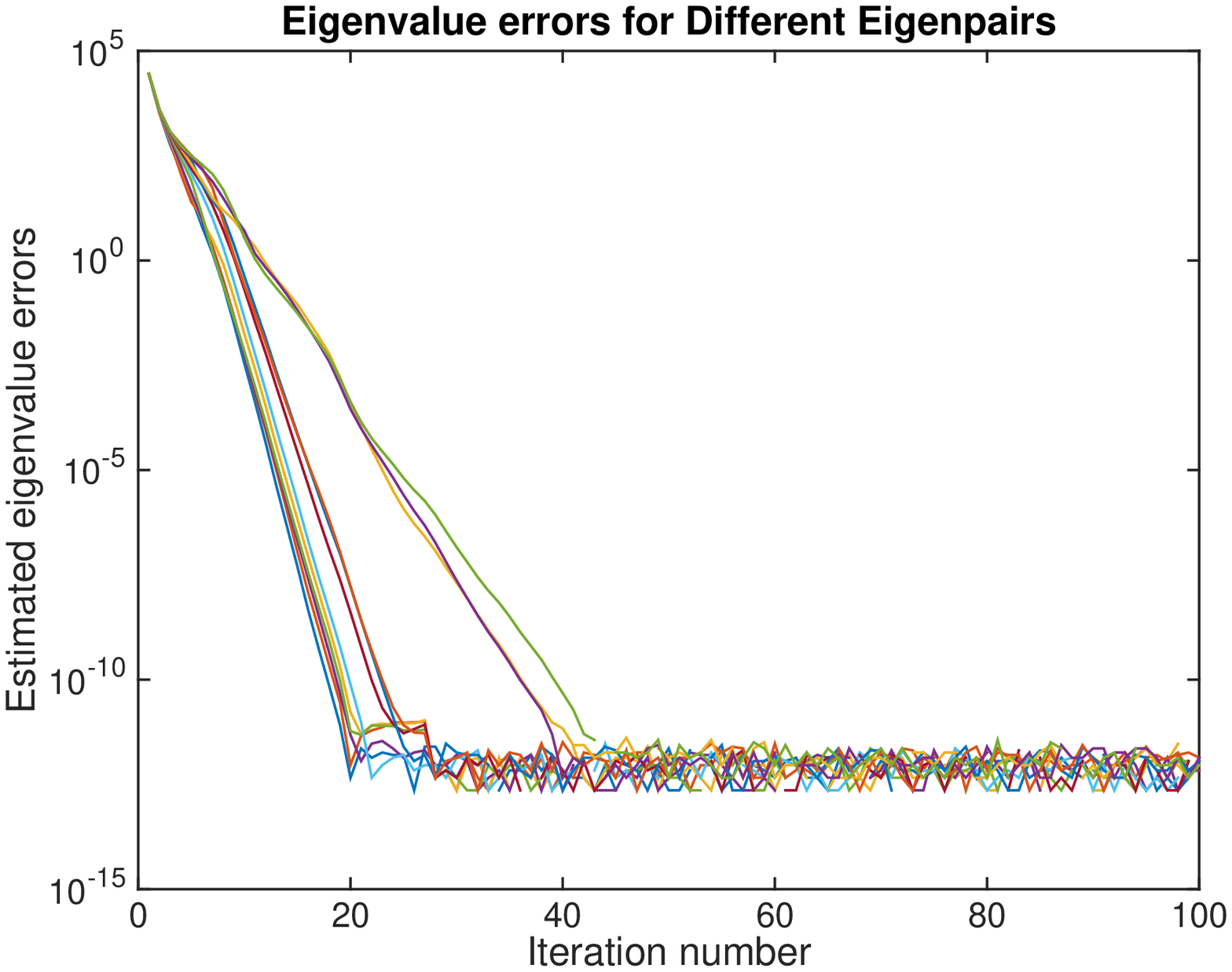}
\includegraphics[width = 0.3\textwidth]{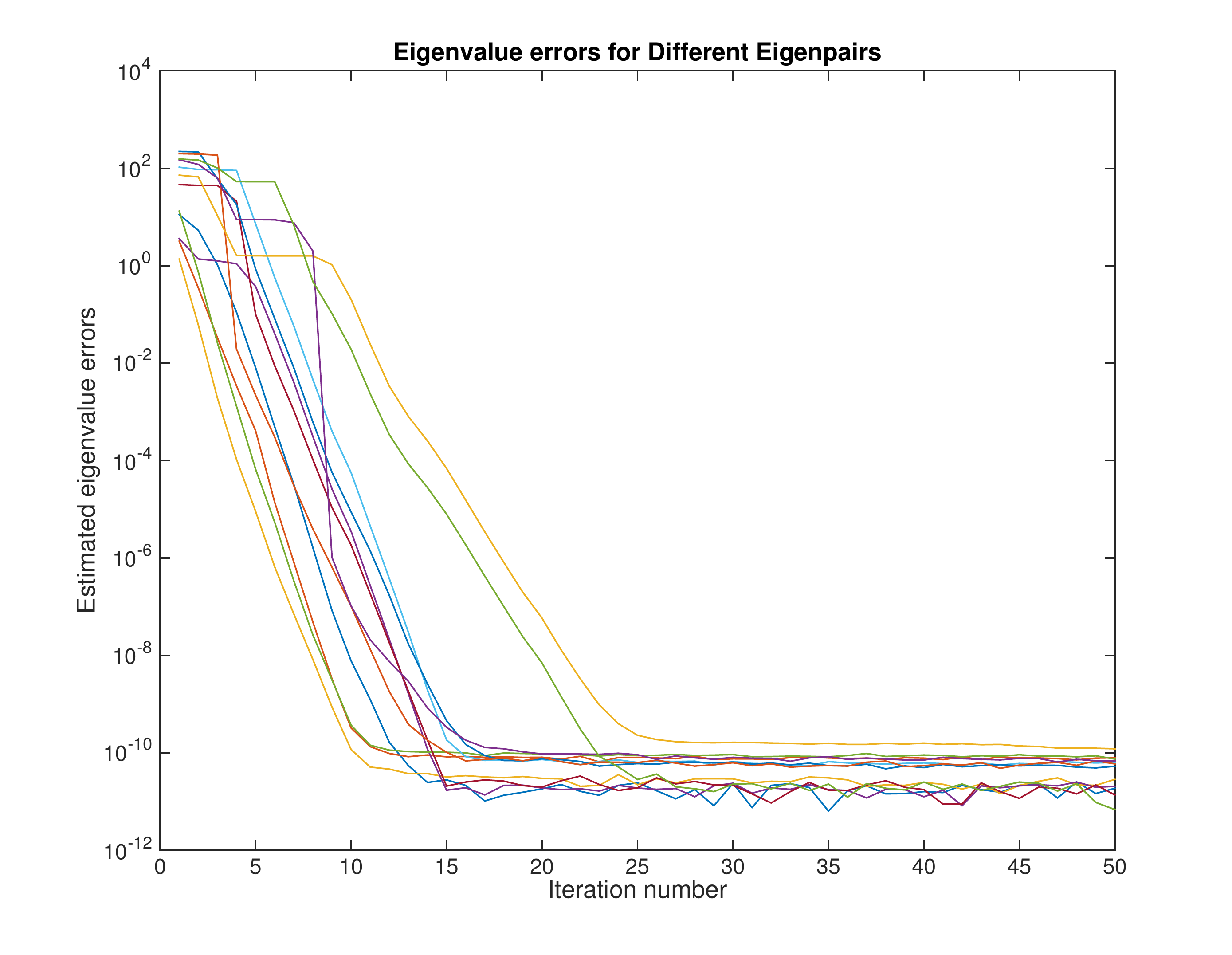}
\caption{Convergence history for first 12 eigenvalues: Left, using the gamblet based multilevel correction method; Middle: using the gamblet preconditioned LOBPCG method; Right, using the hybrid method, namely, generating the initial approximation by the gamblet based multilevel method, then preforming the gamblet preconditioned LOBPCG method until convergence. The iteration number corresponds to the number of correction steps, namely, the outer iteration number. The first a few iterations are on the coarse level $k = 3, \dots, q-1$, and the following iterations are on the finest level $k = q$.  }
\label{fig:historya1c5}
\end{figure}

\section*{Acknowledgement}
HX was partially supported by Science Challenge Project (No. TZ2016002),
National Natural Science Foundations of China (NSFC 11771434, 91330202),
the National Center for Mathematics and Interdisciplinary Science, CAS.
LZ was partially supported by National Natural Science Foundations of China (NSFC 11871339, 11861131004, 11571314). HO gratefully acknowledge support from  the Air Force Office of Scientific Research and the DARPA EQUiPS Program under award number FA9550-16-1-0054 (Computational Information Games) and  the Air Force Office of Scientific Research under award number
FA9550-18-1-0271 (Games for Computation and Learning). We thank Florian Schaefer for stimulating discussions. We thank two anonymous reviewers whose comments have greatly improved this manuscript.

\bibliographystyle{plain}
\bibliography{nh}

\begin{thebibliography}{10}

\bibitem{Adams:2003}
R.~A. Adams and J.~Fournier.
\newblock {\em Sobolev spaces}, volume 140 of {\em Pure and Applied Mathematics
  (Amsterdam)}.
\newblock Elsevier/Academic Press, Amsterdam, second edition, 2003.

\bibitem{Altmann:2018}
R.~Altmann, P.~Henning, and D.~Peterseim.
\newblock Quantitative anderson localization of schr\"{o}dinger eigenstates
  under disorder potentials.
\newblock {\em arXiv:1803.09950}, 2018.

\bibitem{Anderson:1958}
P.~W. Anderson.
\newblock Absence of diffusion in certain random lattices.
\newblock {\em Phys. Rev.}, 109:1492--1505, 1958.

\bibitem{Arnold:2016}
D.~N. Arnold, G.~David, D.~Jerison, S.~Mayboroda, and M.~Filoche.
\newblock Effective confining potential of quantum states in disordered media.
\newblock {\em PRL}, 116:056602, 2016.

\bibitem{BabuskaOsborn_1989}
I.~Babu\v{s}ka and J.~Osborn.
\newblock Finite element-galerkin approximation of the eigenvalues and
  eigenvectors of selfadjoint problems.
\newblock {\em Math. Comp.}, 52:275--297, 1989.

\bibitem{BabuskaOsborn_Book}
I.~Babu\v{s}ka and J.~Osborn.
\newblock Eigenvalue problems.
\newblock In P.~G. Lions and Ciarlet P.G., editors, {\em Handbook of Numerical
  Analysis, Vol. II， Finite Element Methods (Part 1)}, chapter Eigenvalue
  Problems, pages 641--787. North-Holland, Amsterdam, 1991.

\bibitem{Bai:2000}
Z.~Bai, J.~Demmel, J.~Dongarra, A.~Ruhe, and H.~van~der Vorst, editors.
\newblock {\em Templates for the Solution of Agebraic Eigenvalue Problems: A
  Practical Guide}.
\newblock Society for Industrial and Applied Math., Philadelphia, 2000.

\bibitem{Bai:2012}
Z.~Bai and R.~C. Li.
\newblock Minimization principles for the linear response eigenvalue problem i:
  theory.
\newblock {\em SIAM J. Matrix Anal. Appl.}, 33(4):1075–1100, 2012.

\bibitem{Billy:2008}
J.~Billy, V.~Josse, Z.~Zuo, A.~Bernard, B.~Hambrecht, P.~Lugan, D.Cl\'{e}ment,
  L.~Sanchez-Palencia, P.~Bouyer, and A.~Aspect.
\newblock Direct observation of anderson localization of matter waves in a
  controlled disorder.
\newblock {\em Nature}, 453:891–894, 2008.

\bibitem{Bramble:1996}
J.~Bramble, J.~Pasciak, and A.~Knyazev.
\newblock A subspace preconditioning algorithm for eigenvector/eigenvalue
  computation.
\newblock {\em Advances in Computational Mathematics}, 6(1):159–189, 1996.

\bibitem{Brandt:1973}
A.~Brandt.
\newblock Multi-level adaptive technique ({MLAT}) for fast numerical solutions
  to boundary value problems.
\newblock In {\em Proc. 3rd Int’l Conf. Numerical Methods in Fluid
  Mechanics}, 1973.
\newblock Lecture Notes in Physics 18.

\bibitem{brenner2014c0}
Susanne~C Brenner, Peter Monk, and Jiguang Sun.
\newblock C0 ipg method for biharmonic eigenvalue problems.
\newblock In {\em Academy of Mathematics and Systems Science, CAS Colloquia \&
  Seminars}, 2014.

\bibitem{Beylkin:1995}
M.~E. Brewster and G.~Beylkin.
\newblock A multiresolution strategy for numerical homogenization.
\newblock {\em Appl. Comput. Harmon. Anal.}, 2(4):327--349, 1995.

\bibitem{CaoCui}
L.~Cao and J.~Cui.
\newblock Asymptotic expansions and numerical algorithms of eigenvalues and
  eigenfunctions of the dirichlet problem for second order elliptic equations
  in perforated domains.
\newblock {\em Numer. Math.}, 96:525--581, 2004.

\bibitem{Chatelin}
F.~Chatelin.
\newblock {\em Spectral Approximation of Linear Operators}.
\newblock Academic Press Inc, New York, 1983.

\bibitem{ChenXieXu}
H.~Chen, H.~Xie, and F.~Xu.
\newblock A full multigrid method for eigenvalue problems.
\newblock {\em J. Comput. Phys}, 322:747--759, 2016.

\bibitem{DorobantuEngquist1998}
M.~Dorobantu and B.~Engquist.
\newblock Wavelet-based numerical homogenization.
\newblock {\em SIAM J. Numer. Anal.}, 35(2):540--559 (electronic), 1998.

\bibitem{Duersch:2018}
J.~A. Duersch, M.~Shao, C.~Yang, and M.~Gu.
\newblock A robust and efficient implementation of lobpcg.
\newblock {\em SIAM J. Sci. Comput.}, 40(5):C655--C676, 2018.

\bibitem{Dyakonov:1980}
E.~G. D'yakonov and M.~Yu. Orekhov.
\newblock Minimization of the computational labor in determining the first
  eigenvalues of differential operators.
\newblock {\em Math. Notes}, 27:382–391, 1980.

\bibitem{Hackbusch:1978}
W.~Hackbusch.
\newblock A fast iterative method for solving {P}oisson's equation in a general
  region.
\newblock In {\em Numerical treatment of differential equations ({P}roc.
  {C}onf., {M}ath. {F}orschungsinst., {O}berwolfach, 1976)}, pages 51--62.
  Lecture Notes in Math., Vol. 631. Springer, Berlin, 1978.

\bibitem{Hou:2018}
T.~Y. Hou, D.~Huang, K.~C. Lam, and Z.~Zhang.
\newblock A fast hierarchically preconditioned eigensolver based on
  multiresolution matrix decomposition.
\newblock {\em arXiv:1804.03415}, 2018.

\bibitem{hughes1998variational}
T.~J.~R. Hughes, G.~R. Feij{\'o}o, L.~Mazzei, and J.-B. Quincy.
\newblock The variational multiscale method—a paradigm for computational
  mechanics.
\newblock {\em Computer Methods in Applied Mechanics and engineering},
  166(1-2):3--24, 1998.

\bibitem{JiaXieXieXu}
S.~Jia, H.~Xie, M.~Xie, and F.~Xu.
\newblock A full multigrid method for nonlinear eigenvalue problems.
\newblock {\em Sci. China Math.}, 59:2037--2048, 2016.

\bibitem{Knyazev:1998}
A.~Knyazev.
\newblock Preconditioned eigensolvers - an oxymoron?
\newblock {\em Electronic Transactions on Numerical Analysis}, 7:104–123,
  1998.

\bibitem{Kynazev:2001}
A.~Knyazev.
\newblock Toward the optimal preconditioned eigensolver: Locally optimal block
  preconditioned conjugate gradient method.
\newblock {\em SIAM Journal on Scientific Computing}, 23(2):517--541, 2001.

\bibitem{Knyazev:2015}
A.~Knyazev.
\newblock {\em lobpcg.m,
  https://www.mathworks.com/matlabcentral/fileexchange/48-lobpcg-m}, 2015.
\newblock Version 1.5.

\bibitem{Knyazev:2017}
A.~Knyazev.
\newblock Recent implementations, applications, and extensions of the locally
  optimal block preconditioned conjugate gradient method (lobpcg).
\newblock {\em arXiv:1708.08354}, 2017.

\bibitem{Knyazev:2003}
A.~Knyazev and K.~Neymeyr.
\newblock Efficient solution of symmetric eigenvalue problems using multigrid
  preconditioners in the locally optimal block conjugate gradient method.
\newblock {\em Electronic Transactions on Numerical Analysis.}, 15:38--55,
  2003.

\bibitem{kohn1959}
W.~Kohn.
\newblock Analytic properties of {B}loch waves and {W}annier functions.
\newblock {\em Physical Review}, 115(4):809, 1959.

\bibitem{Kornhuber:2018}
R.~Kornhuber, D.~Peterseim, and H.~Yserentant.
\newblock An analysis of a class of variational multiscale methods based on
  subspace decomposition.
\newblock {\em Math. Comp.}, 87:2765--2774, November 2018.
\newblock Submitted for publication.

\bibitem{KornhuberYserentant16}
R.~Kornhuber and H.~Yserentant.
\newblock Numerical homogenization of elliptic multiscale problems by subspace
  decomposition.
\newblock {\em Multiscale Model. Simul.}, 14(3):1017--1036, 2016.

\bibitem{Kressner:2014}
D.~Kressner, M.~Pandur, and M.~Shao.
\newblock An indefinite variant of lobpcg for definite matrix pencils.
\newblock {\em Numerical Algorithms}, 66(4):681--703, 2014.

\bibitem{Kyng:2016}
R.~Kyng and S.~Sachdeva.
\newblock Approximate gaussian elimination for laplacians: Fast, sparse, and
  simple.
\newblock {\em FOCS}, 2016.

\bibitem{LinXie_2011}
Q.~Lin and H.~Xie.
\newblock An observation on aubin-nitsche lemma and its applications.
\newblock {\em Mathematics in Practice and Theory}, 41(17):247--258, 2011.

\bibitem{LinXie_2012}
Q.~Lin and H.~Xie.
\newblock A multilevel correction type of adaptive finite element method for
  steklov eigenvalue problems.
\newblock In {\em Proceedings of the International Conference Applications of
  Mathematics}, pages 134--143, 2012.

\bibitem{LinXie_MultiLevel}
Q.~Lin and H.~Xie.
\newblock A multi-level correction scheme for eigenvalue problems.
\newblock {\em Math. Comp.}, 84:71--88, 2015.

\bibitem{Malqvist2014a}
A.~M{\aa}lqvist and D.~Peterseim.
\newblock Computation of eigenvalues by numerical upscaling.
\newblock {\em Numer. Math.}, 130(2):337--361, 2014.

\bibitem{MaPe:2012}
A.~M{\aa}lqvist and D.~Peterseim.
\newblock Localization of elliptic multiscale problems.
\newblock {\em Mathematics of Computation}, 83(290):2583--2603, 2014.

\bibitem{micchelli1977survey}
C.~A. Micchelli and T.~J. Rivlin.
\newblock A survey of optimal recovery.
\newblock In {\em Optimal Estimation in Approximation Theory}, pages 1--54.
  Springer, 1977.

\bibitem{Olshanskii:2014}
M.~Olshanskii and E.~Tyrtyshnikov.
\newblock {\em Iterative Methods for Linear Systems - Theory and Applications.}
\newblock SIAM, 2014.

\bibitem{OwhadiMultigrid:2017}
H.~Owhadi.
\newblock Multigrid with rough coefficients and multiresolution operator
  decomposition from hierarchical information games.
\newblock {\em SIAM Rev.}, 59(1):99--149, March 2017.

\bibitem{OwhadiScovel:2017}
H.~Owhadi and C.~Scovel.
\newblock Universal scalable robust solvers from computational information
  games and fast eigenspace adapted multiresolution analysis.
\newblock {\em arXiv:1703.10761}, 2017.

\bibitem{OwhScobook2018}
H.~Owhadi and C.~Scovel.
\newblock {\em Operator adapted wavelets, fast solvers, and numerical
  homogenization from a game theoretic approach to numerical approximation and
  algorithm design}.
\newblock Cambridge University Press, 2019.
\newblock Cambridge Monographs on Applied and Computational Mathematics.

\bibitem{Owhadi2017a}
H.~Owhadi and L.~Zhang.
\newblock Gamblets for opening the complexity-bottleneck of implicit schemes
  for hyperbolic and parabolic odes/pdes with rough coefficients.
\newblock {\em J. Comput. Phys.}, 347:99--128, 2017.

\bibitem{OwhadiZhangBerlyand:2014}
H.~Owhadi, L.~Zhang, and L.~Berlyand.
\newblock Polyharmonic homogenization, rough polyharmonic splines and sparse
  super-localization.
\newblock {\em ESAIM Math. Model. Numer. Anal.}, 48(2):517--552, 2014.

\bibitem{Rakhuba:2016}
M.~Rakhuba and I.~Oseledets.
\newblock Calculating vibrational spectra of molecules using tensor train
  decomposition.
\newblock {\em J. Chem. Phys.}, 145:124101, 2016.

\bibitem{Saad:2011}
Y.~Saad.
\newblock {\em Numerical Methods For Large Eigenvalue Problems}.
\newblock SIAM, 2011.

\bibitem{SchaeferSullivanOwhadi17}
F.~Sch\"{a}fer, T.~J. Sullivan, and H.~Owhadi.
\newblock Compression, inversion, and approximate {PCA} of dense kernel
  matrices at near-linear computational complexity.
\newblock {\em arXiv:1706.02205}, 2017.

\bibitem{Sorensen:1997}
D.~Sorensen.
\newblock {\em Implicitly Restarted Arnoldi/Lanczos Methods for Large Scale
  Eigenvalue Calculations}.
\newblock Springer Netherlands, 1997.

\bibitem{Strang:1973}
G.~Strang and G.~Fix.
\newblock {\em An analysis of the finite element method}.
\newblock Prentice‐Hall, 1973.

\bibitem{Szyld:2016}
D.~B. Szyld and F.~Xue.
\newblock Preconditioned eigensolvers for large-scale nonlinear hermitian
  eigenproblems with variational characterizations. i. external eigenvalues.
\newblock {\em Mathematics of Computation}, 85:2887--2918, 2016.

\bibitem{Vecharynski:2015}
E.~Vecharynski, C.~Yang, and J.~E.Pask.
\newblock A projected preconditioned conjugate gradient algorithm for computing
  many extreme eigenpairs of a hermitian matrix.
\newblock {\em Journal of Computational Physics}, 290:73--89, 2015.

\bibitem{wan2000}
W.~L. Wan, Tony~F. Chan, and Barry Smith.
\newblock An energy-minimizing interpolation for robust multigrid methods.
\newblock {\em SIAM J. Sci. Comput.}, 21(4):1632--1649, 1999/00.

\bibitem{wannier1962dynamics}
G.~H. Wannier.
\newblock Dynamics of band electrons in electric and magnetic fields.
\newblock {\em Reviews of Modern Physics}, 34(4):645, 1962.

\bibitem{Xie_JCP}
H.~Xie.
\newblock A multigrid method for eigenvalue problem.
\newblock {\em J. Comput. Phys.}, 274:550--561, 2014.

\bibitem{Xie_IMA}
H.~Xie.
\newblock A type of multilevel method for the steklov eigenvalue problem.
\newblock {\em IMA J. Numer. Anal.}, 34:592--608, 2014.

\bibitem{Xu:2001}
J.~Xu and A.~Zhou.
\newblock A two-grid discretization scheme for eigenvalue problems.
\newblock {\em Mathematics of Computation}, 70(233):17--25, 2001.

\bibitem{yoo2018noising}
Gene~Ryan Yoo and Houman Owhadi.
\newblock De-noising by thresholding operator adapted wavelets.
\newblock {\em Statistics and Computing}, 2019.

\bibitem{ZhangCaoWang}
L.~Zhang, L.~Cao, and X.~Wang.
\newblock Multiscale finite element algorithm of the eigenvalue problems for
  the elastic equations in composite materials.
\newblock {\em Comput. Methods Appl. Mech. Engrg.}, 198:2539--2554, 2009.

\bibitem{zhang2017multi}
Shuo Zhang, Yingxia Xi, and Xia Ji.
\newblock A multi-level mixed element method for the eigenvalue problem of
  biharmonic equation.
\newblock {\em Journal of Scientific Computing}, pages 1--30, 2017.

\end{thebibliography}

\end{document}